\documentclass[reqno]{amsart}

\usepackage[a4paper]{geometry}
\usepackage[T2A]{fontenc}
\usepackage[utf8]{inputenc}
\usepackage[ukrainian,russian,english]{babel}
\usepackage{comment}

\usepackage{amssymb}

\usepackage[mathscr]{eucal}
\usepackage{tikz}
\usetikzlibrary{positioning,calc,decorations.markings}
\usetikzlibrary{patterns}
\usetikzlibrary{calc}
\usepackage{mathtools}
\usepackage{enumerate}
\usepackage{accents}
\usepackage{dsfont}
\usepackage{multicol}

\usepackage{color}
\usepackage[colorlinks=true]{hyperref}
\hypersetup{urlcolor=blue,citecolor=red,linkcolor=blue}
\usepackage[initials,nobysame]{amsrefs}
\definecolor{darkgreen}{rgb}{0.13, 0.55, 0.13}

\numberwithin{equation}{section}
\theoremstyle{plain}
\newtheorem{theorem}{Theorem}[section]
\newtheorem{proposition}[theorem]{Proposition}
\newtheorem{corollary}[theorem]{Corollary}

\theoremstyle{definition}
\newtheorem*{rh-pb*}{Basic RH problem}
\newtheorem*{rh-I*}{RH problem normalized as $I$ at $\infty$}
\newtheorem*{sol-rh-pb*}{Soliton RH problem}
\newtheorem*{data*}{Data of this RH problem associated with $\BS{u_0(x)}$}
\theoremstyle{remark}
\newtheorem{remark}[theorem]{Remark}
\newtheorem*{notations*}{Notations}
\providecommand{\BS}[1]{\boldsymbol{#1}}  

\newcommand{\dd}{\mathrm{d}}
\newcommand{\eul}{\mathrm{e}}
\newcommand{\ii}{\mathrm{i}}
\newlength{\dhatheight}
\newcommand{\doublehat}[1]{%
    \settoheight{\dhatheight}{\ensuremath{\hat{#1}}}%
    \addtolength{\dhatheight}{-0.3ex}%
    \hat{\vphantom{\rule{1pt}{\dhatheight}}%
    \smash{\hat{#1}}}}

\renewcommand{\Im}{\operatorname{Im}}

\newcommand{\ord}{\mathrm{O}}
\DeclareMathOperator{\Res}{Res}

\newif\ifshort
\shorttrue

\title{The Short Wave equation on the half-line by the Unified Transform Method}

\author[Iryna Karpenko]{Iryna Karpenko}
\address{IK: Faculty of Mathematics\\ University of Vienna\\
Oskar-Morgenstern-Platz 1\\ 1090 Wien\\ Austria\\ and B. Verkin Institute for Low Temperature Physics and Engineering\\ 47, Nauky ave\\ 61103 Kharkiv\\ Ukraine}
\email{\href{mailto:iryna.karpenko@univie.ac.at}{iryna.karpenko@univie.ac.at}}

\author[Dmitry Shepelsky]{Dmitry Shepelsky}
\address{DS: B.Verkin Institute for Low Temperature Physics and Engineering, Kharkiv, Ukraine\\ and V.N.Karazin Kharkiv National University, Ukraine}
\email{\href{mailto:shepelsky@yahoo.com }{shepelsky@yahoo.com}}

\author[Ilona Tylevna]{Ilona Tylevna}
\address{IT: Faculty of Mathematics\\ University of Vienna\\
Oskar-Morgenstern-Platz 1\\ 1090 Wien\\ Austria}
\email{\href{mailto:ilonatylevna@gmail.com}{ilonatylevna@gmail.com}}

\begin{document}

\begin{abstract}
We study the initial-boundary value problem for the Short Wave equation on the half-line $x\ge 0$. A distinctive feature of this problem is that the boundary $x=0$ cannot be characterized \emph{a priori} as an inflow or outflow boundary: its character is determined dynamically by the sign of the unknown trace $u(0,t)$. This leads to different analyticity properties of the associated eigenfunctions and, consequently, to different spectral formulations in the regimes $u(0,t)\le0$ and $u(0,t)\ge0$.

Using the Unified Transform Method (aka the Fokas method), we formulate the solution in terms of matrix Riemann--Hilbert problems. We construct the associated spectral functions, derive the global relations, and show how the solution is reconstructed from the corresponding Riemann--Hilbert problem. In the case $u(0,t)\le0$, the solution is determined by the initial data alone (assuming an appropriate decay as $x\to \infty$), whereas for $u(0,t)\ge0$, compatible boundary data are also required for the construction.
\end{abstract}

\maketitle

\section{Introduction}

The Short Wave (SW) equation
\begin{equation}\label{SW_1}
    u_{txx}-2\omega u_x+2u_xu_{xx}+uu_{xxx}=0
\end{equation}
arises as a short-wave limit of the Camassa--Holm (CH) equation
\cite{BoutetDeMonvelShepelskyZielinski2011}
and models short capillary--gravity waves
\cites{Borzi2005,FaquirMannaNeveu2007}. It is also related to the
Dym hierarchy~\cites{Alber1995,Alber1999,Kruskal1975} and reduces,
in the limiting case when the linear dispersion parameter $\omega$ vanishes,
to the Hunter--Saxton equation
\cites{HunterSaxton1991,Lenells2008}. 
In what follows we will deal with the case $\omega>0$, where, due to simple rescaling, one can assume without loss of generality that $\omega=1$.

The SW equation is completely
integrable and admits the Lax pair representation
\begin{equation}\label{Lax}
\begin{aligned}
\psi_{xx} &= -k^2 (m+1)\psi, \\
\psi_t &= \Big(-\frac{1}{2k^2}-u\Big)\psi_x+\frac12 u_x\psi,
\end{aligned}
\end{equation}
where
$m=-u_{xx}$.
Assuming that $m+1\ge 0$, 
the SW equation admits the  ``conservation-law'' formulation
\begin{equation}\label{SW}
\left(\sqrt{m+1}\right)_t=-\left(u\sqrt{m+1}\right)_x, 
\qquad m=-u_{xx},
\end{equation}
which can be interpreted as follows:  the quantity $\sqrt{m+1}$ is transported by the velocity field $u(x,t)$.

The Cauchy problem for the SW equation on the line has been studied by the 
Riemann--Hilbert (RH) problem approach in~\cite{trogdon2015riemann}, which allowed studying smooth solutions, their long-time asymptotics, and cuspon-type solutions~\cite{BoutetDeMonvelShepelskyZielinski2011}.

The analysis becomes considerably more involved when the SW equation is
posed on domains with boundaries, such as the quarter-plane or a finite
time half-strip. A general framework for treating initial-boundary value (IBV)
problems for integrable equations was introduced by Fokas through the
unified transform method (UTM), also known as the Fokas method
\cite{F02}. Based on the simultaneous spectral analysis of both equations
of the Lax pair, the UTM characterizes the solution in terms of a matrix
RH problem whose spectral data are generated by the initial and boundary
values. Since its introduction, the method has been successfully applied
to a broad class of integrable equations
\cites{fokas2008unified,BS08,FIS,BFS06,BoutetShepelsky2009,KarpenkoShepelsky2026,Karpenko2026SineGordon,SKPB24,LF2009,L12}.

For the SW equation, however, the boundary behavior has a special feature
which is already visible from the conservation law \eqref{SW}. Introducing
\[
w=1-u_{xx},
\]
equation \eqref{SW_1} can be rewritten as
\[
w_t+u w_x=-2u_x w.
\]
Thus, the characteristics satisfy
\[
\frac{\dd x}{\dd t}=u(x(t),t).
\]
Assuming an appropriate decay at infinity,
the transport velocity $u(x,t)$ can be reconstructed nonlocally from $w$:
\[
u(x,t)=-\int_x^\infty (s-x)(w(s,t)-1)\,\dd s.
\]
It follows that the sign of $u(0,t)$ is not specified \emph{a priori}. Consequently, the
boundary $x=0$ of the domain $x\ge 0, t\ge 0$ is not intrinsically an inflow or an outflow boundary.
Instead, its character is determined dynamically by the solution itself:
if $u(0,t)>0$, then the characteristics enter the domain through the boundary,
whereas if $u(0,t)<0$, they leave it.

This distinction has a direct spectral counterpart. As in the case of the
CH equation, the analytical properties of the eigenfunctions
entering the RH formulation depend significantly on the sign of the boundary
value $u(0,t)$. Accordingly, in the present paper we consider separately
the cases $u(0,t)\le0$ and $u(0,t)\ge0$. If the boundary value changes sign
during the evolution, the analysis can be continued by partitioning the
domain into successive half-strips on which the sign of $u(0,t)$
remains fixed. For each such half-strip, the RH problem is formulated with
the solution reconstructed at the end of the preceding time interval serving as
the new initial data.

In the present paper, we study the IBV problem for
\eqref{SW} on the half-line $x\ge0$. The initial data are prescribed by
\begin{equation}\label{ic}
     u(x,0) =  u_0(x), \quad x \geq 0
\end{equation}
where $u_0(x)\to 0$ as $x\to\infty$. Consequently, 
\begin{equation*}
     m(x,0) = m_0(x), \quad x \geq 0,
\end{equation*}
where $ m_0(x):=- u_{0xx}(x)$. The boundary traces, when involved in the formulation, are denoted by
\begin{equation}\label{boundary}
     u(0,t) = v_0(t), \quad u_x(0,t) = v_1(t), \quad 
    u_{xx}(0,t) =  v_2(t), \qquad
    0\leq t \leq T<\infty,
\end{equation}

We assume that
\[
m_0(x)+1>0,\qquad x\ge0,
\]
and, whenever $v_2$ is prescribed,
\[
1-v_2(t)>0,\qquad 0\le t\le T,
\]
so that $m(0,t)+1>0$. We also impose the corner compatibility condition
$\partial_x^j  u_0(0)=v_j(0)$, $j=0,1,2$. Observe that the conditions $m(x,0)+1> 0$ and $1-v_2(t)>0$ (the latter whenever $v_2$ is prescribed on the inflow boundary) ensure that $ m(x,t)+1>0$ for all $t$ as long as the solution exists (for the original CH equation, see~\cite{CE98}).  The positivity condition $m+1>0$ is essential for the spectral analysis:
it provides the analytic properties of the eigenfunctions needed to construct the associated RH problems.

The main contributions of the paper are threefold. First, motivated by the characteristic structure of the SW equation, we identify the sign of the boundary trace $u(0,t)$ as the quantity that distinguishes the two relevant boundary regimes. Second, we show that these two regimes lead to different analyticity properties of the eigenfunctions and therefore to two distinct RH formulations, corresponding to $u(0,t)\le0$ and $u(0,t)\ge0$. Third, we construct the associated spectral functions, derive the global relations, and obtain RH-based results, including reconstruction, existence, and uniqueness. In particular, in the regime $u(0,t)\le0$, the solution is determined by the initial data alone, whereas in the regime $u(0,t)\ge0$, compatible boundary data are required.








The paper is organized as follows.

In Section \ref{sec:2}, we introduce two transformations of the Lax pair,
adapted respectively to the behavior of the spectral parameter near
$k=\infty$ and near $k=0$. We define the associated eigenfunctions and
describe their analytic properties. Particular
attention is paid to the dependence of these properties on the sign of the
boundary value $u(0,t)$. We also introduce the spectral functions associated
with the initial and boundary values.

In Section \ref{sec:3}, the direct spectral problems are studied. The
spectral mappings generated by the initial profile and by the boundary
traces are described, together with the corresponding properties of the
spectral functions.

In Section \ref{sec:4}, we derive the global relations. These relations express the compatibility of the initial and boundary values in spectral
terms.

In Section \ref{sec:5}, the inverse spectral mappings are formulated in
terms of RH problems.

In Section \ref{sec:SWinyt}, we first discuss 
the relationship between solutions of the SW equation in the original variables $(x,t)$
and solutions of the SW equation in variables $(y,t)$ suitable for the RH formalism.
Then we discuss 
how a function solving, locally, the SW equation 
in $(y,t)$ variables appears from the solution of a RH
problem parametrized by $y$ and $t$.

In Section \ref{sec:6}, we formulate the RH problems
corresponding to the full IBV problem. We prove that
for the regime $u(0,t)\le0$, the solution is uniquely determined by the
initial data alone; see Corollary \ref{cor:uniq}. For the regime $u(0,t)\ge0$, we derive the
corresponding RH formulation in terms of both the initial and boundary
data and establish sufficient conditions for the existence and uniqueness
of solutions of the IBV problem; see Theorem
\ref{prop:ex}. In both cases, the
solution of the SW equation is recovered from the solution of the
corresponding RH problem; see
Theorems \ref{Prop:rep} and \ref{Prop:rep_geq}.

\begin{notations*}
Throughout the paper, we denote the standard Pauli matrices by
\[
\sigma_1\coloneqq
\begin{pmatrix}0&1\\1&0\end{pmatrix},
\qquad
\sigma_2\coloneqq
\begin{pmatrix}0&-\ii\\ \ii&0\end{pmatrix},
\qquad
\sigma_3\coloneqq
\begin{pmatrix}1&0\\0&-1\end{pmatrix}.
\]
For a $2\times2$ matrix $A$, we use the notation
\[
e^{\hat\sigma_3}A\coloneqq e^{\sigma_3}Ae^{-\sigma_3}.
\]
We also set
\[
\mathbb C^\pm\coloneqq\{k\in\mathbb C:\ \pm\Im k>0\},
\]
and denote the Schwarz conjugate of a function $f$ by
\[
f^*(k)\coloneqq \overline{f(\bar k)}.
\]
Finally, for an interval $\Omega\subseteq\mathbb R$, we use the weighted Sobolev space
\[
H^{k,j}(\Omega)=\{f(x)\in L^1_{\mathrm{loc}}(\Omega):~ f(x), x^{2j}f(x), f'(x), x^{2j}f'(x),...,f^{(k)}(x), x^{2j}f^{(k)}(x) \in L^2(\Omega)\}.
\]
In particular,when $j=0$, this reduces to the standard Sobolev space,
\[H^{k,0}(\Omega)=H^{k}(\Omega).\]

\end{notations*}

\section{Eigenfunctions and Spectral Functions}\label{sec:2}

Assume that we are given a solution $u(x,t)$ of the SW equation in the domain $(x,t)\in(0,\infty)\times (0,T)$ such that 
\begin{enumerate}[(i)]
    \item $u(\cdot, t)\in  H^{4}(0,\infty)$ and $u_{xx}(\cdot, t)\in  H^{2,1}(0,\infty)$ for all $t\in[0,T]$,
    \item $m(x,0)+1>0$ for $x\ge0$,
    \item the boundary traces satisfy $v_j\in H^{1,0}(0,T)$, $ j=0,1,2$, and $1-v_2(t)>0$ for $0\le t\le T$.
\end{enumerate}

\begin{remark}
    The assumption (i), together with the positivity of $m+1$, implies that $\frac{m_x}{4(m+1)}$, $m\in L^1(0,\infty)$  for all $t\in[0,T]$, which allows us to apply the Neumann series argument below.
\end{remark}

 We analyze the analytic properties, in the complex plane of the spectral parameter, of the associated solutions of the Lax pair equations, referred to as eigenfunctions. This analysis is carried out to derive relations among these eigenfunctions that lead to a factorization RH problem.

\begin{remark}
 The condition $m(x,0)+1>0$ is propagated by the SW flow; hence
$
m(x,t)+1>0
$
for all $t$ for which the solution exists. An analogous property is known for the original CH equation; see~\cite{CE98}.
\end{remark}

\subsection{Eigenfunctions near $k=\infty$}

To control the large $k$ behavior, we introduce 

\begin{equation*}
\tilde{\Phi}_\infty
= G
\begin{pmatrix}
\psi \\ \psi_x
\end{pmatrix},
\end{equation*}
where
\begin{equation*}
\begin{aligned}
G(x,t,k)
&= G_0(k)
\begin{pmatrix}
(m(x,t)+1)^{1/4} & 0\\
0 & (m(x,t)+1)^{-1/4}
\end{pmatrix},\\
G_0(k)
&= \frac{1}{2}
\begin{pmatrix}
1 & -\frac{1}{ik}\\
1 & \frac{1}{ik}
\end{pmatrix}.
\end{aligned}
\end{equation*}

This transforms \eqref{Lax} into

\begin{subequations}\label{Lax-Q-form}
\begin{align}\label{x-eq}
&\tilde{\Phi}_{\infty x} + ik\sqrt{m+1}\sigma_3\tilde{\Phi}_\infty = U_\infty\tilde{\Phi}_\infty,\\
&\tilde{\Phi}_{\infty t} - ik\left(\frac{1}{2k^2}+u\sqrt{m+1}\right)\sigma_3
\tilde{\Phi}_\infty = V_\infty\tilde{\Phi}_\infty,
\end{align}
\end{subequations}
where $U_\infty\equiv U_\infty(x,t,k)$ and $V_\infty\equiv V_\infty(x,t,k)$ are given by
\begin{subequations}\label{Lax-1}
\begin{equation}\label{U-hat}
U_\infty=\frac{m_x}{4(m+1)}\sigma_1
\end{equation}
and 
\begin{equation}\label{hat-V}
\begin{aligned}
 V_\infty&=
\left(\frac{m_t}{4(m+1)}+\frac{u_x}{2}\right)\sigma_1
- \frac{1}{4ik}
\left(
\sqrt{m+1} + \frac{1}{\sqrt{m+1}} -2
\right)\sigma_3+\frac{1}{4k}\left(\sqrt{m+1}-\frac{1}{\sqrt{m+1}}\right)\sigma_2.
\end{aligned}
\end{equation}
\end{subequations}

Define $\mathcal Q$ by 
\begin{subequations}\label{Qp}
\begin{equation}\label{Q}
\mathcal Q(x,t,k):= p(x,t,k)\sigma_3, 
\end{equation}
with
\begin{equation}\label{p}
p(x,t,k):= \ii k\left(\int_0^{x} \left(\sqrt{ m+1}\right)(\xi,t)\dd\xi-\int_0^{t} \left(u\sqrt{ m+1}\right)(0,\tau)\dd\tau-\frac{t}{2 k^2}\right).
\end{equation}
\end{subequations}
Then 
\[
p_x=ik\sqrt{m+1},\qquad 
p_t=- ik\left(\frac{1}{2k^2}+u\sqrt{m+1}\right),\qquad 
p(0,0,k)=0,
\]
and 
equations in \eqref{Lax-Q-form} take the form
\begin{subequations}\label{phi-tilde}
   \begin{align}
&\tilde \Phi_{\infty x}+\mathcal Q_x\tilde \Phi_\infty = U_\infty \tilde \Phi_\infty,\\
&\tilde \Phi_{\infty t} +\mathcal Q_t\tilde \Phi_\infty = V_\infty \tilde \Phi_\infty,
\end{align} 
\end{subequations}
which is appropriate to control the large-$k$ behavior of 
its solutions.

\begin{remark}
    Notice that \eqref{SW} implies the ``conservation law''
\begin{equation} \label{cons_law}
\nu(t)=\nu(0)-\eta(t),  
\end{equation}
where

\[
\nu(t):=\int_0^\infty \left(\sqrt{ m+1}-1\right)(\xi,t)\dd \xi,\qquad 
    \eta(t):=-\int_0^t \left(u \sqrt{ m+1}\right)(0,\tau)\dd \tau.
\]
    
\end{remark}

Introduce the solutions $\tilde\Phi_{\infty j}$, $j=1,2,3$ to \eqref{phi-tilde} through the solutions $\Phi_{\infty j}$ of the integral equations, which are fixed by the associated
initial 
integration point. Namely, $\tilde \Phi_{\infty j }(x,t,k)=\Phi_{\infty j }(x,t,k)\eul^{-p(x,t,k)\sigma_3}$, where $\Phi_{\infty j}$, $j=1,2,3$ satisfy
the integral equation 
\begin{equation*}
\Phi_{\infty}(x,t,k)=I+\int_{(x^*,t^*)}^{(x,t)}
	\eul^{(p(y,\tau,k)-p(x,t,k))\hat\sigma_3}( U_\infty\Phi_\infty \dd y+V_\infty\Phi_\infty \dd \tau)(y,\tau,k),
\end{equation*}
where $(x^*,t^*)$ is taken respectively as $(0,T)$, $(0,0)$ and $(+\infty,t)$, and the integration paths are as in Figure \ref{fig:integration-paths} (due to the compatibility of the Lax pair equations, the integrals don't depend on integration paths). More precisely,
\begin{equation}\label{inteq_inf1}
    \begin{aligned}
\Phi_{\infty 1}(x,t,&k)=I+
\int_{0}^{x}
	\eul^{-\ii k\int_y^x\left(\sqrt{ m+1}\right)(\xi,t)\dd \xi\hat\sigma_3}( U_\infty\Phi_{\infty1})(y,t,k) \dd y   \\
&-\eul^{-\ii k\int_0^x\left(\sqrt{ m+1}\right)(\xi,t)\dd \xi\hat\sigma_3}
\int_{t}^{T}
	\eul^{\left(-\ii k \int_t^\tau \left(u \sqrt{ m+1}\right)(0,s)\dd s-\frac{\ii}{2 k}(\tau-t)\right)\hat\sigma_3}( V_\infty\Phi_{\infty1})(0,\tau,k) \dd \tau,
\end{aligned}
\end{equation}
\begin{equation}\label{inteq_inf2}
    \begin{aligned}
\Phi_{\infty 2}(x,t,&k)=I+
\int_{0}^{x}
	\eul^{-\ii k\int_y^x\left(\sqrt{ m+1}\right)(\xi,t)\dd \xi\hat\sigma_3}( U_\infty\Phi_{\infty2})(y,t,k) \dd y   \\
&+\eul^{-\ii k\int_0^x\left(\sqrt{ m+1}\right)(\xi,t)\dd \xi\hat\sigma_3}
\int_{0}^{t}
	\eul^{\left(-\ii k \int_t^\tau \left(u \sqrt{ m+1}\right)(0,s)\dd s-\frac{\ii}{2 k}(\tau-t)\right)\hat\sigma_3}( V_\infty\Phi_{\infty2})(0,\tau,k) \dd \tau,
\end{aligned}
\end{equation}    

\begin{equation}\label{inteq_inf3}
    \begin{aligned}
\Phi_{\infty 3}(x,t,k)=&I-
\int_{x}^{+\infty}
	\eul^{\ii k\int_x^y\left(\sqrt{ m+1}\right)(\xi,t)\dd \xi\hat\sigma_3}( U_\infty\Phi_{\infty3})(y,t,k) \dd y.
\end{aligned}
\end{equation}

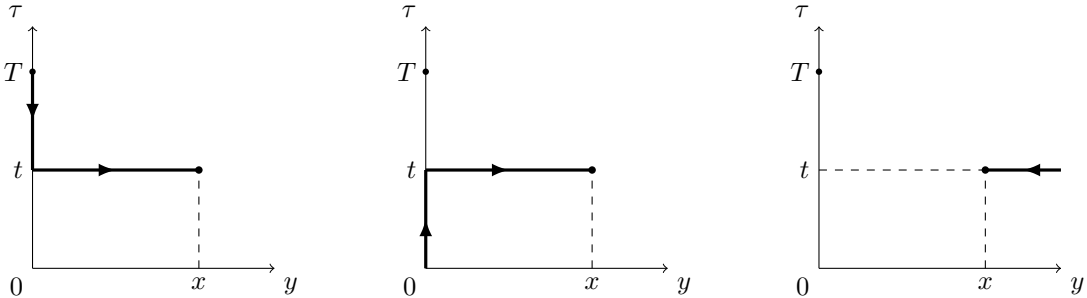
\begin{figure}[ht]
\centering
\begin{tikzpicture}[scale=1]
  \def\x{2.2}  
  \def\t{1.3}  
  \def\T{2.6}  
  \def\W{3.2}  

  \begin{scope}
    \draw[->] (0,0) -- (\W,0) node[below right] {$y$};
    \draw[->] (0,0) -- (0,\W) node[above left] {$\tau$};

    \node[below left] at (0,0) {$0$};
    \node[below]      at (\x,0) {$x$};
    \node[left]       at (0,\t) {$t$};

    \fill (0,\T) circle (1.2pt) node[left] {$T$};

    \draw[dashed] (\x,0) -- (\x,\t);

    \draw[very thick,postaction={decorate},
          decoration={markings, mark=at position 0.5 with {\arrow{latex}}}] 
          (0,\T) -- (0,\t);
    \draw[very thick,postaction={decorate},
          decoration={markings, mark=at position 0.5 with {\arrow{latex}}}] 
          (0,\t) -- (\x,\t);

    \fill (\x,\t) circle (1.4pt);
  \end{scope}

  \begin{scope}[xshift=5.2cm]
    \draw[->] (0,0) -- (\W,0) node[below right] {$y$};
    \draw[->] (0,0) -- (0,\W) node[above left] {$\tau$};

    \node[below left] at (0,0) {$0$};
    \node[below]      at (\x,0) {$x$};
    \node[left]       at (0,\t) {$t$};

    \fill (0,\T) circle (1.2pt) node[left] {$T$};

    \draw[dashed] (\x,0) -- (\x,\t);

    \draw[very thick,postaction={decorate},
          decoration={markings, mark=at position 0.5 with {\arrow{latex}}}] 
          (0,0) -- (0,\t);
    \draw[very thick,postaction={decorate},
          decoration={markings, mark=at position 0.5 with {\arrow{latex}}}] 
          (0,\t) -- (\x,\t);

    \fill (\x,\t) circle (1.4pt);
  \end{scope}

  \begin{scope}[xshift=10.4cm]
    \draw[->] (0,0) -- (\W,0) node[below right] {$y$};
    \draw[->] (0,0) -- (0,\W) node[above left] {$\tau$};

    \node[below left] at (0,0) {$0$};
    \node[below]      at (\x,0) {$x$};
    \node[left]       at (0,\t) {$t$};

    \fill (0,\T) circle (1.2pt) node[left] {$T$};

    \draw[dashed] (\x,0) -- (\x,\t);
    \draw[dashed] (0,\t) -- (\x,\t);

    \draw[very thick,postaction={decorate},
          decoration={markings, mark=at position 0.5 with {\arrow{latex}}}] 
          (\W,\t) -- (\x,\t);

    \fill (\x,\t) circle (1.4pt);
  \end{scope}
\end{tikzpicture}

\caption{Paths of integration for $\Phi_{01}, \Phi_{02},$ and $\Phi_{03}$ 
($\Phi_{\infty1}, \Phi_{\infty2},$ and $\Phi_{\infty3}$).}
\label{fig:integration-paths}
\end{figure}

The domains in the complex $k$-plane in which the exponential factors are bounded are determined by the signature table shown in Figure~\ref{fig:sign}. In order for the analytic properties of $\Phi_{\infty 1}$ and $\Phi_{\infty 2}$ to remain unchanged for all $x\geq 0$ and $t\geq 0$, the quantity
$
(u\sqrt{m+1})(0,t)
$
must preserve its sign on the time interval under consideration. Thus, the two cases
\[
u(0,t)\leq 0,\qquad 0\leq t<T,
\]
and
\[
u(0,t)\geq 0,\qquad 0\leq t<T,
\]
lead to different analytic properties for the columns of $\Phi_{\infty 1}$ and $\Phi_{\infty 2}$.

\begin{figure}[ht]
    \centering
\begin{tikzpicture}[scale=1.2]

        \node at (0,-1) {$\mathrm{sign}\,\Im k$};
    
    \draw[thick] (-2.5,0) -- (2.5,0);
    
    \node at (0,0.5) {$+$};
    \node at (0,-0.5) {$-$};

    \node at (7,-1) {$\mathrm{sign}\,\Im \frac{1}{k}$};
    
    \draw[thick] (4.5,0) -- (9.5,0);
    
    \node at (7,0.5) {$-$};
    \node at (7,-0.5) {$+$};
\end{tikzpicture}
    \caption{Signature tables for $p(x,t,k)$  }
    \label{fig:sign}
\end{figure}
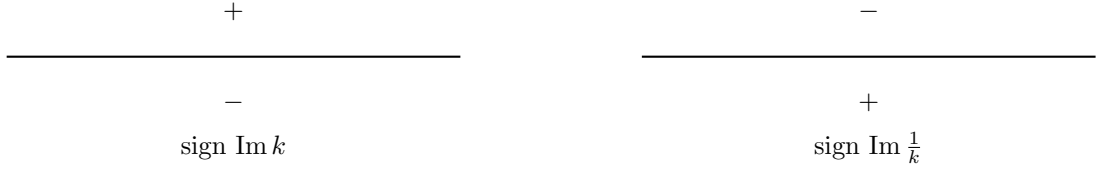

 We first consider the case $ u(0,t) \leq 0$.
Let $A^{(1)}$ and $A^{(2)}$ denote the columns of a $2\times 2$ matrix $A = \bigl( A^{(1)}\ \ A^{(2)} \bigr)$. With this notation, using Neumann series expansions, we obtain the following properties of $\Phi_{\infty i}^{(j)}(x,t,k)$:

\begin{enumerate}
    \item $\Phi_{\infty 1}^{(1)}(x,t,k)$ is analytic in $\mathbb{C}\setminus\{0\}$.
 Moreover, $\Phi_{\infty 1}^{(1)}(0,t,k)$=$\begin{pmatrix}
        1\\0
    \end{pmatrix}+O(\frac{1}{k})$ as $k\to\infty$ in $\mathbb{C}^-$.

    \item $\Phi_{\infty 1}^{(2)}(x,t,k)$ is analytic in $\mathbb{C}\setminus\{0\}$.
 Moreover, $\Phi_{\infty 1}^{(2)}(0,t,k)$=$\begin{pmatrix}
        0\\1
    \end{pmatrix}+O(\frac{1}{k})$ as $k\to\infty$ in $\mathbb{C}^+$.

     \item $\Phi_{\infty 2}^{(1)}(x,t,k)$ is analytic in $\mathbb{C}\setminus\{0\}$ and $\Phi_{\infty 2}^{(1)}(x,0,k)$ is analytic in $\mathbb{C}$.
 Moreover, $\Phi_{\infty 2}^{(1)}(x,t,k)$=$\begin{pmatrix}
        1\\0
    \end{pmatrix}+O(\frac{1}{k})$ as $k\to\infty$ in $\mathbb{C}^+$.

    \item $\Phi_{\infty 2}^{(2)}(x,t,k)$ is analytic in $\mathbb{C}\setminus\{0\}$ and $\Phi_{\infty 2}^{(2)}(x,0,k)$ is analytic in $\mathbb{C}$.
 Moreover, $\Phi_{\infty 2}^{(2)}(x,t,k)$=$\begin{pmatrix}
        0\\1
    \end{pmatrix}+O(\frac{1}{k})$ as $k\to\infty$ in $\mathbb{C}^-$.

      \item $\Phi_{\infty 3}^{(1)}(x,t,k)$ is analytic in $\mathbb{C}^-$ and continuous up to the boundary.
 Moreover, $\Phi_{\infty 3}^{(1)}(x,t,k)$=$\begin{pmatrix}
        1\\0
    \end{pmatrix}+O(\frac{1}{k})$ as $k\to\infty$.

    \item $\Phi_{\infty 3}^{(2)}(x,t,k)$ is analytic in $\mathbb{C}^+$ and continuous up to the boundary.
 Moreover, $\Phi_{\infty 3}^{(2)}(x,t,k)$=$\begin{pmatrix}
        0\\1
    \end{pmatrix}+O(\frac{1}{k})$ as $k\to\infty$.   
    
\end{enumerate}

We next consider the case \textbf{$ u(0,t) \geq 0$}. In this case, the corresponding analytic properties of the columns $\Phi_{\infty i}^{(j)}(x,t,k)$ are as follows.

\begin{enumerate}

   \item $\Phi_{\infty 1}^{(1)}(x,t,k)$ is analytic in $\mathbb{C}\setminus\{0\}$.
 Moreover, $\Phi_{\infty 1}^{(1)}(x,t,k)$=$\begin{pmatrix}
        1\\0
    \end{pmatrix}+O(\frac{1}{k})$ as $k\to\infty$ in $\mathbb{C}^+$.

    \item $\Phi_{\infty 1}^{(2)}(x,t,k)$ is analytic in $\mathbb{C}\setminus\{0\}$.
 Moreover, $\Phi_{\infty 1}^{(2)}(x,t,k)$=$\begin{pmatrix}
        0\\1
    \end{pmatrix}+O(\frac{1}{k})$ as $k\to\infty$ in $\mathbb{C}^-$.

     \item $\Phi_{\infty 2}^{(1)}(x,t,k)$ is analytic in $\mathbb{C}\setminus\{0\}$ and $\Phi_{\infty 2}^{(1)}(x,0,k)$ is analytic in $\mathbb{C}$.
 Moreover, $\Phi_{\infty 2}^{(1)}(x,0,k)$=$\begin{pmatrix}
        1\\0
    \end{pmatrix}+O(\frac{1}{k})$ as $k\to\infty$ in $\mathbb{C}^+$; $\Phi_{\infty 2}^{(1)}(0,t,k)$=$\begin{pmatrix}
        1\\0
    \end{pmatrix}+O(\frac{1}{k})$ as $k\to\infty$ in $\mathbb{C}^-$.

    \item $\Phi_{\infty 2}^{(2)}(x,t,k)$ is analytic in $\mathbb{C}\setminus\{0\}$ and $\Phi_{\infty 2}^{(2)}(x,0,k)$ is analytic in $\mathbb{C}$.
 Moreover, $\Phi_{\infty 2}^{(2)}(x,0,k)$=$\begin{pmatrix}
        0\\1
    \end{pmatrix}+O(\frac{1}{k})$ as $k\to\infty$ in $\mathbb{C}^-$;  Moreover, $\Phi_{\infty 2}^{(2)}(0,t,k)$=$\begin{pmatrix}
        0\\1
    \end{pmatrix}+O(\frac{1}{k})$ as $k\to\infty$ in $\mathbb{C}^+$.

       \item $\Phi_{\infty 3}^{(1)}(x,t,k)$ is analytic in $\mathbb{C}^-$ and continuous up to the boundary.
 Moreover, $\Phi_{\infty 3}^{(1)}(x,t,k)$=$\begin{pmatrix}
        1\\0
    \end{pmatrix}+O(\frac{1}{k})$ as $k\to\infty$.

    \item $\Phi_{\infty 3}^{(2)}(x,t,k)$ is analytic in $\mathbb{C}^+$ and continuous up to the boundary.
 Moreover, $\Phi_{\infty 3}^{(2)}(x,t,k)$=$\begin{pmatrix}
        0\\1
    \end{pmatrix}+O(\frac{1}{k})$ as $k\to\infty$.

\end{enumerate}

\begin{remark}
    In the present problem, the spectral parameter enters through $k$ and $\frac{1}{k}$. This distinguishes the corresponding RH formulation from that for the CH equation, where the spectral dependence involves $k$ and $\frac{k}{(4k^2+1)}$. In particular, in the present case, the signs of the imaginary parts of $k$ and $\frac{1}{k}$ separate the complex $k$-plane simply into the upper and lower half-planes. By contrast, in the CH case, the presence of $\frac{k}{(4k^2+1)}$ leads to a different partition of the spectral plane. In both problems, however, the formulation depends on the sign of the solution at $x=0$.
\end{remark}

Further, observing  that $U_\infty(k)\equiv U_\infty(x,t,k)$ and  $V_\infty(k)\equiv V_\infty(x,t,k)$ satisfy the symmetries
\begin{subequations}\label{sym-UV}
\begin{alignat}{3}\label{sym-U}
U_\infty(k)&=\sigma_1\overline{U_\infty(\bar k)}\sigma_1,&\qquad&U_\infty(k)=\sigma_1 U_\infty(-k)\sigma_1\\
V_\infty(k)&=\sigma_1\overline{V_\infty(\bar k)}\sigma_1,&\qquad&V_\infty(k)=\sigma_1 V_\infty(-k)\sigma_1,
\end{alignat}
\end{subequations}
and that  $p(k)\equiv p(x,t,k)$ satisfies the symmetries
\begin{equation}\label{sym-p}
p^*(k)=-p(k)=p(-k),
\end{equation}
it follows that
$\Phi_{\infty j}$ also satisfy the same symmetries as in \eqref{sym-UV}:
\begin{equation}\label{sym-Phi}
\Phi_{\infty j}(k)=\sigma_1\overline{\Phi_{\infty j}(\bar k)}\sigma_1,\quad\Phi_{\infty j}(k)=\sigma_1\Phi_{\infty j}(-k)\sigma_1.
\end{equation}

Also, since  the coefficients in  \eqref{Lax-1} are traceless matrices, 
it follows that
$\det \Phi_{\infty j}\equiv 1$.

\subsection{Eigenfunctions near $k=0$} Since the Lax pair is singular at both $k=0$ and $k=\infty$, we use a second gauge transformation to analyze the behavior of the eigenfunctions near $k=0$. The eigenfunctions obtained from the two gauge transformations will later be related on the auxiliary circle $\{|k|=\epsilon\}$.

Introducing
\begin{equation*}
\tilde{\Phi}_0
= G_0
\begin{pmatrix}
\psi \\ \psi_x
\end{pmatrix},
\end{equation*}
we transform the Lax pair \eqref{Lax} into
\begin{subequations}\label{Lax-2}
\begin{align}
&\tilde{\Phi}_{0x} + ik\sigma_3 \tilde{\Phi}_0 = U_0 \tilde{\Phi}_0,\\
&\tilde{\Phi}_{0t} + \frac{1}{2ik}\sigma_3 \tilde{\Phi}_0 = V_0 \tilde{\Phi}_0
\end{align}
\end{subequations}
where
\begin{subequations}
\begin{equation}\label{U0-hat}
 U_0(x,t,k):=-\frac{ik}{2} m (i\sigma_2+\sigma_3),
\end{equation}
and
\begin{equation}\label{hat-V0}
 V_0(x,t, k):=\frac{u_x}{2}\sigma_1
+ iku\left(\sigma_3+\frac{m}{2}(i\sigma_2+\sigma_3)\right).
\end{equation}
\end{subequations}
Introduce (compare with \eqref{p})
\begin{equation}\label{p_0mu}
p_0(x,t,k):=\ii k \left(x-\frac{t}{2 k^2}\right).
\end{equation}
Then introduce the solutions $\tilde \Phi_{0j}(x,t,k )$, 
$j=1,2,3$ to Equations 
\eqref{Lax-2} similarly to $\tilde \Phi_{\infty j}(x,t,k )$:
$\tilde\Phi_{0 j}(x,t,k):= \Phi_{0 j}(x,t,k )\eul^{-p_0(x,t,k)\sigma_3}$, where $\{\Phi_{0 j}(x,t,k)\}$ are the solutions of the integral equations:

\begin{equation}\label{inteq_01}
    \begin{aligned}
\Phi_{0 1}(x,t,k )=&I+
\int_{0}^{x}
	\eul^{-\ii k(x-y)\hat\sigma_3}( U_0\Phi_{01})(y,t,k ) \dd y  - \eul^{-\ii k x\hat\sigma_3}
\int_{t}^{T}
	\eul^{-\frac{\ii}{2 k}(\tau-t)\hat\sigma_3}( V_0\Phi_{01})(0,\tau,k ) \dd \tau,
\end{aligned}
\end{equation}

\begin{equation}\label{inteq_02}
  \begin{aligned}
\Phi_{02}(x,t,k )=&I+
\int_{0}^{x}
	\eul^{-\ii k(x-y)\hat\sigma_3}( U_0\Phi_{02})(y,t,k ) \dd y  + \eul^{-\ii k x\hat\sigma_3}
\int_{0}^{t}
	\eul^{-\frac{\ii}{2 k}(\tau-t)\hat\sigma_3}( V_0\Phi_{02})(0,\tau,k ) \dd \tau,
\end{aligned}  
\end{equation}    

\begin{equation}\label{inteq_03}
  \begin{aligned}
\Phi_{0 3}(x,t,k )=&I-
\int_{x}^{+\infty}
	\eul^{\ii k(y-x)\hat\sigma_3}( U_0\Phi_{03})(y,t,k ) \dd y.
\end{aligned}  
\end{equation}    

The properties of $\Phi_{0}(x,t,k )$ are independent of the sign of $u(0,t)$:

\begin{enumerate}
    \item $\Phi_{0 1}^{(1)}(x,t,k)$ is analytic in $\mathbb{C}\setminus\{0\}$.
 Moreover, $\Phi_{0 1}^{(1)}(x,t,k)=\begin{pmatrix}
        1\\0
    \end{pmatrix}+O(k)$ as $k\to 0$ in $\mathbb{C}^-$.

    \item $\Phi_{0 1}^{(2)}(x,t,k)$ is analytic in $\mathbb{C}\setminus\{0\}$.
 Moreover, $\Phi_{0 1}^{(2)}(x,t,k)=\begin{pmatrix}
        0\\1
    \end{pmatrix}+O(k)$ as $k\to 0$ in $\mathbb{C}^+$.

     \item $\Phi_{0 2}^{(1)}(x,t,k)$ is analytic in $\mathbb{C}\setminus\{0\}$.
 Moreover, $\Phi_{0 2}^{(1)}(x,t,k)=\begin{pmatrix}
        1\\0
    \end{pmatrix}+O(k)$ as $k\to 0$ in $\mathbb{C}^+$.

    \item $\Phi_{0 2}^{(2)}(x,t,k)$ is analytic in $\mathbb{C}\setminus\{0\}$ and $\Phi_{0 2}^{(2)}(x,0,k)$ is analytic in $\mathbb{C}$.
 Moreover, $\Phi_{0 2}^{(2)}(x,t,k)=\begin{pmatrix}
        0\\1
    \end{pmatrix}+O(k)$ as  $k\to 0$ in $\mathbb{C}^-$.

      \item $\Phi_{0 3}^{(1)}(x,t, k)$ is analytic in $\mathbb{C}^-$ and continuous up to the boundary.
 Moreover, $\Phi_{0 3}^{(1)}(x,t, k)$=$\begin{pmatrix}
        1\\0
    \end{pmatrix}+O(k)$ as $k\to 0$ in $\mathbb{C}^-$.

    \item $\Phi_{0 3}^{(2)}(x,t,k)$ is analytic in $\mathbb{C}^+$ and continuous up to the boundary.
 Moreover, $\Phi_{03}^{(2)}(x,t,k)=\begin{pmatrix}
        0\\1
    \end{pmatrix}+O(k)$ as $ k\to 0$ in $\mathbb{C}^+$.
\end{enumerate}

Since $U_0(x,t,k)$, $ V_0(x,t,k)$ satisfy the symmetries \eqref{sym-UV} and $p_0(x,t,k)$ satisfy the symmetries \eqref{sym-p}, it follows that $\Phi_{0 j}(x,t,k)$ satisfy the symmetries \eqref{sym-Phi}. 

\subsection{Spectral functions}

The functions $\Phi_{\infty j}$,
being the solutions of both equations of \eqref{Lax-Q-form}, are related (where they are defined) by matrices independent of $x$ and $t$:
\begin{subequations}\label{rel_inf}
    \begin{align}
        &\Phi_{\infty 3}(x,t,k)=\Phi_{\infty 2}(x,t,k)\eul^{-p(x,t,k)\sigma_3}s(k)\eul^{p(x,t,k)\sigma_3},\qquad k\in {\mathbb C}^+,\\
        &\Phi_{\infty 1}(x,t,k)=\Phi_{\infty 2}(x,t,k)\eul^{-p(x,t,k)\sigma_3}S(k)\eul^{p(x,t,k)\sigma_3}.
    \end{align}
\end{subequations}
Due to the symmetries of $U_\infty$ and $V_\infty$,  $s(k)$ and $S(k)$
can be written as 
        \begin{equation}
        \label{s}
s(k)=\Phi_{\infty 3}(0,0,k)=\begin{pmatrix}
        a^*(k) & b(k)\\
        b^*( k) & a(k)
    \end{pmatrix},\qquad k\in {\mathbb C}^+
    \end{equation}
and
              \begin{equation}
        \label{S}
S(k)=\Phi_{\infty 1}(0,0,k)=\begin{pmatrix}
        A^*( k) & B(k)\\
        B^*( k) & A( k)
    \end{pmatrix},
    \end{equation}
    with some $a( k)$, $b( k)$, $A( k)$, and  $B( k)$.

Similarly, the functions $\Phi_{0 j}$ are related (in the domain where they are defined) by matrices independent of $x$ and $t$:

\begin{subequations}\label{rel_0}
    \begin{align}\label{rel_0-a}
        &\Phi_{0 3}(x,t, k)=\Phi_{0 2}(x,t, k)\eul^{-p_0(x,t, k)\sigma_3}\tilde s( k)\eul^{p_0(x,t, k)\sigma_3},\qquad  k\in {\mathbb C}^+,\\ \label{rel_0-b}
        &\Phi_{0 1}(x,t, k)=\Phi_{0 2}(x,t, k)\eul^{-p_0(x,t, k)\sigma_3}\tilde S( k)\eul^{p_0(x,t, k)\sigma_3}.
    \end{align}
\end{subequations}
In turn,  symmetries of $U_0$ and $V_0$ imply
    \begin{equation}
        \label{tils}
        \tilde s( k)=\Phi_{0 3}(0,0, k)=\begin{pmatrix}
       \tilde a^*(  k) & \tilde b( k)\\
        \tilde b^*(  k) & \tilde a( k)
    \end{pmatrix}
    \end{equation}
and
          \begin{equation}
        \label{tilS}
\tilde S( k)=\Phi_{0 1}(0,0, k)=\begin{pmatrix}
        \tilde A^*(  k) &\tilde  B( k)\\
        \tilde B^*(  k) &\tilde  A( k)
    \end{pmatrix},
    \end{equation}
with some $\tilde a( k)$, $\tilde b( k)$, $\tilde A( k)$, and  $\tilde B( k)$.

Notice that
\[
G(x,t,k)G_0^{-1}(k)
= \frac12
\begin{pmatrix}
q+\frac{1}{q} & q-\frac{1}{q}\\
q-\frac{1}{q} & q+\frac{1}{q}
\end{pmatrix}=:Q(x,t),
\]
where
\[
q(x,t) := (m(x,t)+1)^{1/4}.
\]

Since $\Phi_0$ and $\Phi_\infty$ 
come from the 
 same system of ODEs \eqref{Lax}, they are related as follows:

\begin{subequations}\label{Phi_0_inf_rel}
\begin{equation}\label{Phi_0_inf_rel_}
 \Phi_{\infty j}(x,t, k)=Q(x,t)\Phi_{0 j}(x,t, k)\eul^{-p_0(x,t, k)\sigma_3}C_j( k)\eul^{p(x,t, k)\sigma_3}   
\end{equation}
with $C_j( k)$ independent of $x$ and $t$ given by:
\begin{align}\label{Phi_0_inf_coeff}
      C_1&=\eul^{-\frac{\ii T}{2 k}\sigma_3}Q^{-1}(0,T)\eul^{-\ii k\left(\eta(T)-\frac{T}{2 k^2}\right)\sigma_3},\\\label{Phi_0_inf_coeff2}
    C_2&=Q^{-1}(0,0),\\\label{Phi_0_inf_coeff3}
    C_3&=\eul^{-\ii k \nu(0)\sigma_3}.
\end{align}
\end{subequations}
In particular,
\begin{align}\label{s_via_til_s}
          &s(k)=Q(0,0)\tilde s(k)\eul^{-\ii k \nu(0)\sigma_3},\\\label{S_via_til_S}
          &S(k)=Q(0,0)\tilde S(k)C_1(k)=Q(0,0)\tilde S(k)\eul^{-\frac{\ii T}{2 k}\sigma_3}Q^{-1}(0,T)\eul^{-\ii k\left(\eta(T)-\frac{T}{2 k^2}\right)\sigma_3}
        \end{align}

In what follows, we use the notations
\[
Q(0,0)=
\begin{pmatrix}
    \kappa_1^0 & \kappa_2^0 \\
    \kappa_2^0 & \kappa_1^0
\end{pmatrix},
\]
where $\kappa_1^0$ and $\kappa_2^0$ are expressed in terms of a single quantity 
$q(0,0)$. Namely,
\begin{equation*}\label{kappas-x}
  \kappa_1^0:=\frac{1}{2}\left(q(0,0)+\frac{1}{q(0,0)}\right),
\qquad
\kappa_2^0:=\frac{1}{2}\left(q(0,0)-\frac{1}{q(0,0)}\right). 
\end{equation*}

Using this notation, \eqref{s_via_til_s} reads as
\begin{equation}\label{tilde-a--a}
a(k)=(\kappa_2^0\tilde b(k)+\kappa_1^0\tilde a(k))\eul^{\ii k\nu(0)},\qquad
b(k)=(\kappa_1^0\tilde b(k)+\kappa_2^0\tilde a(k))\eul^{\ii k\nu(0)}.
\end{equation}

\begin{remark}\label{rem:simpl}
In the case $u_{xx}(0,0)=0$, we have $Q^{-1}(0,0)=I$ 
 and thus \eqref{Phi_0_inf_rel_} for $j=2,3$
simplifies to 
\begin{subequations}
\begin{align*}\label{Phi_0_inf_reduced_2}
\Phi_{\infty 2}(x,t,k)& =Q(x,t)\Phi_{0 2}(x,t,k)
\eul^{\ii k h_2(x,t))\sigma_3},\\
\Phi_{\infty 3}(x,t,k)& =Q(x,t)\Phi_{0 3}(x,t,k)
\eul^{\ii k h_3(x,t))\sigma_3}
\end{align*}
\end{subequations}
with 
\[
h_2(x,t) = -\int_x^\infty (\sqrt{m+1}-1)(\xi,t)d\xi + \int_0^\infty
(\sqrt{m+1}-1)(\xi,0)d\xi,
\quad 
h_3(x,t) = -\int_x^\infty (\sqrt{m+1}-1)(\xi,t)d\xi.
\]
Consequently, in this case
\[
s(k)=\tilde s(k)\eul^{-\ii k \nu(0)\sigma_3},
\]
or 
\[
a(k) = \tilde{a}(k)\eul^{\ii k \nu(0)},\quad
b(k) = \tilde{b}(k)\eul^{\ii k \nu(0)},
\]
which 
results in a significant simplification of the subsequent analysis
and constructions.
\end{remark}

\section{Spectral Mappings: Direct Problems}\label{sec:3}

Since the integral equations for $\Phi_{03}(x,0, k)$ and $\Phi_{\infty 3}(x,0, k)$ involve only the
 initial condition $m_0(x):=-u_{0xx}(x)$, it follows that the direct $x$-spectral mapping
 \begin{equation*}
     \{m_0(x)\} \longrightarrow \{a( k), b( k)\}
 \end{equation*}
can be defined via the solution of the integral equation \eqref{inteq_inf3} taken at $t=0$
or it can be defined by  
\begin{equation*}
     \{ m_0(x)\} \longrightarrow \{\tilde a( k),\tilde b( k)\},
 \end{equation*}
in terms of  the solution of the integral equation \eqref{inteq_03} taken at $t=0$.

Similarly, since the integral equations for $\Phi_{01}(0,t, k)$ and $\Phi_{\infty 1}(0,t, k)$ involve only the boundary values $v_j(t)$, it follows that the direct $t$-spectral mapping
 \begin{equation*}
     \{v_j(t)\}_0^2 \longrightarrow \{A( k), B( k)\}
 \end{equation*}
can be defined via the solution of the integral equation \eqref{inteq_inf1} taken at $x=0$.
Alternatively, 
 the direct $t$-spectral mapping
\begin{equation*}
     \{v_j(t)\}_0^2 \longrightarrow \{\tilde A( k),\tilde B( k)\}
 \end{equation*}
can be defined via the solution of the integral equation \eqref{inteq_01} taken at $x=0$.

\subsection{The direct $x$-spectral problem \texorpdfstring{$\{  m_0(x) \} \longrightarrow \{ a( k),b( k) \}$}{m_0(x) -> (a(mu), b(mu))}}

Assume that $m_0\in H^{2,1}(0,\infty)$ and 
$1+m_0(x)>0$ for all $x\ge0$.
Consider Equation \eqref{inteq_inf3} for $t=0$:
\begin{equation}\label{inteq_inf3_t_0}    
\Phi_{\infty 3}(x,0, k)=I-
\int_{x}^{+\infty}
	\eul^{\ii k\int_x^y\sqrt{m+1}(\xi,0)\dd \xi\hat\sigma_3}( U_\infty\Phi_{\infty3})(y,0, k) \dd y
\end{equation}
with $U_\infty$ given via \eqref{U-hat}  with $m$ replaced by $m_0(x)$ (now we do not require $ m_0(x)$ to be the initial value of a solution of the SW equation). Then the solution of \eqref{inteq_inf3_t_0} evaluated at $x=0$
and \eqref{s} determine the eigenfunctions and the corresponding spectral functions  ${a}( k)$ and $ {b}( k)$  associated with $ m_0(x)$.

We emphasize that the analytic properties of $\Phi_{\infty 3}$ are independent of the sign of $u(0,t)$. Indeed, since $\Phi_{\infty 3}$ is defined by a Volterra integral equation with integration from $x$ to $+\infty$, the boundedness of the corresponding exponential factors is determined only by the sign of $\Im k$. Hence the columns $\Phi_{\infty 3}^{(1)}$ and $\Phi_{\infty 3}^{(2)}$ have the same analyticity domains in the cases $u(0,t)\leq 0$ and $u(0,t)\geq 0$. Consequently, the direct $x$-spectral problem, and therefore the spectral functions $a(k)$ and $b(k)$ defined through $\Phi_{\infty 3}(0,0,k)$, are determined solely by the initial profile $m_0(x)$.

 Analyzing the Volterra integral equation \eqref{inteq_inf3_t_0} yields the following properties of $a( k)$ and $b( k)$:

\begin{enumerate}
    \item $ a( k)$ and $ b( k)$ are analytic in $\mathbb{C}^+$. Moreover, $ a( k)=1+O(\frac{1}{ k})$ and $ b( k)=O(\frac{1}{ k})$ as $ k\to\infty$. 

    \item The determinant relation 
        \begin{equation*}\label{detrel_ab}
           a( k) a^*(  k)-  b( k) b^*(  k)=1,  \qquad k\in \mathbb R
        \end{equation*}
holds (since the matrices $U_\infty$ and $V_\infty$ are traceless).

    \item The following symmetry relations hold:
    \begin{equation*}\label{sym_a_}
        \overline{a(-\bar  k)}=a( k),\qquad \overline{b(-\bar  k)}=b( k).
    \end{equation*}

\end{enumerate}

\begin{remark}
For the Cauchy problem, one can show that the spectral function $a(k)$ has no zeros in its domain of analyticity. In the IBV problem considered here, however, the same argument does not apply, and therefore the absence of zeros of $a(k)$ cannot be concluded in general.
\end{remark}

In order to specify the behavior of $ a( k)$ and $ b( k)$ as $ k\to 0$, it is convenient to consider  Equation \eqref{inteq_03} for $t=0$:
\begin{equation}\label{inteq_03_t_0}    
\Phi_{0 3}(x,0, k)=I-
\int_{x}^{+\infty}
	\eul^{\ii k(y-x)\hat\sigma_3}( U_0\Phi_{03})(y,0, k) \dd y
\end{equation}
with $U_0$ given via \eqref{U0-hat} with $ m$ replaced by $m_0(x)$.
Then $\tilde{a}(k)$ and $\tilde{b}(k)$ are defined by (cf. \eqref{tils})
\[
\begin{pmatrix}
    \tilde{b}(k) \\ \tilde{a}(k)
\end{pmatrix} = \Phi_{03}^{(2)}(0,0,k).
\]
and  have the following properties:

\begin{itemize}
   
    \item $\tilde a(k)$ and $\tilde b(k)$ are analytic in $\mathbb{C}^+$. Moreover, $\tilde a(k)=1+O(k)$ and $\tilde b(k)=O(k)$ as $k\to 0$  in $\mathbb{C}^+$.

    \item The determinant relation 
        \begin{equation*}\label{detrel_tilab}
          \tilde a(k)\tilde a^*( k)- \tilde b(k)\tilde b^*( k)=1, \qquad k\in \mathbb R
        \end{equation*}
holds (since the matrices $U_0$ and $V_0$ are traceless).
    \item The following symmetry relations hold: 
    \begin{equation*}\label{sym_tila_}
    \overline{\tilde a(-\bar k)}=\tilde a(k),\qquad
    \overline{\tilde b(-\bar k)}=\tilde b(k).
\end{equation*}

\end{itemize}

Taking into account \eqref{tilde-a--a}, it follows that
\begin{enumerate}[(4)]
    \item Behavior at $ k=0$: 
\begin{equation}\label{a_at_i}
       a(k)=\kappa_1^0+O(k), \qquad 
    b(k)=\kappa_2^0+O(k),\qquad k\to 0,\qquad k\in\mathbb{C}^+.
   \end{equation}
\end{enumerate}

\subsection{The direct $t$-spectral problem \texorpdfstring{$\{v_j(t) \}_{j=0}^2 \longrightarrow \{{A}( k), {B}( k) \}$}{vj(x) -> (A(mu), B(mu))}}

Assume that $v_j\in H^{1,0}(0,T)$, $ j=0,1,2$, and $1-v_2(t)>0$ for $0\le t\le T$.
Consider Equation \eqref{inteq_inf1} for $x=0$:
\begin{equation}\label{inteq_inf1_x_0}    
\Phi_{\infty 1}(0,t, k)=I-\int_{t}^{T}
	\eul^{\left(-\ii k\int_t^\tau (u\sqrt{m+1})(0,s)\dd s-\frac{\ii }{2k}(\tau-t)\right)\hat\sigma_3}( V_\infty\Phi_{\infty1})(0,\tau, k) \dd \tau
\end{equation}
with $V_\infty$ given by \eqref{hat-V} with $u$, $m$ and $um_x$ replaced by $v_0(t)$, $-v_2(t)$, and $v_{2t}(t)-2 v_1(t)(1-v_2(t))$, respectively.  
Evaluating \eqref{inteq_inf1_x_0} at $t=0$
and \eqref{S} allows determining the eigenfunctions and the corresponding spectral functions $A( k)$ and $B( k)$  associated with $\{v_j(t)\}_0^2$.

Here we do not require $\{v_j(t)\}_0^2$ to be the boundary values of a solution of the SW equation, but if $v_0(t)\equiv 0$ on an interval $[T_1,T_2]$, then we require that $v_{2t}(t)-2v_{1}(t)(1-v_2(t))\equiv 0$ on this interval, which is consistent with the SW equation with $v_0\equiv 0$.

The spectral functions  $A( k)$ and $ B( k)$ satisfy the following properties:

\begin{enumerate}
    \item $ A( k)$ and $ B( k)$ are analytic in $\mathbb{C}\setminus\{0\}$. Moreover, $ A( k)=1+O\left(\frac{1}{ k}\right)+O\left(\frac{\eul^{2\ii k \eta(T)}}{k}\right)$ and $ B( k)=O\left(\frac{1}{ k}\right)+O\left(\frac{\eul^{2\ii k \eta(T)}}{k}\right)$ as $ k\to\infty$. Here,
    \[
    \eta(t)=-\int_0^t v_0(\tau)\sqrt{-v_2(\tau)+1}\dd\tau.
    \]

    \item The determinant relation 
\begin{equation*}\label{detrel_AB}
           A( k) A^*(  k)-  B( k) B^*(  k)=1  
        \end{equation*}
        holds.
    \item The following symmetry relations hold:

    \begin{equation*}\label{sym_A_}
    \overline{A(-\bar  k)}=A( k),\qquad
    \overline{B(-\bar  k)}=B(k).
    \end{equation*}
    
\end{enumerate}

The large-$k$ asymptotic properties of $A(k)$ and $B(k)$ depend on the sign of $v_0(t)$. This dependence arises through the function
$
\eta(t) .
$
Indeed, the sign of $v_0(t)$ determines the sign of $\eta(t)$, and therefore determines the half-plane in which the exponential factor
$
\eul^{2\ii k\eta(T)}
$
is bounded as $k\to\infty$. Consequently, 
\begin{itemize}
    \item if $v_0(t)\leq0$, then $ A(k)=1+O\left(\frac{1}{k}\right)$ and $ B(k)=O\left(\frac{1}{k}\right)$ as $k\to\infty$ in $\mathbb{C}^+$;

    \item if $v_0(t)\geq0$, then  $ A(k)=1+O\left(\frac{1}{k}\right)$ and $ B(k)=O\left(\frac{1}{k}\right)$ as $k\to\infty$ in $\mathbb{C}^-$.
\end{itemize}

Alternatively, we can define another direct $t$-spectral 
mapping 
\[
\{v_j(t) \}_0^2 \longrightarrow \{ \tilde{A}(k), \tilde{B}(k) \}
\]
considering  Equation \eqref{inteq_01} for $x=0$:
\begin{equation}\label{inteq_01_x_0}    
\Phi_{0 1}(0,t,k)=I-
\int_{t}^{T}
	\eul^{-\frac{\ii}{2k}(\tau-t)\hat\sigma_3}( V_0\Phi_{01})(0,\tau,k ) \dd \tau,
\end{equation}
where $V_0$ is given by \eqref{hat-V0} with $u$, $u_x$ and $m$ replaced by $v_0$, $v_1$ and $-v_2$, respectively. 
Using \eqref{inteq_01_x_0}, one defines the eigenfunctions and the corresponding spectral functions associated with the boundary data $\{v_j(t)\}_{j=0}^2$. In particular, the functions $\tilde A(k)$ and $\tilde B(k)$ satisfy the following properties, which are independent of the sign of $v_0(t)$:

\begin{enumerate}
    \item $\tilde A(k)$ and $\tilde B(k)$ are analytic in $\mathbb{C}\setminus\{0\}$. Moreover, $\tilde A(k)=1+O(k)+O\left(k\eul^{-\frac{\ii T}{k}}
\right)$ and $\tilde B(k)=O(k)+O\left(k\eul^{-\frac{\ii T}{k}}\right)$ as $k\to 0$.

    \item The determinant relation
        \begin{equation*}\label{detrel_tilAB}
          \tilde A(k)\tilde A^*(k)- \tilde B(k)\tilde B^*(k)=1  
        \end{equation*}
holds.
    \item The following symmetry relations hold:
    \begin{equation*}\label{sym_tilA_}
    \overline{\tilde A(-\bar k)}=\tilde A(k),\qquad 
    \overline{\tilde B(-\bar k)}=\tilde B(k).
\end{equation*}

\end{enumerate}

Moreover, Equation \eqref{S_via_til_S} implies that the behavior as $k\to\infty$  is given by
\begin{itemize}
    \item if $v_0(t)\leq0$, then $\tilde A(k)=\kappa_1^T\kappa_1^0\eul^{-\ii k\eta(T)}+O\left(\frac{\eul^{-\ii k\eta(T)}}{k}\right)$ and $\tilde B(k)=-\kappa_1^T\kappa_2^0\eul^{-\ii k\eta(T)}+O\left(\frac{\eul^{-\ii k\eta(T)}}{k}\right)$ as $k\to\infty$ in $\mathbb C^+$;
\item if $v_0(t)\geq0$, then $\tilde A(k)=\kappa_1^T\kappa_1^0\eul^{-\ii k\eta(T)}+O\left(\frac{\eul^{-\ii k\eta(T)}}{k}\right)$ and $\tilde B(k)=-\kappa_1^T\kappa_2^0\eul^{-\ii k\eta(T)}+O\left(\frac{\eul^{-\ii k\eta(T)}}{k}\right)$ as $k\to\infty$ in $\mathbb C^-$.
\end{itemize}

\section{Compatibility of initial and boundary values}\label{sec:4}

In this section, we derive the relations among the spectral functions.

Assume first that $\Phi_{\infty 3}$ is the solution of the Lax pair equations \eqref{Lax-Q-form}, where $u(x,t)$ satisfies the SW equation. Evaluating this eigenfunction at $x=0$, $t=T$, and using \eqref{rel_inf}, we obtain
    \begin{equation*}
        \Phi_{\infty3}(0,T, k)=\eul^{-p(0,T, k)\sigma_3}S^{-1}( k)s( k)\eul^{p(0,T, k)\sigma_3}.
    \end{equation*}
In particular, the $(12)$ entry is given by
\begin{equation*}
    \Phi^{(12)}_{\infty3}(0,T, k)=(A( k)b( k)-B( k)a( k))\eul^{2\ii k \int_0^T(u\sqrt{m+1})(0,\tau)\dd\tau}\eul^{\frac{\ii }{k}T}.
\end{equation*}
On the other hand, the analytic properties of $\Phi_{\infty 3}$ imply
 $ \Phi^{(12)}_{\infty3}=O\left(\frac{1}{ k}\right)$ as $ k\to\infty$ in $\mathbb{C}^+$.
Therefore, the spectral functions satisfy the asymptotic relation
\begin{equation}
\label{relations_Phi_inf3}
       (A( k)b( k)-B( k)a( k))\eul^{2\ii k \int_0^T(u\sqrt{m+1})(0,\tau)\dd\tau}=O\left(\frac{1}{ k}\right),\quad k\to\infty, \quad k\in\mathbb{C}^+.
    \end{equation}

Similarly, assume that $\Phi_{0 3}$ is the solution of the Lax pair equations \eqref{Lax-2}, where $u(x,t)$ satisfies the SW equation. Evaluating this eigenfunction at $x=0$, $t=T$, and using \eqref{rel_0}, we obtain
    \begin{equation*}
        \Phi_{03}(0,T, k)=\eul^{-p_0(0,T, k)\sigma_3}\tilde S^{-1}( k)\tilde s( k)\eul^{p_0(0,T, k)\sigma_3}.
    \end{equation*}
Thus its $(12)$ entry is
\begin{equation*}
    \Phi^{(12)}_{03}(0,T, k)=(\tilde A( k)\tilde b( k)-\tilde B( k)\tilde a( k))\eul^{\frac{\ii }{k}T}.
\end{equation*}
Since the properties of $\Phi_{0 3}$ imply
 $ \Phi^{(12)}_{03}=O( k)$ as $ k\to 0$, we obtain the second asymptotic relation
\begin{equation}\label{relations_Phi_03_i}
 (\tilde A( k)\tilde b( k)-\tilde B( k)\tilde a( k))\eul^{\frac{\ii }{k}T}=O( k),\quad k\to 0.
\end{equation} 

The asymptotic formulas \eqref{relations_Phi_inf3} and \eqref{relations_Phi_03_i} are the so-called
\emph{Global Relations}. They express the compatibility of the initial and boundary values in spectral terms, namely through the spectral functions $\{a( k), b( k), A( k), B( k)\}$ and $\{\tilde a( k), \tilde b( k), \tilde A( k), \tilde B( k)\}$.

Notice that there is an important difference  between 
the cases $u(0,t)\geq 0$ and $u(0,t)\leq 0$. Namely, 
in the case $u(0,t)\geq 0$, relation \eqref{relations_Phi_inf3} follows directly from the properties of the spectral functions listed in Section~\ref{sec:3}, and thus it does not 
present any additional restriction on the spectral functions generated by the initial and boundary values.

\section{Spectral Mappings: Inverse  Problems}\label{sec:5}

We now describe the inverse spectral mappings. These mappings are expressed in terms of solutions of the associated Riemann–Hilbert problems, whose jump matrices are constructed from the corresponding spectral functions.

For later use, it is convenient to introduce the functions
\begin{equation}\label{psi_0}
    \Psi_{0j}(x,t,k)\coloneqq Q(x,t)\Phi_{0j}(x,t,k)\eul^{(p-p_0)(x,t,k)\sigma_3}, 
\end{equation}
which will simplify the form of jump conditions.
Furthermore, we will assume that $\epsilon>0$ is sufficiently small so that all 
eventual
zeros of $\tilde a(k)$, $ a(k)$, $\tilde A(k)$, and $ A(k)$ appearing in the denominators below lie outside the disks $\{|k|\leq\epsilon\}$.

\subsection{The Inverse $x$-Spectral Mapping}

 The inverse $x$-spectral mapping
 \[ \{ a( k),b( k) \}\longrightarrow \{m_0(x) \}\]
is formulated in terms of a RH problem whose jump matrix is determined by the spectral functions $a(k)$ and $b(k)$. This construction is independent of the sign of the boundary value $u(0,t)$. Indeed, it relies only on the relation between the eigenfunctions $\Phi_{\infty 2}(x,0,k)$ and $\Phi_{\infty 3}(x,0,k)$ in the direct $x$-spectral problem, and therefore depends solely on the initial profile $m_0(x)$.

\begin{remark}\label{rem:data_ab}
 Notice that $\kappa_1^0$, $\kappa_2^0$, and $\nu(0)$ (pushing the expansion of $\eul^{\ii k\nu(0)}$ in \eqref{tilde-a--a} a step further) can be obtained by expanding $a(k)$ and $b(k)$ as $k\to 0$. In particular, $\kappa_1^0=a(0)$, $\kappa_2^0=b(0)$, and $\nu(0)=-\ii\frac{a'(0)+b'(0)}{a(0)+b(0)}$
 Therefore, having $a(k)$ and $b(k)$, we can determine $\tilde a(k)$ and $\tilde b(k)$ via \eqref{s_via_til_s}.
\end{remark}

\begin{figure}[ht]
    \centering

\begin{tikzpicture}[scale=2]

\draw[red] (0,0) circle (1);
\draw[thick] (-2,0) -- (2,0);

\draw[thick, postaction={decorate},
      decoration={markings, mark=at position 0.5 with {\arrow[scale=2]{>}}}]
      (-2,0) -- (-1,0);

\draw[thick, postaction={decorate},
      decoration={markings, mark=at position 0.7 with {\arrow[scale=2]{<}}}]
      (-1,0) -- (0,0);  

\draw[thick, postaction={decorate},
      decoration={markings, mark=at position 0.5 with {\arrow[scale=2]{>}}}]
      (1,0) -- (2,0);

\draw[thick, red, postaction={decorate},
      decoration={markings, mark=at position 0.25 with {\arrow[scale=2]{<}}}] 
      (0,0) circle (1);

\draw[thick, red, postaction={decorate},
      decoration={markings, mark=at position 0.75 with {\arrow[scale=2]{>}}}] 
      (0,0) circle (1);

\node at (0,0.1) {\textcolor{red}{$0$}};
\fill[red] (0,0) circle(0.02);

\node at (1.1,0.1) {\textcolor{blue}{$\epsilon$}};
\fill[blue] (1,0) circle(0.02);

\end{tikzpicture}

    \caption{The oriented contour for the Riemann--Hilbert problems \textbf{RH$^{(x)}$}, \textbf{RH$^{(t)}$}, and \textbf{RH$^{(xt)}$}}
    \label{fig:contour_RH_x}
\end{figure}
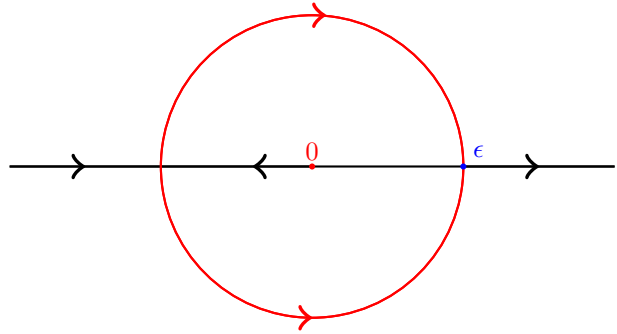

Now, define a matrix-valued function $M^{(x)}(x,k)$ with $\det M^{(x)}\equiv1$ in the domains separated by the contour in Figure \ref{fig:contour_RH_x}:

\begin{equation}\label{M_(x)}
M^{(x)}(x,k)=\begin{cases}

\left( \frac{\Phi_{\infty 2}^{(1)}(x,0,k)}{a(k)},\Phi_{\infty 3}^{(2)}(x,0,k)\right),\quad k\in\mathbb{C}^+\cap\{|k| >\epsilon\},\\

\left( \frac{\Psi_{0 2}^{(1)}(x,0,k)}{\tilde a(k)},\Psi_{0 3}^{(2)}(x,0,k)\right)\eul^{-\ii k\nu(0)\sigma_3},\quad k\in\mathbb{C}^+\cap\{|k| <\epsilon\},\\

\left( \Phi_{\infty 3}^{(1)}(x,0,k),\frac{\Phi_{\infty 2}^{(2)}(x,0,k)}{a^*( k)}\right),\quad k\in\mathbb{C}^-\cap\{|k| >\epsilon\},\\

\left( \Psi_{0 3}^{(1)}(x,0,k),\frac{\Psi_{0 2}^{(2)}(x,0,k)}{\tilde a^*( k)}\right)\eul^{-\ii k\nu(0)\sigma_3},\quad k\in\mathbb{C}^-\cap\{|k| <\epsilon\},
\end{cases}
\end{equation}
where the functions $\Phi_{\infty2}(x,0,k)$, $\Phi_{\infty3}(x,0,k)$, $\Phi_{02}(x,0,k)$, $\Phi_{03}(x,0,k)$ are to be understood as defined by \eqref{inteq_inf2}, \eqref{inteq_inf3}, \eqref{inteq_02}, and \eqref{inteq_03}, respectively,  for $t=0$, where $u(x,0)$ in $U_\infty(x,0)$ and $U_0(x,0)$ is replaced by $u_0(x)$.

\begin{enumerate}
    \item Jump relation across $\mathbb{R}\cup\{|k|=\epsilon\}$:
    \begin{subequations}
        \label{jump_M_(x)}
        \begin{equation}
           M_-^{(x)}(x,k)=M_+^{(x)}(x,k)J(x,k),\quad k\in \mathbb{R}\cup\{|k|=\epsilon\}, 
        \end{equation}
        where
         \begin{equation}
          J(x,k)=\eul^{-p(x,0,k)\sigma_3} J_0^{(x)}(k)\eul^{p(x,0,k)\sigma_3}
        \end{equation} 
        with $p(x,0,k)=\ii k \int_0^x \sqrt{m_0(\xi)+1}d\xi= \ii k \int_0^x \sqrt{-u_{0xx}+1}d\xi$ and
\begin{equation}\label{J0}
   J_0^{(x)}(k)=\begin{cases}
       \begin{pmatrix}
       1&-r^*(k)\\r(k)&1-r(k)r^*(k)
   \end{pmatrix},\quad k\in\mathbb{R}\cap\{|k| >\epsilon\},\\
   \begin{pmatrix}
       1-\tilde r(k)\tilde r^*(k)&\tilde r^*(k)\eul^{2\ii k\nu(0)}\\-\tilde r(k)\eul^{-2\ii k\nu(0)}&1
   \end{pmatrix},\quad k\in\mathbb{R}\cap\{|k| <\epsilon\},\\
   \begin{pmatrix}
       1&0\\
       \frac{\kappa_2^0\eul^{-\ii k\nu(0)}}{a(k)\tilde a(k)}&1
   \end{pmatrix},\quad k\in\mathbb{C}^+\cap\{|k| =\epsilon\},\\
   \begin{pmatrix}
       1&-\frac{\kappa_2^0\eul^{\ii k\nu(0)}}{a^*(k)\tilde a^*(k)}\\
       0&1
   \end{pmatrix},\quad k\in\mathbb{C}^-\cap\{|k| =\epsilon\},
   \end{cases} 
\end{equation}
where $r(k)=\frac{b^*(k)}{a(k)}$ and $\tilde r(k)=\frac{\tilde b^*(k)}{\tilde a(k)}$. 
    \end{subequations}

\begin{remark}
    In the case $\kappa_2^0=0$, $J_0^{(x)}(k)$ simplifies to
    \[
    J_0^{(x)}(k)=
       \begin{pmatrix}
       1&-r^*(k)\\r(k)&1-r(k)r^*(k)
   \end{pmatrix},\quad k\in\mathbb{R}.
    \]
    Here $\mathbb{R}$ is oriented from left to right.
\end{remark}

\item Behavior at $\infty$:
\begin{subequations}\label{inf_M_(x)}
\begin{equation}\label{inf_M_exp(x)}
     M^{(x)}(x,k)=I+\frac{1}{8\ii k}M^{\infty}(x)+o\left(\frac{1}{k}\right),\quad k\to\infty,
\end{equation}
with
\begin{equation}\label{inf_M_inf(x)}
M^{\infty}(x)=\begin{pmatrix}
    \frac{1}{4}\int_x^\infty
\frac{m_{0x}^2}{(m_0+1)^{\frac{5}{2}}}(y)\dd y&\frac{m_{0x}}{(m_0+1)^{\frac{3}{2}}}\\
    -\frac{m_{0x}}{(m_0+1)^{\frac{3}{2}}}&-\frac{1}{4}\int_x^\infty
\frac{m_{0x}^2}{(m_0+1)^{\frac{5}{2}}}(y)\dd y\end{pmatrix}.
\end{equation}
\end{subequations}

\begin{proof}
Expanding $\Phi_{\infty 2}^{(1)}(x,0, k)$ and $\Phi_{\infty 3}^{(2)}(x,0, k)$ at $ k=\infty$ in $\mathbb{C}^+$ via the Neumann series using integration by parts yields \eqref{inf_M_(x)}.
\end{proof}

\begin{remark}
   Under assumption {\rm (i)}, the large-$k$ expansions below hold with a remainder $o(k^{-1})$. If, in addition,  $m_0(x)\in H^{3,1}(0,\infty)$,  then the integration-by-parts argument can be iterated once more, and the remainder improves to $O(k^{-2})$. 
\end{remark}

\item $\det M^{(x)}(x,k)\equiv1$

\item Symmetry properties:
\begin{equation}\label{sym-M_(x)}
M^{(x)}( k)=\sigma_1\overline{M^{(x)}(\bar k)}\sigma_1,\qquad M^{(x)}(k)=\sigma_1M^{(x)}(-k)\sigma_1
\end{equation}

\item Residue properties. Let $\{k_j\}$ be zeros of $a(k)$ in $\mathbb{C}^+$. Assume that they are all simple. Then

\begin{align}\label{res-M+_(x)}
\Res_{k_j}M^{(x)(1)}(x,k)&=\frac{e^{2p(x,0,k_j)}}{\dot a(k_j)b(k_j)}M^{(x)(2)}(x,k_j),\\
\label{res-M-_(x)}
\Res_{\bar k_j}M^{(x)(2)}(x,k)&=\frac{e^{-2p(x,0,\bar k_j)}}{\dot a^*(\bar k_j)b^*(\bar k_j)}M^{(x)(1)}(x,\bar k_j).
\end{align}
 Notice that $\dot a^*(\bar k_j)b^*(\bar k_j)=\overline{\dot a( k_j)b( k_j)}$.   

\item Behavior at $0$:
\begin{subequations}\label{i_beh-M_(x)}
  \begin{equation}\label{i_beh-M_(x)_exp}
M^{(x)}(x,k)=Q_0(x)+O(k), \quad k\to 0, 
\end{equation}
where 
\begin{equation}
    \label{i_beh-M_(x)_coeff}
    Q_0(x)=\frac{1}{2}\begin{pmatrix}
q_0+\frac{1}{q_0} & q_0-\frac{1}{q_0}\\
q_0-\frac{1}{q_0} & q_0+\frac{1}{q_0}
\end{pmatrix}, \quad 
q_0(x)= (m_0(x)+1)^{1/4}.
\end{equation}
\end{subequations}
\begin{proof}
Let $u_0$ denote the solution of  $-u_{0xx}=m_0$ satisfying $u_0(x),u_{0x}(x)\to0$ as $x\to\infty$. Thus, $u_{0x}(x)=\int_x^\infty m_0(y)\dd y$,  $u_0(x)=-\int_x^\infty (y-x)m_0(y)\dd y$. 

    Expanding $\Phi_{02}^{(1)}(x,0, k)$ at $ k=0$ via the Neumann series and using integration by parts yields
\[
\Phi_{02}^{(1)}(x,0, k)=\begin{pmatrix}
    1\\0
\end{pmatrix}+\frac{\ii k}{2}\begin{pmatrix}u_{0x}(x)-u_{0x}(0)
    \\-u_{0x}(x)+u_{0x}(0)
\end{pmatrix}+ k^2\begin{pmatrix} 
    0\\u_0(x)-u_0(0)-xu_{0x}(0)
\end{pmatrix}+O(k^3).
\]
Similarly, expanding $\Phi_{03}^{(2)}(x,0, k)$ at $ k=0$, we obtain
\[
\Phi_{03}^{(2)}(x,0, k)=\begin{pmatrix}
    0\\1
\end{pmatrix}+\frac{\ii k}{2}\begin{pmatrix}u_{0x}(x)
    \\-u_{0x}(x)
\end{pmatrix}+ k^2\begin{pmatrix} 
    u_0(x)\\0
\end{pmatrix}+o(k^2),
\]
and, in particular,
\begin{equation}\label{tila_at_0_}
    \tilde a(k)=1-\frac{\ii k}{2}u_{0x}(0)+O(k^3)
\end{equation}
\end{proof}
    
\end{enumerate}

The properties of $M^{(x)}$ stated above (that follow from the analysis of the direct problem) can be interpreted as a 
family of RH factorization problems parametrized by $(x,t)$.
However, the construction of the jump matrix involves 
\[
p(x,0, k) =\ii k \int_0^x \sqrt{m_0(\xi)+1}d\xi
\]
which, in turn, involves $ m_0(x)$.
This suggests the introduction of a new variable
\begin{equation}\label{y_(x)}
y(x)=\int_0^x \sqrt{m_0(\xi)+1}d\xi,
\end{equation}
 which makes the RH problem  explicitly dependent
on $y$ as a parameter:

\textbf{The Riemann--Hilbert problem RH$^{x}$:} Given $a( k)$, $b(k)$ for $ k\in\mathbb{C}_+$ and a set $\{ k_j\}_1^N$ with $ k_j\in\mathbb{C}_+\setminus\{|k|<\epsilon\}$, find a  piecewise meromorphic $2\times 2$ matrix-valued function $\hat M^{(x)}(y, k)$ that satisfies the following conditions:

\begin{enumerate}
    \item Jump condition across $\mathbb{R}\cup\{|k|=\epsilon\}$
    \begin{subequations}
        \label{jump_hatM_(x)}
        \begin{equation}
           \hat M_-^{(x)}(y, k)=\hat M_+^{(x)}(y, k)\hat J(y, k),\quad k\in\mathbb{R}\cup\{|k|=\epsilon\} 
        \end{equation}
        where
         \begin{equation}
          \hat J(y, k)=\eul^{-\hat p(y,0, k)\sigma_3} J_0( k)\eul^{\hat p(y,0, k)\sigma_3}
        \end{equation} 
        with $\hat p(y,0, k)= \ii k y$ and $J_0( k)$ defined in \eqref{J0}.
    \end{subequations}

\item Behavior at $\infty$:
\begin{equation}\label{inf_hatM_(x)}
     \hat M^{(x)}(y, k)=I+O(\frac{1}{ k}),\quad k\to\infty.
\end{equation}

\item Residue conditions:  for $j=1,\dots,N$,
\begin{subequations}
    \label{res-hatM_(x)}
\begin{align}\label{res-hatM+_(x)}
\Res_{ k_j}\hat M^{(x)(1)}(y, k)&=\frac{e^{2\hat p(y,0, k_j)}}{\dot a( k_j)b( k_j)}\hat M^{(x)(2)}(y, k_j),\\
\label{res-hatM-_(x)}
\Res_{\bar k_j}\hat M^{(x)(2)}(y, k)&=\frac{e^{-2\hat p(y,0,\bar k_j)}}{\dot a^*(\bar k_j)b^*(\bar k_j)}\hat M^{(x)(1)}(y,\bar k_j).
\end{align}   
\end{subequations}
\end{enumerate}

\begin{proposition}
    \begin{enumerate}
        \item $\det \hat M^{(x)}\equiv1$
    
        \item If a solution of the RH problem \eqref{jump_hatM_(x)}--\eqref{res-hatM_(x)} exists, it is unique.

    \end{enumerate}
\end{proposition}

The uniqueness of the solution of the RH problem \textbf{RH$^{(x)}$} 
and properties of $M^{(x)}(x, k)$ justify the following procedure for the inverse mapping
\[
    \{a( k),\, b( k)\} \longrightarrow \{ m_0(x)\}
\]
for the $x$-problem:

\begin{enumerate}[Step 1.]
    \item Given $a( k)$ and $b( k)$, compute $\kappa_i^0$, $i=1,2$, $\tilde a(k)$ and $\tilde b(k)$ via \eqref{a_at_i} and \eqref{tilde-a--a}, and construct the RH problem \textbf{RH$^{(x)}$};

    \item Solve the constructed RH problem \textbf{RH$^{(x)}$};

    \item Evaluate the solution $ \hat M^{(x)}(y, k)$ of this RH problem  at $k=0$:
    \begin{equation*}
       \hat M^{(x)}(y, k)=\begin{pmatrix}
           \hat \alpha(y)&\hat \beta(y)\\
           \hat \beta(y)& \hat \alpha(y)
       \end{pmatrix}+O(k).
    \end{equation*}

\item

Then $q_0(x)=(m_0(x)+1)^{1/4}$ is given implicitly as follows
\begin{subequations}
    \begin{align}
\label{m_0_via_RH_x_m_2}
&\hat{m}_0(y)=\left(\hat \alpha(y)+\hat\beta(y)\right)^4-1,\\\label{M_(x)_x_y}
&x(y)=\int_0^y \tfrac{1}{\sqrt{\hat m_0(\xi)+1}}\dd \xi.
\end{align}
\end{subequations}

\end{enumerate}

Finally, having $m_0(x)$, we can reconstruct $u_0(x)$ by:
\begin{equation}\label{Green_m0}
    \begin{aligned}
   u_0(x)&=-\int_x^\infty(\xi-x)m_0(\xi)d\xi.
\end{aligned}
\end{equation}

\subsection{The Inverse $t$-Spectral Mapping}\label{sec:5.2}

The inverse $t$-spectral mapping
 \[ \{A(k), B(k), \tilde A( k), \tilde B( k) \}\longrightarrow \{ v_0(t), v_1(t),v_2(t)  \}\]
 is described in terms of the solution of the RH problem, whose jump matrix is determined by $ A$, $ B$, $\tilde A$, and $\tilde B$.
The construction of the RH problem follows from the relation
between the eigenfunctions $\Phi_{\infty1}(0,t, k)$, $\Phi_{\infty2}(0,t, k)$, $\Phi_{01}(0,t, k)$ and $\Phi_{02}(0,t, k)$ in the direct problem setting.

Throughout this section, we assume  $v_j\in H^{2,0}(0,T)$, $j=0,1,2$, and $1-v_2(t)>0$ for $0\leq t\leq T$. We also assume that the boundary data satisfy the
compatibility relation
\begin{equation}
    \label{inner_comp}
    v_0(t)\left(1-\frac{v_2(t)}{2}\right)
-\frac{v_1^2(t)}{4}
=
\frac{v_{1t}(t)}{2},
\qquad 0\leq t\leq T.
\end{equation}

\begin{remark}
    For boundary data that are traces of a classical solution, the
compatibility relation follows directly
from the SW equation. Indeed, integrating \eqref{SW} once with respect to
$x$ and using the decay at infinity gives
\[
\frac{u_{xt}}{2}
=
u\left(1+\frac{m}{2}\right)-\frac{u_x^2}{4}.
\]
Evaluating this identity at $x=0$ and using \eqref{boundary},
we obtain \eqref{inner_comp}
\end{remark}

\subsubsection[The Riemann--Hilbert problem formalism in the case $u(0,t)\le 0$]{\texorpdfstring{${u}(0,t)\le 0$}{u-negative}}
\label{sec:leq_inv}

We introduce the matrix-valued function $M^{(t)}(t,k)$, with $\det M^{(t)}\equiv1$, defined in
the domains separated by the contour in Figure \ref{fig:contour_RH_x}:

\begin{equation}\label{M_(t)}
M^{(t)}(t, k)=\begin{cases}

\left( \Phi_{\infty 2}^{(1)}(0,t, k),\frac{\Phi_{\infty 1}^{(2)}(0,t, k)}{A( k)}\right),\quad  k\in\mathbb{C}^+\cap\{| k| >\epsilon\},\\

\left(\Psi_{0 2}^{(1)}(0,t, k),\frac{\Psi_{0 1}^{(2)}(0,t, k)}{\tilde A( k)}\right),\quad  k\in\mathbb{C}^+\cap\{| k| <\epsilon\},\\

\left(\frac{\Phi_{\infty 1}^{(1)}(0,t, k)}{A^*( k)},\Phi_{\infty 2}^{(2)}(0,t, k)\right),\quad  k\in\mathbb{C}^-\cap\{| k| >\epsilon\}\\

\left(\frac{\Psi_{0 1}^{(1)}(0,t, k)}{\tilde A^*( k)},\Psi_{0 2}^{(2)}(0,t, k)\right),\quad  k\in\mathbb{C}^-\cap\{| k| <\epsilon\},
\end{cases}
\end{equation}
where the functions $\Psi_{0j}$
are defined by \eqref{psi_0}.

Then the function $M^{(t)}(t, k)$ has the following properties:

\begin{enumerate}
    \item Jump relation across $\mathbb{R}\cup\{|k|=\epsilon\}$
    \begin{subequations}
        \label{jump_M_(t)}
        \begin{equation}
           M_-^{(t)}(t, k)=M_+^{(t)}(t, k)J^{(t)}(t, k),\quad k\in \mathbb{R}\cup\{|k|=\epsilon\},
        \end{equation}
        where
         \begin{equation}
          J^{(t)}(t, k)=\eul^{-p(0,t, k)\sigma_3} J^{(t)}_0( k)\eul^{p(0,t, k)\sigma_3}        \end{equation} 
        with
\[p(0,t, k)= \ii k\left(\eta(t)-\frac{t}{2 k^2}\right),\]
 where $\eta(t)=-\int_0^tv_0(\tau)\sqrt{-v_2(\tau)+1}\dd\tau$, and
\begin{equation}\label{J_0_jump_M_(t)}
   J^{(t)}_0( k)=\begin{cases}
       \begin{pmatrix}
          1-R( k)R^*( k)&-R( k)\\
          R^*( k)&1
       \end{pmatrix},\quad  k\in\mathbb{R}\cap\{| k|>\epsilon\},\\
       \begin{pmatrix}
          1&\tilde R( k)\\
        -\tilde R^*( k)&1-\tilde R( k)\tilde R^*( k)
       \end{pmatrix},\quad  k\in\mathbb{R}\cap\{| k|<\epsilon\},\\
       
       \begin{pmatrix}
        \kappa_1^0-R(k)\kappa_2^0&\tilde R(k)\kappa_1^0-\tilde R(k)R(k)\kappa_2^0+\kappa_2^0-R(k)\kappa_1^0\\
          \kappa_2^0 & \kappa_1^0+\tilde R(k)\kappa_2^0
       \end{pmatrix},\quad  k\in\mathbb{C}^+\cap\{| k|=\epsilon\},\\

              \begin{pmatrix}
          \kappa_1^0-R^*(k)\kappa_2^0&-\kappa_2^0\\
         - \tilde R^*(k)\kappa_1^0+\tilde R^*(k)R^*(k)\kappa_2^0-\kappa_2^0+R^*(k)\kappa_1^0&\kappa_1^0+\tilde R^*(k)\kappa_2^0
       \end{pmatrix},\quad  k\in\mathbb{C}^-\cap\{| k|=\epsilon\}
       
   \end{cases}
\end{equation}
with $R( k):=\frac{B( k)}{A( k)}$ and $\tilde R( k):=\frac{\tilde B( k)}{\tilde A( k)}$.
    \end{subequations}

    \begin{remark}
         In the  case $Q(0,0)=I$ (i.e., with $\kappa_1^0=1$ and $\kappa_2^0=0$),
    the jump matrix $J_0^{(t)}(k)$ simplifies to
\begin{equation}\label{J_0_jump_M_(t)_simpl}
   J^{(t)}_0( k)=\begin{cases}
       \begin{pmatrix}
          1-R( k)R^*( k)&-R( k)\\
          R^*( k)&1
       \end{pmatrix},\quad  k\in\mathbb{R}\cap\{| k|>\epsilon\},\\
       \begin{pmatrix}
          1&\tilde R( k)\\
        -\tilde R^*( k)&1-\tilde R( k)\tilde R^*( k)
       \end{pmatrix},\quad  k\in\mathbb{R}\cap\{| k|<\epsilon\},\\
       
       \begin{pmatrix}
        1 &\tilde R(k)-R(k)\\
          0 & 1
       \end{pmatrix},\quad  k\in\mathbb{C}^+\cap\{| k|=\epsilon\},\\

              \begin{pmatrix}
          1&0\\
         - \tilde R^*(k)+R^*(k)&1
       \end{pmatrix},\quad  k\in\mathbb{C}^-\cap\{| k|=\epsilon\}.
       
   \end{cases}
\end{equation}
    \end{remark}

\item Behavior at $\infty$:
\begin{equation}\label{inf_M_(t)}
     M^{(t)}(t, k)=I+O\left(\frac{1}{ k}\right),\quad k\to\infty.
\end{equation}

\item $\det M^{(t)}(t, k)\equiv1$.

\item Symmetry properties:
\begin{equation}\label{sym-M_(t)}
M^{(t)}( k)=\sigma_1\overline{M^{(t)}(\bar k)}\sigma_1,\qquad M^{(t)}(k)=\sigma_1M^{(t)}(-k)\sigma_1.
\end{equation}

\item Residue properties. 
Let $\{\nu_j\}$ be zeros of $A( k)$  in $\mathbb{C}^+$.
Assume that they are all simple. Then

\begin{align}\label{res-M+_(t)}
\Res_{\nu_j}M^{(t)(2)}(t, k)&=\frac{B(\nu_j)e^{-2p(0,t,\nu_j)}}{\dot A(\nu_j)}M^{(t)(1)}(t,\nu_j),\\
\label{res-M-_(t)}
\Res_{\bar\nu_j}M^{(t)(1)}(t, k)&=\frac{B^*(\bar\nu_j)e^{2p(0,t,\bar\nu_j)}}{\dot A^*(\bar\nu_j)}M^{(t)(2)}(t,\bar\nu_j).
\end{align}

\item Behavior at $0$. 
\begin{subequations}\label{i_beh-M_(t)}
\begin{equation}
M^{(t)}(t, k)=Q(0,t)\left(I+\ii kM_1^{(t)}(t)+k^2M_2^{(t)}(t)+o(k^2)\right) \eul^{\ii k \eta(t)\sigma_3}, 
\end{equation}
where
\begin{equation}
    M_1^{(t)}=\begin{pmatrix}
    \int_0^t\frac{v_{1t}}{2}(\tau)\dd\tau&\frac{v_1}{2}\\
    -\frac{v_1}{2}&-\int_0^t\frac{v_{1t}}{2}(\tau)\dd\tau
\end{pmatrix}=\frac{1}{2}\begin{pmatrix}
    v_1(t)-v_1(0)&v_1(t)\\
    -v_1(t)&-v_1(t)+v_1(0)
\end{pmatrix}
\end{equation}

\begin{align}
    M_2^{(t)}&=\begin{pmatrix}
    *&\frac{v_{1t}}{2}+\frac{v_0v_2}{2}+\frac{v_1}{2}\int_0^t\frac{v_{1t}}{2}(\tau)\dd\tau\\
    \frac{v_{1t}}{2}+\frac{v_0v_2 m}{2}+\frac{v_1}{2}\int_0^t\frac{v_{1t}}{2}(\tau)\dd\tau&*
\end{pmatrix}=\\
&=\begin{pmatrix}
    *&v_0(t)-\frac{v_1(t)v_1(0)}{4}\\
   v_0(t)-\frac{v_1(t)v_1(0)}{4}&*
\end{pmatrix}.
\end{align}
\end{subequations}

\begin{proof}
    Expanding $\Phi_{02}^{(1)}(0,t, k)$ at $ k=0$ via the Neumann series and using integration by parts yields
  \begin{equation*}
      \begin{aligned}
          \Phi_{02}^{(1)}(0,t, k)&=\begin{pmatrix}
    1\\0
\end{pmatrix}+\ii k\begin{pmatrix}\int_0^t\left(u(1+\frac{m}{2})-\frac{u_x^2}{4}\right)(0,\tau)\dd\tau
    \\-\frac{u_x}{2}
\end{pmatrix}\\
&+ k^2\begin{pmatrix} 
    *\\ \frac{u_{xt}}{2}-\frac{u m}{2}+\frac{u_x}{2}\int_0^t\left(u(1+\frac{m}{2})-\frac{u_x^2}{4}\right)(0,\tau)\dd\tau
\end{pmatrix}+O(k^3).
      \end{aligned}
        \end{equation*}
Similarly, expanding $\Phi_{01}^{(2)}(0, t,k)$ at $ k=0$, we obtain
\begin{equation*}
    \begin{aligned}
        \Phi_{01}^{(2)}(0,t, k)&=\begin{pmatrix}
    0\\1
\end{pmatrix}+\ii k\begin{pmatrix} \frac{u_x}{2}
    \\\int_t^T\left(u(1+\frac{m}{2})-\frac{u_x^2}{4}\right)(0,\tau)\dd\tau
\end{pmatrix}\\&+ k^2 \begin{pmatrix} 
    \frac{u_{xt}}{2}-\frac{u m}{2}-\frac{u_x}{2}\int_t^T\left(u(1+\frac{m}{2})-\frac{u_x^2}{4}\right)(0,\tau)\dd\tau \\ *
\end{pmatrix}+O(k^3),
    \end{aligned}
\end{equation*}
and, in particular,
\[
\tilde A(k)=1+\ii k \int_0^T\left(u(1+\frac{m}{2})-\frac{u_x^2}{4}\right)(0,\tau)\dd\tau+O(k^2).
\]
Finally, \eqref{SW} implies 
\[
u\left(1+\frac{m}{2}\right)-\frac{u_x^2}{4}=\frac{u_{tx}}{2}.
\]
Therefore, using $u(0,t)=v_0(t)$,
$u_x(0,t)=v_1(t)$, $m(0,t)=-v_2(t)$ together with \eqref{inner_comp}, we obtain \eqref{i_beh-M_(t)}.
\end{proof}

\begin{remark}
 If, in addition,
$v_j\in H^{3,0}(0,T)$, $j=0,1,2$,
then one further integration by parts gives the sharper remainder
$O(k^3)$.
\end{remark}

\end{enumerate}

As in the $x$-spectral mapping, the properties of $M^{(t)}(t, k)$ stated above 
can be seen as those specifying a 
family of RH factorization problems.
However, we again face the same problem: the construction of the jump matrix involves $p(0,t, k)$
which in turn involves 
\begin{equation}\label{y_(t)}
\eta(t)=-\int_0^tv_0(\tau)\sqrt{-v_{2}(\tau)+1}\dd\tau. 
\end{equation}
Consequently, this suggests the introduction of a new parameter for the RH problem,
$z$, in terms of which $\hat J^{(t)}$ can be explicitly written.
Namely, introducing
\[\hat p(z,t, k)=\ii k\left(z-\frac{t}{2k^2}\right),\] 
we have that 
\[
p(0,t,k)=\hat p(\eta(t),t,k).
\]

Consider the following family of RH problems parametrized by $z$ and $t$.

\textbf{The Riemann--Hilbert problem RH$^{(t)}$:} Given $A(k)$, $B(k)$, $\tilde A( k)$, and $\tilde B( k)$ for $ k\in\mathbb{C}^+$, and a set $\{\nu_j\}_1^K\subset \mathbb{C}^+$, find a piecewise meromorphic $2\times 2$  matrix-valued function $\hat M^{(t)}(z,t, k)$ that satisfies the following conditions:

\begin{enumerate}
    \item Jump relation across $\mathbb{R}\cup\{|k|=\epsilon\}$
    \begin{subequations}
        \label{jump_hatM_(t)}
        \begin{equation}
           \hat M_-^{(t)}(z,t, k)=\hat M_+^{(t)}(z,t, k)\hat J^{(t)}(z,t, k),\quad k\in\mathbb{R}\cup\{|k|=\epsilon\},
        \end{equation}
        where
         \begin{equation}
          \hat J^{(t)}(z,t, k)=\eul^{-\ii k\left(z-\frac{t}{2k^2}\right)\sigma_3}J^{(t)}_0( k)\eul^{\ii k\left(z-\frac{t}{2k^2}\right)\sigma_3}
        \end{equation} 
        with $J^{(t)}_0( k)$ defined in \eqref{J_0_jump_M_(t)}.
        \end{subequations}

\item Behavior at $\infty$:
\begin{equation}\label{inf_hatM_(t)}
     \hat M^{(t)}(z,t, k)=I+O(\frac{1}{ k}),\quad k\to\infty.
\end{equation}

\item Residue conditions: for $j=1,\dots,K$,
\begin{subequations}
  \label{res-hatM_(t)}
\begin{align}\label{res-hatM+_(t)}
\Res_{\nu_j}\hat M^{(t)(2)}(z,t, k)&=\frac{B(\nu_j)e^{-2\hat p(z,t,\nu_j)}}{\dot A(\nu_j)}\hat M^{(t)(1)}(z,t,\nu_j),\\
\label{res-hatM-_(t)}
\Res_{\bar\nu_j}\hat M^{(t)(1)}(z,t, k)&=\frac{B^*(\bar\nu_j)e^{2\hat p(z,t,\bar\nu_j)}}{\dot A^*(\bar\nu_j)}\hat M^{(t)(2)}(z,t,\bar\nu_j).
\end{align}
\end{subequations}

\end{enumerate}

\begin{proposition}
    \begin{enumerate}
        \item $\det \hat M^{(t)}\equiv1$.
    
        \item If the solution of the RH problem \eqref{jump_hatM_(t)}--\eqref{res-hatM_(t)} exists, it is unique.
 
    \end{enumerate}
\end{proposition}

The uniqueness of the solution of the RH problem \eqref{jump_hatM_(t)}--\eqref{res-hatM_(t)} justifies the following procedure for the inverse mapping
\[
    \{A(k), B(k), \tilde A( k),\, \tilde B( k) \} \longrightarrow \{v_0(t), v_1(t), v_2(t)\}
\]
for the $t$-problem:

\begin{enumerate}[Step 1.]
    \item Given $ A( k)$, $ B( k)$, $\tilde A( k)$, $\tilde B( k)$ construct the RH problem  \eqref{jump_hatM_(t)}--\eqref{res-hatM_(t)};

    \item Solve the constructed RH problem \eqref{jump_hatM_(t)}--\eqref{res-hatM_(t)};

    \item Evaluate the solution $\hat M^{(t)}(z,t, k)$ of this RH problem  at $ k=0$:

    \begin{equation*}\label{M(t)_at_0}
       \hat M^{(t)}(z,t, k)=\begin{pmatrix}
           \hat \alpha(z,t)&\hat \beta(z,t)\\
           \hat \beta(z,t)&\hat \alpha(z,t)
       \end{pmatrix}\left( I+\ii k\begin{pmatrix}
           \hat f_1(z,t)&\hat f_2(z,t)\\-\hat f_2(z,t)&-\hat f_1(z,t)
       \end{pmatrix}+k^2\begin{pmatrix}
          *& \hat g(z,t)\\  \hat g(z,t)&*
       \end{pmatrix}+o(k^2)\right)\eul^{\ii k \rho(t)\sigma_3}. 
    \end{equation*}

    \item Define $\hat v_j(z,t)$ from these expansions in the following way (cf. \eqref{i_beh-M_(t)}):
    \begin{subequations}
        \label{hat_v_j_via_RH_t}
         \begin{align}
     \hat v_0(z,t)&=\hat g-\hat f_2\hat f_1+\hat f_2^2, \\
     \hat v_1(z,t)&= 2\hat f_2,\\
     \hat v_2(z,t)&=1-(\hat \alpha+\hat \beta)^4.
    \end{align}
    \end{subequations}

    Define $z(t)$ as the solution of the differential equation (cf. \eqref{y_(t)}):
\begin{subequations}\label{y_(t)_via_RH_t}
    \begin{align}
     &\frac{\dd z}{\dd t}=-(\hat g-\hat f_2\hat f_1+\hat f_2^2)(\hat \alpha+\hat \beta)^2,\\
     &z(0)=0.
    \end{align}
\end{subequations}
    
Then
\begin{equation}\label{v_j_via_RH_t}
     v_j(t)=
     \hat v_j(z(t),t), \quad j=0,1,2.
    \end{equation}

\end{enumerate}

\subsubsection[The Riemann--Hilbert problem formalism in the case $u(0,t)\ge 0$]{\texorpdfstring{${u}(0,t)\ge 0$}{u-positive}}
\label{sec:geq_inv}

We introduce the matrix-valued function $M^{(t)}(t,k)$, with $\det M^{(t)}\equiv1$, defined in
the domains separated by the contour in Figure \ref{fig:contour_RH_x}:

\begin{equation}\label{M_(t)_geq}
M^{(t)}(t, k)=\begin{cases}

\left(\frac{\Phi_{\infty 1}^{(1)}(0,t, k)}{A^*( k)} ,\Phi_{\infty 2}^{(2)}(0,t, k)\right),\quad  k\in\mathbb{C}^+\cap\{| k| >\epsilon\},\\

\left(\Psi_{0 2}^{(1)}(0,t, k),\frac{\Psi_{0 1}^{(2)}(0,t, k)}{\tilde A( k)}\right),\quad  k\in\mathbb{C}^+\cap\{| k| <\epsilon\},\\

\left(\Phi_{\infty 2}^{(1)}(0,t, k) ,\frac{\Phi_{\infty 1}^{(2)}(0,t, k)}{A( k)}\right),\quad  k\in\mathbb{C}^-\cap\{| k| >\epsilon\}\\

\left(\frac{\Psi_{0 1}^{(1)}(0,t, k)}{\tilde A^*( k)},\Psi_{0 2}^{(2)}(0,t, k)\right),\quad  k\in\mathbb{C}^-\cap\{| k| <\epsilon\},
\end{cases}
\end{equation}
where the functions $\Psi_{0j}$
are defined by \eqref{psi_0}.

Then the function $M^{(t)}(t, k)$ has the following properties:

\begin{enumerate}
    
    \item Jump relation across $\mathbb{R}\cup\{|k|=\epsilon\}$
    \begin{subequations}
        \label{jump_M_(t)_geq}
        \begin{equation}
           M_-^{(t)}(t, k)=M_+^{(t)}(t, k)J^{(t)}(t, k),\quad k\in \mathbb{R}\cup\{|k|=\epsilon\}
        \end{equation}
        where
         \begin{equation}
          J^{(t)}(t, k)=\eul^{-p(0,t, k)\sigma_3} J^{(t)}_0( k)\eul^{p(0,t, k)\sigma_3}        \end{equation} 
        with
\[p(0,t, k)= \ii k\left(\eta(t)-\frac{t}{2 k^2}\right),\]
 where $\eta(t)=-\int_0^tv_0(\tau)\sqrt{-v_2(\tau)+1}\dd\tau$, and
\begin{equation}\label{J_0_jump_M_(t)_geq}
   J^{(t)}_0( k)=\begin{cases}
       \begin{pmatrix}
          1&R( k)\\
          -R^*( k)&1-R( k)R^*( k)
       \end{pmatrix},\quad  k\in\mathbb{R}\cap\{| k|>\epsilon\},\\
       
       \begin{pmatrix}
          1&\tilde R( k)\\
        -\tilde R^*( k)&1-\tilde R( k)\tilde R^*( k)
       \end{pmatrix},\quad  k\in\mathbb{R}\cap\{| k|<\epsilon\},\\

              \begin{pmatrix}
      \kappa_1^0    &\kappa_2^0+\tilde R\kappa_1^0\\\kappa_2^0-
           R^*\kappa_1^0
   & \kappa_1^0(1-\tilde R R^*)+\kappa_2^0(\tilde R- R^*)
       \end{pmatrix}
       
       ,\quad k\in\mathbb{C}^+\cap\{| k|=\epsilon\},\\

\begin{pmatrix}
        \kappa_1^0&R\kappa_1^0-\kappa_2^0
           \\
         -\kappa_2^0-\tilde R^*\kappa_1^0 & \kappa_1^0(1-\tilde R ^* R)+\kappa_2^0(\tilde R^*- R)
       \end{pmatrix}

       ,\quad  k\in\mathbb{C}^-\cap\{| k|=\epsilon\}
       
   \end{cases}
\end{equation}
with $R( k):=\frac{B( k)}{A( k)}$ and $\tilde R( k):=\frac{\tilde B( k)}{\tilde A( k)}$.
    \end{subequations}

    \begin{remark}
    In the  case $Q(0,0)=I$ (with $\kappa_1^0=1$ and $\kappa_2^0=0$),
    the jump matrix $J_0^{(t)}(k)$ simplifies to
\begin{equation}\label{J_0_jump_M_(t)_simpl_geq}
   J^{(t)}_0( k)=\begin{cases}
       \begin{pmatrix}
          1&R( k)\\
          -R^*( k)&1-R( k)R^*( k)
       \end{pmatrix},\quad  k\in\mathbb{R}\cap\{| k|>\epsilon\},\\
       \begin{pmatrix}
          1&\tilde R( k)\\
        -\tilde R^*( k)&1-\tilde R( k)\tilde R^*( k)
       \end{pmatrix},\quad  k\in\mathbb{R}\cap\{| k|<\epsilon\},\\
       
       \begin{pmatrix}
        1 &\tilde R(k)\\
          -R^* & 1-\tilde R(k) R^* 
       \end{pmatrix},\quad  k\in\mathbb{C}^+\cap\{| k|=\epsilon\},\\

              \begin{pmatrix}
          1&R\\
         - \tilde R^*(k)&1-R\tilde R^*
       \end{pmatrix},\quad  k\in\mathbb{C}^-\cap\{| k|=\epsilon\}.
       
   \end{cases}
\end{equation}
    \end{remark}

    \item Behavior at $\infty$:
\begin{equation}\label{inf_M_(t)_geq}
     M^{(t)}(t, k)=I+O\left(\frac{1}{ k}\right),\quad k\to\infty.
\end{equation}

\item $\det M^{(t)}(t, k)\equiv1$

\item Symmetry properties:
\begin{equation}\label{sym-M_(t)_geq}
M^{(t)}( k)=\sigma_1\overline{M^{(t)}(\bar k)}\sigma_1,\qquad M^{(t)}(k)=\sigma_1M^{(t)}(-k)\sigma_1
\end{equation}

\item Residue properties. 
Let $\{\nu_j\}$ be zeros of $A( k)$  in $\mathbb{C}^-$.
Assume that they are all simple. Then

\begin{align}\label{res-M+_(t)_geq}
\Res_{\nu_j}M^{(t)(2)}(t, k)&=\frac{B(\nu_j)e^{-2p(0,t,\nu_j)}}{\dot A(\nu_j)}M^{(t)(1)}(t,\nu_j),\\
\label{res-M-_(t)_geq}
\Res_{\bar\nu_j}M^{(t)(1)}(t, k)&=\frac{B^*(\bar\nu_j)e^{2p(0,t,\bar\nu_j)}}{\dot A^*(\bar\nu_j)}M^{(t)(2)}(t,\bar\nu_j).
\end{align}

\item Behavior at $0$. 
\begin{subequations}\label{i_beh-M_(t)_geq}
\begin{equation}
M^{(t)}(t, k)=Q(0,t)\left(I+\ii kM_1^{(t)}(t)+k^2M_2^{(t)}(t)+O(k^3)\right) \eul^{\ii k \eta(t)\sigma_3}, 
\end{equation}
where
\begin{equation}
    M_1^{(t)}=\begin{pmatrix}
    \int_0^t\frac{v_{1t}}{2}(\tau)\dd\tau&\frac{v_1}{2}\\
    -\frac{v_1}{2}&-\int_0^t\frac{v_{1t}}{2}(0,\tau)\dd\tau
\end{pmatrix}=\frac{1}{2}\begin{pmatrix}
    v_1(t)-v_1(0)&v_1(t)\\
    -v_1(t)&-v_1(t)+v_1(0)
\end{pmatrix},
\end{equation}

\begin{align}
    M_2^{(t)}&=\begin{pmatrix}
    *&\frac{v_{1t}}{2}+\frac{v_0v_2}{2}+\frac{v_1}{2}\int_0^t\frac{v_{1t}}{2}(\tau)\dd\tau\\
    \frac{v_{1t}}{2}+\frac{v_0v_2}{2}+\frac{v_1}{2}\int_0^t\frac{v_{1t}}{2}(\tau)\dd\tau&*
\end{pmatrix}=\\
&=\begin{pmatrix}
    *&v_0(t)-\frac{v_1(t)v_1(0)}{4}\\
    v_0(t)-\frac{v_1(t)v_1(0)}{4}&*
\end{pmatrix}
\end{align}
    
\end{subequations}

\end{enumerate}

The properties of $M^{(t)}(t,k)$ define a family of RH factorization problems. As before, the jump matrix contains the phase $p(0,t,k)$, which depends on
$\eta(t)=-\int_0^t v_0(\tau)\sqrt{1-v_2(\tau)}\,\dd\tau $.
We therefore introduce an independent parameter $z$ and set
$\hat p(z,t,k)=\ii k\left(z-\frac{t}{2k^2}\right)$.
Then, for
$z=z(t)=\eta(t)$,
one has $p(0,t,k)=\hat p(\eta(t),t,k)$.

Consider the following family of RH problems parametrized by $z$ and $t$:

\textbf{The Riemann--Hilbert problem RH$^{(t)}$:} Given $A(k)$ and $B(k)$ for $ k\in\mathbb{C}^-$, $\tilde A( k)$ and $\tilde B( k)$ for $ k\in\mathbb{C}$ and a set $\{\nu_j\}_1^K\subset \mathbb{C}^-$, find a piecewise meromorphic $2\times 2$  matrix-valued function $\hat M^{(t)}(z,t, k)$ that satisfies the following conditions:

\begin{enumerate}
    \item Jump relation across $\mathbb{R}\cup\{|k|=\epsilon\}$
    \begin{subequations}
        \label{jump_hatM_(t)_geq}
        \begin{equation}
           \hat M_-^{(t)}(z,t, k)=\hat M_+^{(t)}(z,t, k)\hat J^{(t)}(z,t, k),\quad k\in\mathbb{R}\cup\{|k|=\epsilon\},
        \end{equation}
        where
         \begin{equation}
          \hat J^{(t)}(z,t, k)=\eul^{-\ii k\left(z-\frac{t}{2k^2}\right)\sigma_3}J^{(t)}_0( k)\eul^{\ii k\left(z-\frac{t}{2k^2}\right)\sigma_3}
        \end{equation} 
        with $J^{(t)}_0( k)$ defined in \eqref{J_0_jump_M_(t)_geq}.
        \end{subequations}

\item Behavior at $\infty$:
\begin{equation}\label{inf_hatM_(t)_geq}
     \hat M^{(t)}(z,t, k)=I+O(\frac{1}{ k}),\quad k\to\infty.
\end{equation}

\item Residue conditions: for $j=1,\dots,K$,
\begin{subequations}
    \label{res-hatM_(t)_geq}
\begin{align}\label{res-hatM+_(t)_geq}
\Res_{\nu_j}\hat M^{(t)(2)}(z,t, k)&=\frac{B(\nu_j)e^{-2\hat p(z,t,\nu_j)}}{\dot A(\nu_j)}\hat M^{(t)(1)}(z,t,\nu_j),\\
\label{res-hatM-_(t)_geq}
\Res_{\bar\nu_j}\hat M^{(t)(1)}(z,t, k)&=\frac{B^*(\bar\nu_j)e^{2\hat p(z,t,\bar\nu_j)}}{\dot A^*(\bar\nu_j)}\hat M^{(t)(2)}(z,t,\bar\nu_j).
\end{align}
\end{subequations}

\end{enumerate}

\begin{proposition}
    \begin{enumerate}
        \item $\det \hat M^{(t)}\equiv1$.
    
        \item If the solution of the RH problem \eqref{jump_hatM_(t)_geq}--\eqref{res-hatM_(t)_geq} exists, it is unique.
 
    \end{enumerate}
\end{proposition}

The uniqueness of the solution of the RH problem \eqref{jump_hatM_(t)_geq}--\eqref{res-hatM_(t)_geq} justifies the following procedure for the inverse mapping
\[
    \{A(k), B(k),\tilde A( k),\, \tilde B( k) \} \longrightarrow \{v_0(t), v_1(t), v_2(t)\}
\]
for the $t$-problem:

\begin{enumerate}[Step 1.]
    \item Given $ A( k)$, $ B( k)$, $\tilde A( k)$, $\tilde B( k)$ construct the RH problem  \eqref{jump_hatM_(t)_geq}--\eqref{res-hatM_(t)_geq};

    \item Solve the constructed RH problem \eqref{jump_hatM_(t)_geq}--\eqref{res-hatM_(t)_geq};

    \item Evaluate the solution $\hat M^{(t)}(z,t, k)$ of this RH problem  at $ k=0$:

    \begin{equation}\label{M(t)_at_0_geq}
       \hat M^{(t)}(z,t, k)=\begin{pmatrix}
           \hat \alpha(z,t)&\hat \beta(z,t)\\
           \hat \beta(z,t)&\hat \alpha(z,t)
       \end{pmatrix}\left( I+\ii k\begin{pmatrix}
           \hat f_1(z,t)&\hat f_2(z,t)\\-\hat f_2(z,t)&-\hat f_1(z,t)
       \end{pmatrix}+k^2\begin{pmatrix}
          *& \hat g(z,t)\\  \hat g(z,t)&*
       \end{pmatrix}+o(k^2)\right)\eul^{\ii k \rho(t)\sigma_3}. 
    \end{equation}

    \item Define $\hat v_j(z,t)$ from these expansions in the following way (cf. \eqref{i_beh-M_(t)_geq}):
    \begin{subequations}
        \label{hat_v_j_via_RH_t_geq}   
    \begin{align}
     \hat v_0(z,t)&=\hat g-\hat f_2\hat f_1+\hat f_2^2, \\
     \hat v_1(z,t)&= 2\hat f_2,\\
     \hat v_2(z,t)&=1-(\hat \alpha+\hat \beta)^4.
    \end{align}
  \end{subequations}

    Define $z(t)$ as the solution of the differential equation:
\begin{subequations}\label{y_(t)_via_RH_t_geq}
      \begin{align}
     &\frac{\dd z}{\dd t}=-(\hat g-\hat f_2\hat f_1+\hat f_2^2)(\hat \alpha+\hat \beta)^2,\\
     &z(0)=0.
    \end{align}
\end{subequations}
Then
\begin{equation}\label{v_j_via_RH_t_geq}
     v_j(t)=
     \hat v_j(z(t),t), \quad j=0,1,2.
    \end{equation}

\end{enumerate}

\section{The SW equation,  associated Lax pair, and  Riemann--Hilbert formalism in $y,t$ variables}\label{sec:SWinyt}

\subsection{The SW equation and the  associated Lax pair in $y,t$ variables}

To express the data of the RH problem explicitly in terms of the initial data alone, we introduce a new spatial variable $y=y(x,t)$ satisfying
\begin{subequations}\label{y-xt}
\begin{align}
\label{y_x} y_x & = \sqrt{m+1},\\
   \label{y_t} y_t &=-u \sqrt{m+1}.
\end{align}
    \end{subequations}
  These relations are compatible due to \eqref{SW} and determine $y$ up to an additive constant, which is fixed differently depending on the spatial domain.

In particular, for the problem on the line, this can be done by setting
\begin{equation*}
    y(x,t)=x+\int_{-\infty}^x \sqrt{m(\xi,t)+1}-1\dd\xi
\end{equation*}
whereas for the problem on the half-line $x\in (0,+\infty)$ we set
\begin{equation}\label{y}
    y(x,t)=\int_0^x \sqrt{m(\xi,t)+1}\dd\xi-\int_0^t u(0,\zeta)\sqrt{m(0,\zeta)+1}\dd\zeta.
\end{equation}
Thus for the half-line problem, we fix the additive constant by the normalization $y(0,0)=0$.

Then, differentiating the identity $x(y(x,t),t)=x$, we obtain the equations characterizing the 
inverse mapping $x=x(y,t)$:
\begin{subequations}\label{x-yt}
\begin{align}
\label{x_y} x_y & = \frac{1}{\sqrt{\hat m+1}},\\
   \label{x_t} x_t &=\hat u,
\end{align}
    \end{subequations}
where $\hat{ m }(y,t) = m(x(y,t),t)$ and  $\hat{ u }(y,t) = u(x(y,t),t)$.

\begin{proposition} (SW equation in the $(y,t)$ variables)  Let $u(x,t)$ and $m(x,t)$ ($m(x,t)+1>0$) satisfy \eqref{SW} and let $y(x,t)$ be 
such that \eqref{y-xt} hold.
Then   $\hat m(y,t):=m(x(y,t),t)$, $\hat u(y,t):=u(x(y,t),t)$, and 
$\hat v(y,t):=u_x(x(y,t),t)$ satisfy the following system of equations:

\begin{subequations}\label{SW_in_y}
\begin{align}\label{SW_in_y-1}
&(\sqrt{\hat m+1})_t=-\hat v\sqrt{\hat m+1},\\\label{SW_in_y-2}
&\hat v=\hat u_y\sqrt{\hat m+1},\\\label{SW_in_y-3}
&\hat m=-\hat v_y\sqrt{\hat m+1}.
\end{align}
\end{subequations}
    
\end{proposition}

\begin{proof}
As we discussed above, \eqref{y-xt} implies \eqref{x-yt}. Substituting $\left(\sqrt{m+1}\right)_t=-\left(u\sqrt{m+1}\right)_x$ from \eqref{SW} and $x_t=\hat u(y,t)$ from \eqref{x_t} into the equality
\[
\left(\sqrt{\hat m(y,t)+1}\right)_t=\left. \left(\left(\sqrt{ m(x,t)+1}\right)_x x_t(y,t)+\left(\sqrt{ m(x,t)+1}\right)_t\right)\right|_{x=x(y,t)}
\]
we get \eqref{SW_in_y-1}.
Now, substituting $x_y=\frac{1}{\sqrt{\hat m+1}}$ from \eqref{x_y} into the equality
$\hat u_y(y,t)=\left. \left(u_x(x,t)x_y(y,t)\right)\right|_{x=x(y,t)}$
we get \eqref{SW_in_y-2}.
Finally, \eqref{SW_in_y-3} follows from substituting \eqref{x_y} into
$
\hat v_y(y,t)=\left. \left(u_{xx}(x,t) x_y(y,t)\right)\right|_{x=x(y,t)}$.
\end{proof}

\begin{remark}
    The system of equations \eqref{SW_in_y} can be written in the conservation law form
    \begin{equation}\label{cons_y}
        \left(\frac{1}{\sqrt{\hat m+1}}\right)_t=\hat u_y.
    \end{equation}

\end{remark}

\begin{remark}
    Let
\[
r:=\ln\sqrt{\hat m+1}.
\]
Then system \eqref{SW_in_y} can be written as
\begin{subequations}\label{SW_q}
\begin{align}
r_t&=-\hat v, \label{SW_q-1}\\
\hat u_y&=\hat v e^{-r}, \label{SW_q-2}\\
\hat v_y&=e^{-r}-e^{r}. \label{SW_q-3}
\end{align}
\end{subequations}
Differentiating \eqref{SW_q-1} with respect to $y$ and using
\eqref{SW_q-3}, we obtain
\[
r_{ty}=e^{r}-e^{-r}=2\sinh r,
\]
that is, the sinh--Gordon equation
\begin{equation}\label{sinh-Gordon}
r_{ty}=2\sinh r.
\end{equation}
Moreover, from \eqref{SW_q-1} and \eqref{SW_q-2} we have
\[
\hat u_y=-e^{-r}r_t=(e^{-r})_t.
\]
\end{remark}

Similarly, the reverse  change of variable $(y,t)\mapsto(x,t)$ 
reduces \eqref{SW_in_y} to \eqref{SW} with 
$m(x,t):=\hat m(y(x,t),t)$ and $u(x,t):=\hat u(y(x,t),t)$.

\begin{proposition}\label{prop:hatu_u}
    Let $\hat u(y,t)$, $\hat v(y,t)$,  and $\hat m(y,t)$ with $\hat m(y,t)+1>0$ satisfy 
    \eqref{SW_in_y} and let $x(y,t)$ be 
such that \eqref{x-yt} holds. Define $u(x,t):=\hat u(y(x,t),t)$ and $m(x,t):=\hat m(y(x,t),t)$.
Then \eqref{SW} holds for $u(x,t)$ and $m(x,t)$.
\end{proposition}
\begin{proof}
Similarly to the above, \eqref{x-yt} implies \eqref{y-xt}.
Then 
\[
u_x=\hat u_y y_x = \frac{\hat v}{\sqrt{m+1}}\sqrt{m+1} = \hat v|_{y=y(x,t)},
\]
where we have used \eqref{SW_in_y-2} and \eqref{y_x}.
Further, differentiating this w.r.t. $x$ and using \eqref{SW_in_y-3} we get
\[
u_{xx}=\hat v_y y_x= \hat v_y \sqrt{m+1} = - m
\]
and thus $m(x,t)=-u_{xx}(x,t)$.

Finally, using \eqref{SW_in_y-3} and \eqref{y_t} we have
\begin{align*}
\left(\sqrt{m+1}\right)_t & = \left(\sqrt{\hat m+1}\right)_t + \left(\sqrt{\hat m+1}\right)_y 
y_t = -\hat v \sqrt{m+1} + \left(\sqrt{m+1}\right)_x x_y (-u\sqrt{m+1}) \\
& = -u_x \sqrt{m+1} -u \left(\sqrt{m+1}\right)_x = - \left(u\sqrt{m+1}\right)_x.
\end{align*}

\end{proof}

Now, let us reformulate the original Lax pair equations in the $(y,t)$ variables.

\begin{proposition}
The Lax pair \eqref{Lax-Q-form} in the variables $(y,t)$ takes the form 
\begin{subequations}\label{Lax_y}
\begin{align}\label{Lax_y_y}
    &\hat \Psi_{ y}+\ii k \sigma_3 \hat \Psi=\frac{1}{4}\frac{\hat m_y}{\hat m+1}\begin{pmatrix}
        0&1\\1&0
    \end{pmatrix}\hat \Psi,\\
   &\hat\Psi_{ t}=-\frac{1}{4\ii k}\begin{pmatrix}
       \sqrt{\hat m+1}+\frac{1}{\sqrt{\hat m+1}} &  -\sqrt{\hat m+1}+\frac{1}{\sqrt{\hat m+1}}\\\label{Lax_y_t}
        \sqrt{\hat m+1}-\frac{1}{\sqrt{\hat m+1}}& -\sqrt{\hat m+1}-\frac{1}{\sqrt{\hat m+1}}
   \end{pmatrix} \hat\Psi.
\end{align}
\end{subequations}
\end{proposition}

\begin{proof}
  Introducing $\hat\Psi(y,t) = \hat\Phi (x(y,t),t)$ and taking into account \eqref{x_y} and \eqref{x_t}, the Lax pair \eqref{Lax-Q-form} in the variables $(y,t)$ takes the form \eqref{Lax_y}.
\end{proof}

\subsection{The  Riemann--Hilbert formalism in $y,t$ variables}
In this subsection, we discuss how to arrive at a (local) solution of the SW equation in the 
$y,t$ variables starting from a RH problem parametrized by $y$ and $t$, suggested 
by the properties of Jost solutions.

Let $\Gamma$ denote an oriented contour in the complex plane, invariant under the symmetries $k\mapsto-k$ and $k\mapsto\bar k$. Suppose that ${\mathbb C}\setminus\Gamma = D_1\cup D_2$ so that $\Gamma$ is the counterclockwise boundary of $D_1$ (the clockwise boundary of $D_2$).

Consider the following \textbf{Riemann--Hilbert problem parametrized by $y$ and $t$}:
find a piecewise meromorphic ($ k\in {\mathbb C}\setminus\Gamma$), $2\times 2$-matrix-valued function $\hat M(y,t, k)$ satisfying the following conditions:
\begin{enumerate}[\textbullet]
\item
\emph{Jump} condition
\begin{equation}\label{jump-y_loc}
\hat M_-(y,t, k)=\hat M_+(y,t,k) \hat J(y,t, k),\qquad  k\in\Gamma,
\end{equation}
where 
$\hat J(y,t, k)=\eul^{-\hat p(y,t, k)\sigma_3}\hat J_0(k)\eul^{\hat p(y,t, k)\sigma_3}$ with $\hat p(y,t, k)=\ii k\left(y-\frac{t}{2k^2}\right)$ and some 
$\hat J_0( k)$ with $\det \hat J_0(k)\equiv 1$ and  such that
\begin{equation}\label{jump_at_inf}
    \hat J_0( k)=I+O\left(\frac{1}{ k}\right),\quad  k\to\infty.
\end{equation}
If $0\in\Gamma$, we additionally assume that
\begin{equation}\label{jump_at_ii}
  \hat  J_0( k)=I+O( k), \quad  k\to 0. 
\end{equation}

\item  \emph{Normalization} condition:
\begin{equation}\label{norm-m-hat_loc}
\hat M(y,t, k)=I+\ord\left(\frac{1}{ k}\right), \quad  k\to\infty.
\end{equation}

\item \emph{Residue} conditions:
\begin{align}\label{res_hatM_loc}
\Res_{  k_j}\hat M^{(1)}(y,t, k)&= c_j \hat M^{(2)}(y,t,  k_j) \eul^{2\hat p(y,t,  k_j)}, ~   k_j\in D_1, \\\label{res_hatM__loc}
\Res_{\bar k_j}\hat M^{(2)}(y,t, k)&=\bar c_j \hat M^{(1)}(y,t,\bar  k_j)\eul^{ -2\hat p(y,t,\bar  k_j)} , ~ \bar  k_j\in D_2
\end{align}
with some $\{  k_j\}$, $\{c_j\}$ such that the set $\{  k_j\}$ is symmetric under the mapping $ k\mapsto - \bar k$.
\end{enumerate}

\begin{proposition}
 Assume that $\hat M $ is a solution of RH problem \eqref{jump-y_loc}--\eqref{res_hatM__loc}. Then $\det \hat M=1$.
\end{proposition}
\begin{proof}
The fact that $\det \hat J_0(k)\equiv 1$ implies that $\det \hat M$ has neither a jump across $\Gamma$ nor singularities at the points $k_j$ and $\bar k_j$, and thus is an entire function. Therefore, \eqref{jump_at_inf}
together with Liouville's theorem implies $\det \hat M=1$.

\end{proof}

\begin{proposition}\label{prop:uniqness}
     If a solution of the RH problem \eqref{jump-y_loc}--\eqref{res_hatM__loc} exists, it is unique.
\end{proposition}

\begin{proof}
     The statement follows by considering two solutions $\hat M_1$ and $\hat M_2$ and applying Liouville's theorem to their ratio $\hat M_1 (\hat M_2)^{-1}$.
\end{proof}

Now  assume that the RH problem \eqref{jump-y_loc}--\eqref{res_hatM__loc} 
has a solution $\hat M(y,t,k)$ that satisfies the \emph{symmetries}
\begin{equation}
 \label{sym_1}
 \hat M(k)=\sigma_1\overline{\hat M(\bar k)}\sigma_1
 \end{equation}
and 
 \begin{equation}
 \label{sym_2}
\hat M(k)=\sigma_1\hat M(-k)\sigma_1
 \end{equation} 

 We also assume that its asymptotic
expansions can be differentiated termwise with respect to $y$ and $t$.
More precisely,
\begin{equation}\label{regularity_M_infty}
\hat M(y,t,k)
=
I+\frac{\hat M_1^\pm(y,t)}{k}
+R_\infty^\pm(y,t,k),
\qquad
k\to\infty,\quad k\in\mathbb C^\pm,
\end{equation}
and
\begin{equation}\label{regularity_M_zero}
\hat M(y,t,k)
=
B(y,t)
\left(
I\pm\ii k\Gamma_\pm(y,t)
+k^2G_\pm(y,t)
+R_0^\pm(y,t,k)
\right),
\qquad
k\to0,\quad k\in\mathbb C^\pm.
\end{equation}
The coefficient matrices $\hat M_1^\pm$, $B$, $\Gamma_\pm$, and
$G_\pm$ are assumed to be continuously differentiable with respect
to $y$ and $t$, and the remainders satisfy
\[
\mathcal D R_\infty^\pm(y,t,k)=o(k^{-1}),
\qquad
\mathcal D R_0^\pm(y,t,k)=o(k^2),
\qquad
\mathcal D\in\{1,\partial_y,\partial_t\}.
\]
The estimates are understood locally uniformly in $(y,t)$ and
uniformly for $k$ in closed subsectors of $\mathbb C^\pm$.

Evaluating $\hat M(y,t,k)$ at particular points of $\overline{\mathbb C}$, one can obtain a solution of the SW equation in the $y,t$ variables. We proceed as follows:
\begin{enumerate}[(a)]
\item 
Starting from $\hat M(y,t,k)$, define $2\times 2$-matrix-valued functions 
\[ \hat \Psi (y,t,k):= \hat M(y,t,k)\eul^{-\hat p(y,t,k)\sigma_3},\quad \hat p(y,t, k)=\ii k\left(y-\frac{t}{2k^2}\right).\] 
and show that $\hat\Psi(y,t,k)$ satisfies the system of differential equations:
\begin{equation}\label{Lax-hat-hat}
\begin{split}
\hat\Psi_y&=\doublehat{U}\hat\Psi, \\
\hat \Psi_t&=\doublehat{V}\hat\Psi,
\end{split}
\end{equation}
where $\doublehat{U}$ and $\doublehat{V}$ have the same (rational) dependence on $k$ as in \eqref{Lax_y_y} and \eqref{Lax_y_t}, with coefficients given in terms of $\hat M(y,t,k)$ evaluated at appropriate values of $k$.
\item
Show that the compatibility condition for \eqref{Lax-hat-hat}, i.e., the equality $\doublehat{U}_t - \doublehat{V}_y + [\doublehat{U},\doublehat{V}]=0$, reduces to 
\eqref{SW_in_y}.
\end{enumerate}

\begin{proposition}\label{Prop_Lax_y}
    Let $\hat M(y,t,k)$ be the solution of the RH problem \eqref{jump-y_loc}--\eqref{res_hatM__loc} that satisfies symmetries \eqref{sym_1} and \eqref{sym_2}. Define $\hat \Psi (y,t,k):= \hat M(y,t,k)\eul^{-\hat p(y,t,k)\sigma_3}$. 

    Then $\hat\Psi(y,t,k)$ satisfies the differential equation
\begin{equation}\label{Lax_y_Psi}
    \hat\Psi_y=\doublehat{U}\hat\Psi
\end{equation}
with 
\begin{equation}\label{hat_hat_U}
    \doublehat{U}(y,t,k)=-\ii k \sigma_3 + \hat\eta(y,t)\begin{pmatrix}
    0&1\\1&0
\end{pmatrix},
\end{equation}
where 
\begin{enumerate}[(i)]
    \item 
$\hat \eta(y,t)\in\mathbb{R}$ can be obtained from the large $k$ expansion of $\hat M(y,t,k)$: 
\begin{equation}\label{alpha}
\hat \eta(y,t)=-2\hat m_{12}^\infty(y,t),
\end{equation}
where
\begin{equation}\label{alpha-M}
\hat M(y,t,k)=I+\frac{\ii}{k}\begin{pmatrix}
    \hat m_1^\infty & \hat m_2^\infty\\
    -\hat m_2^\infty& -\hat m_1^\infty
\end{pmatrix}(y,t)+O\left(\frac{1}{k^2}\right),\qquad k\to\infty.
\end{equation}
\end{enumerate}

\end{proposition}
\begin{proof}
    First, notice that $\hat\Psi(y,t,k)$ satisfies the jump condition
\[
\hat\Psi_-(y,t,k)=\hat\Psi_+(y,t,k)J_0(k)
\]
with the jump matrix $\hat J_0(k)$ independent of $y$. Hence, $\hat\Psi_y(y,t,k)$ satisfies the same jump condition.

As for the residue conditions, we notice that the symmetry assumptions  
\eqref{sym_1} and \eqref{sym_2} imply that if $k_j$ is a pole of $\hat M^{(1)}$ then $-\bar k_j$ is also a pole of $\hat M^{(1)}$ with $c_{k_j}=-\overline{c_{-\bar k_j}}$. Analogously, if $\bar k_j$ is a pole of $\hat M^{(2)}$ then $-k_j$ is also a pole of $\hat M^{(2)}$ with $c_{\bar k_j}=\overline{c_{k_j}}$.

Since $\{c_j\}$ are independent of $y$, 
$\hat\Psi_y(y,t,k)$ satisfies the same residue conditions as $\hat\Psi(y,t,k)$ does:

\begin{align*}
\Res_{k_j}\hat \Psi^{(1)}&= c_j \hat \Psi^{(2)}(k_j) , ~ k_j\in \mathbb{C}^+, \\
\Res_{\bar k_j}\hat \Psi^{(2)}&=\bar{ c}_j \hat \Psi^{(1)}(\bar k_j) , ~ \bar k_j\in \mathbb{C}^-.
\end{align*}
with constants $c_j$ independent of $y$.

Consequently,  $\hat \Psi_y \hat \Psi^{-1}=\hat M_y \hat M^{-1}- \ii k \hat M \sigma_3 \hat M ^{-1}$  has neither jump nor singularities at $k_j$, and thus it is a meromorphic function, with a possible singularity at $k=\infty$.

Let us analyze the behavior of $\hat \Psi_y \hat \Psi^{-1}$ as $k\to\infty$.

\begin{enumerate}[(i)]
    \item As $k\to\infty$, we have

    \begin{equation*}
   \hat M = I+\frac{\hat M_\infty}{k}+o\left(\frac{1}{k}\right),
   \end{equation*}
   where, due to the symmetries \eqref{sym_1} and \eqref{sym_2},
    \begin{equation*}
   \hat M = I+\frac{\ii}{k}\begin{pmatrix}
    \hat m_{11}^\infty & \hat m_{12}^\infty\\
    \hat m_{21}^\infty& \hat m_{22}^\infty
\end{pmatrix}+o(\frac{1}{k}), \quad k\in\mathbb{C}^+
   \end{equation*}
with $\hat m_{+ij}^\infty \in\mathbb{R}$. Moreover, the symmetry \eqref{sym_1} implies
    \begin{equation*}
   \hat M = I-\frac{\ii}{k}\begin{pmatrix}
    \hat m_{22}^\infty & \hat m_{21}^\infty\\
    \hat m_{12}^\infty& \hat m_{11}^\infty
\end{pmatrix}+o(\frac{1}{k}), \quad k\in\mathbb{C}^-.
\end{equation*}
Consequently,  $
\hat M_y= \frac{\hat M_{\infty y}}{k}+o\left(\frac{1}{k}\right)$  and thus 
 $\hat M_y \hat M^{-1}=O(\frac{1}{k})$, which leads to 
the following expansion for $\hat\Psi_y\hat\Psi^{-1}$:

\begin{equation*}\label{Psi_y_Psi_inf_C+}
  \hat \Psi_y \hat \Psi^{-1}=-\ii k \sigma_3 + \begin{pmatrix}
    \hat m_{11}^\infty+\hat m_{22}^\infty&-2\hat m_{12}^\infty\\2\hat m_{21}^\infty&-\hat m_{11}^\infty-\hat m_{22}^\infty
\end{pmatrix}+O(\frac{1}{k}),  \quad k\to \infty, \quad k\in\mathbb{C}^+,
\end{equation*}
\begin{equation*}\label{Psi_y_Psi_inf_C-}
  \hat \Psi_y \hat \Psi^{-1}=-\ii k \sigma_3 - \begin{pmatrix}
    \hat m_{11}^\infty+\hat m_{22}^\infty&-2\hat m_{21}^\infty\\2\hat m_{12}^\infty&-\hat m_{11}^\infty-\hat m_{22}^\infty
\end{pmatrix}+O(\frac{1}{k}),  \quad k\to \infty, \quad k\in\mathbb{C}^-.
\end{equation*}
Since $\hat \Psi_y \hat \Psi^{-1}$ is meromorphic in $\mathbb{C}$ these expansions should coincide and thus, 
$\hat m_{11}^\infty=-\hat m_{22}^\infty$ and $\hat m_{12}^\infty=-\hat m_{21}^\infty$.

Consequently,

\begin{equation}\label{Psi_y_Psi_inf}
  \hat \Psi_y \hat \Psi^{-1}=-\ii k \sigma_3 - 2\hat m_{12}^\infty\begin{pmatrix}
    0&1\\1&0
\end{pmatrix}+O(\frac{1}{k}),  \quad k\to \infty.
\end{equation}

By Liouville's theorem, \eqref{Psi_y_Psi_inf} implies \eqref{Lax_y_Psi}--\eqref{alpha-M}.

\end{enumerate}

\end{proof}

\begin{proposition}\label{Prop_Lax_t}
$\hat\Psi(y,t,k)$ defined as in Proposition \ref{Prop_Lax_y} satisfies the differential equation
\begin{equation}\label{Lax_t_Psi}
    \hat\Psi_t=\doublehat{V}\hat\Psi
\end{equation}
with 
\begin{equation}\label{hat_hat_V}
    \doublehat{V}=-\frac{1}{2\ii k} \begin{pmatrix}
        \hat \alpha^2+\hat \beta^2&-2\hat \alpha \hat \beta \\
        2\hat \alpha \hat \beta&-\hat \alpha^2-\hat \beta^2
    \end{pmatrix},
\end{equation}
where 
\begin{enumerate}[(i)]
    \item 
$\hat \alpha(y,t)\in\mathbb{R}$ and $\hat \beta(y,t)\in\mathbb{R}$ can be obtained  from the expansion of $\hat M(y,t,k)$ as $k\to 0$: 
\begin{equation}\label{alpha_beta-M}
    \hat M(y,t,k)=\begin{pmatrix}
    \hat \alpha &\hat \beta\\
   \hat \beta&\hat \alpha
\end{pmatrix}+
O(k),\quad k\to 0.
\end{equation}
\end{enumerate}

\end{proposition}

\begin{proof}
    Similarly to Proposition \ref{Prop_Lax_y}, we notice that $\hat\Psi_t \hat\Psi^{-1}=\hat M_t\hat M^{-1}+\frac{\ii}{2k}\hat M\sigma_3\hat M^{-1}$ has neither jump nor poles at $ k_j$, and thus it is a meromorphic function, with possible singularities at $k=\infty$ and $k=0$, the latter being due to the singularity of $\hat p_t$ at $k=0$:

Evaluating $\hat\Psi_t\hat\Psi^{-1}$ near these points, we have the following.
\begin{enumerate}[(i)]

\item As $k\to\infty$, we have $\hat M_t\hat M^{-1}=O(\frac{1}{k})$, $p_t(k)=O(\frac{1}{k})$ and thus
\begin{equation}\label{psi-inf-t}
\hat\Psi_t\hat\Psi^{-1}(k)=\ord(k^{-1}),\qquad k\to\infty.
\end{equation}

\item As $k\to 0$, the asymptotic behavior $J_0(k)=I+O(k)$ as $k\to0$ together with the symmetries \eqref{sym_1} and  \eqref{sym_2},  leads to the following expansions for $\hat M$ near $k=0$ in $\mathbb{C}^+$ and $\mathbb{C}^-$:

\begin{align}\label{hatM_at_0_pl}
    &\hat M =B\left(I+\ii k\Gamma_++k^2 G_+ +o(k^2)\right),\quad k\to 0,\quad k\in\mathbb{C}^+,\\\label{hatM_at_0_min}
    &\hat M =B\left(I-\ii k\Gamma_-+k^2 G_-+o(k^2)\right),\quad k\to 0,\quad k\in\mathbb{C}^-,
\end{align}
where
\begin{equation*}
    B=\begin{pmatrix}
        \hat \alpha&\hat \beta\\
        \hat \beta&\hat \alpha
    \end{pmatrix},\quad 
    \Gamma_+=\begin{pmatrix}
        \hat f_{1}&\hat f_{2}\\
        \hat f_{3}&\hat f_{4}
    \end{pmatrix},\quad 
    \Gamma_-=\begin{pmatrix}
    \hat f_{4}&\hat f_{3}\\
    \hat f_{2}&\hat f_{1}
    \end{pmatrix}, \quad 
    G_+=\begin{pmatrix}
        \hat g_{1}&\hat g_{2}\\
        \hat g_{3}&\hat g_{4}
    \end{pmatrix},\quad 
    G_-=\begin{pmatrix}
    \hat g_{4}&\hat g_{3}\\
    \hat g_{2}&\hat g_{1}
    \end{pmatrix}.
\end{equation*}
with $\hat \alpha\in\mathbb{R}$, $\hat \beta\in\mathbb{R}$, $\hat f_{j}$, and $\hat g_{j}\in\mathbb{R}$.

Therefore, we have
\begin{align}\label{psi_t_psi_0}
  &  \hat \Psi_t\hat \Psi^{-1}=\frac{\ii}{2 k}B\sigma_3B^{-1}+\left(B_tB^{-1}-\frac{1}{2}B\left(\pm \sigma_3\Gamma_{\pm}^{\sharp}\pm\Gamma_{\pm}\sigma_3\right)B^{-1}\right)+\\
    &\ii k\left(\pm B_t\Gamma_{\pm}\pm B \Gamma_{\pm t}\pm B_t \Gamma_{\pm}^{\sharp}+\frac{1}{2}B\left[\sigma_3G_{\pm}^{\sharp}+G_{\pm}\sigma_3-\Gamma_{\pm}\sigma_3\Gamma_{\pm}^{\sharp}\right]\right)B^{-1}+O(k^2),\quad k\to 0, \quad k\in\mathbb{C}^\pm,
\end{align}
where, for a $2\times 2$ matrix $A=\begin{pmatrix}
    a&b\\c&d
\end{pmatrix}$, we denote its adjugate by $A^\sharp=\begin{pmatrix}
    d&-b\\-c&a
\end{pmatrix}$.

In particular,
\begin{equation}\label{Psi_t_Psi_0}
  \hat \Psi_t \hat \Psi^{-1}=-\frac{1}{2\ii k} \begin{pmatrix}
\hat \alpha^2+ \hat \beta^2&-2 \hat \alpha \hat \beta\\2 \hat \alpha \hat \beta&-\hat \alpha^2-\hat \beta^2
\end{pmatrix}+O(1),  \quad k\to 0.
\end{equation}

Now, \eqref{psi-inf-t} and \eqref{Psi_t_Psi_0} together with  Liouville's theorem imply \eqref{Lax_t_Psi}--\eqref{alpha_beta-M}.

\end{enumerate}

\end{proof}

The next Proposition follows directly from the compatibility of \eqref{Lax_y_Psi} and \eqref{Lax_t_Psi}
\begin{equation*}\label{compat}
\doublehat{U}_t - \doublehat{V}_y + [\doublehat{U},\doublehat{V}]=0.
\end{equation*}

\begin{proposition}
Let $\hat\alpha(y,t)$, $\hat\beta(y,t)$, $\hat \eta(y,t)$ be the functions determined in terms of $\hat M(y,t,k)$ as in Propositions \ref{Prop_Lax_y} and \ref{Prop_Lax_t}. Then they satisfy the following equations:
\begin{subequations}\label{rel}
\begin{align}\label{rel-a}
&\hat\eta_{ t}= 2\hat\alpha\hat\beta;\\
\label{rel-b}
&(\hat\alpha^2+\hat\beta^2)_y=4\hat\eta\hat\alpha\hat\beta;\\
\label{rel-c}
&(\hat\alpha\hat\beta)_y=\hat\eta(\hat\alpha^2+\hat\beta^2).
\end{align}
\end{subequations}
\end{proposition}

We now derive several additional relations among the coefficients $\hat \alpha$, $\hat \beta$, $\hat f_1$, $\hat f_2$, $\hat g_1$, $\hat g_2$, and $\hat g_4$.

 \begin{enumerate}

\item Determinant relation implies \begin{equation}
    \label{rel_1}
    \hat \alpha^2-\hat\beta^2=1
\end{equation}
    \item Since $\hat \Psi_t \hat \Psi^{-1}$ is meromorphic in $\mathbb{C}$, its expansions near $k=0$ in $\mathbb{C}^+$ and $\mathbb{C}^-$ should coincide. Consequently, \[ \sigma_3\Gamma_{+}^{\sharp}+\Gamma_{+}\sigma_3=- \sigma_3\Gamma_{-}^{\sharp}-\Gamma_{-}\sigma_3,\] which implies
   $ \hat f_{4}=-\hat f_{1}$,  $ \hat f_{3}=-\hat f_{2}$, and thus
    \begin{equation}
        \label{Gamma_pm}
    \Gamma:=\Gamma_+=-\Gamma_-,\quad \Gamma^{\sharp}=-\Gamma.
    \end{equation}
    Moreover, \eqref{psi_t_psi_0} together with \eqref{Gamma_pm} implies
    \[
    B_tB^{-1}-\frac{1}{2}B\left( \sigma_3\Gamma_{}^{\sharp}+\Gamma_{}\sigma_3\right)B^{-1}=0,
    \]
    and thus
    \begin{align}\label{rel_2}
        \hat \alpha_t=-\hat\beta\hat f_2,\\\label{rel_3}
        \hat \beta_t=-\hat\alpha\hat f_2.
    \end{align}

 \item  Similarly,
 \[ \sigma_3G_{+}^{\sharp}+G_{+}\sigma_3= \sigma_3G_{-}^{\sharp}+G_{-}\sigma_3,\] which implies
   $ \hat g_{2}=\hat g_{3}$.
 In particular, together with
\eqref{Gamma_pm}, this yields
\begin{subequations}\label{hat_M_near_0}
\begin{equation}
    \hat M^{(xt)}
=
\begin{pmatrix}
\hat\alpha&\hat\beta\\
\hat\beta&\hat\alpha
\end{pmatrix}
\left(
I+\ii k\Gamma+k^2G+o(k^2)
\right),\qquad k\to 0, \quad k\in\mathbb{C^+},
\end{equation}
where
\begin{equation}
    \Gamma=
\begin{pmatrix}
\hat f_1&\hat f_2\\
-\hat f_2&-\hat f_1
\end{pmatrix},
\qquad
G=
\begin{pmatrix}
\hat g_1&\hat g_2\\
\hat g_2&\hat g_4
\end{pmatrix}.
\end{equation}
\end{subequations}
Moreover, from \eqref{psi_t_psi_0}  with \eqref{Gamma_pm} we get, comparing terms of order $k$, the equations
  \begin{align}\label{rel_4}
     \hat f_{1t}+\frac{1}{2}\left(\hat g_1+\hat g_4+\hat f_1^2+\hat f_2^2\right)=0,\\\label{rel_5}
        \hat f_{2t}-\hat g_2+\hat f_1\hat f_2=0.
    \end{align}

\item Substituting \eqref{hatM_at_0_pl} and \eqref{hatM_at_0_min} into $ \Psi_y\Psi^{-1}=\hat M_y \hat M^{-1}-\ii k \hat M \sigma_3 \hat M ^{-1}$, we arrive at
\begin{align*}
    &\Psi_y\Psi^{-1}=B_yB^{-1}+\ii k\left(\pm B_y \Gamma_{\pm}^\sharp B^{-1}\pm B_y\Gamma_{\pm}B^{-1}\pm B\Gamma_{\pm y}B^{-1}-B\sigma_3B^{-1}\right)+\\
    &+k^2\left(B_y\left[G_{\pm}^\sharp-\Gamma_{\pm}\Gamma_{\pm}^\sharp+G_{\pm}\right]-B\left[\Gamma_{\pm y}\Gamma_{\pm }^\sharp-G_{\pm  y}\mp \sigma_3\Gamma_{\pm}^\sharp\mp\Gamma_{\pm}\sigma_3\right]\right)B^{-1}+o(k^2),\quad k\to 0, \quad k\in\mathbb{C}^\pm.
\end{align*}
Comparing this with \eqref{hat_hat_U} we obtain
\begin{align*}
 &B_yB^{-1}=\hat \eta\begin{pmatrix}
     0&1\\
     1&0
 \end{pmatrix},\\
 & \Gamma_y-\sigma_3=-B^{-1}\sigma_3B,\\
 &B_y\left[G_{\pm}^\sharp-\Gamma\Gamma^\sharp+G_{\pm}\right]-B\left[\Gamma_{y}\Gamma^\sharp-G_{\pm  y}- \sigma_3\Gamma^\sharp-\Gamma\sigma_3\right]=0.
\end{align*}
Therefore,
    \begin{align}\label{rel_6}
        &\hat \alpha\hat\beta_y-\hat\beta\hat \alpha_y=\hat\eta,\\\label{rel_7}
       &\hat f_{1y}=-2\hat\beta^2,\\\label{rel_8}
       &\hat f_{2y}=-2\hat\beta\hat\alpha,\\\label{rel_9}
      & \hat g_{1y}=\hat g_{4y},\\\label{rel_10}
      & \hat g_{1y}=\frac{1}{2}\left(\hat f_2^2-\hat f_1^2\right)_y=2\hat\beta^2\hat f_1-2\hat\beta\hat\alpha\hat f_2,\\\label{rel_11}
    &  \hat\eta \left(\hat g_1+\hat g_4+\hat f_1^2-\hat f_2^2\right)+\hat g_{2y}-2\hat f_2+\hat f_2\hat f_{1y}-\hat  f_1\hat f_{2y}=0.
    \end{align}
      
\item The determinant relation implies (comparing the term of order $k^2$)
\begin{equation}
    \label{rel_12}
    \hat g_1+\hat g_4+\hat f_1^2-\hat f_2^2=0.
\end{equation}
\end{enumerate}

The preceding relations allow us to recover the SW equation directly from the coefficients in the local expansion of the solution of the RH problem \eqref{jump-y_loc}--\eqref{res_hatM__loc}.

\begin{theorem}
    Introducing \begin{equation}\label{hmbbgg}
  \hat m(y,t)=(\hat \alpha(y,t)+\hat{\beta}(y,t))^4-1,\qquad \hat v(y,t)= 2\hat f_2 (y,t) ,\qquad \hat u=\hat g_2-\hat f_1\hat f_2+\hat f_2^2,
    \end{equation} Equations \eqref{rel}--\eqref{rel_12} reduce to \eqref{SW_in_y}.
\end{theorem}

\begin{proof}

Since, $\hat m(y,t)=(\hat \alpha(y,t)+\hat{\beta}(y,t))^4-1$, it follows that
\[
\sqrt{\hat m+1}=(\hat \alpha+\hat{\beta})^2.
\]
Differentiating this equation with respect to $t$ and using \eqref{rel_2} and \eqref{rel_3}, we obtain
\[
\left(\sqrt{\hat m+1}\right)_t=-2\hat f_2(\hat \alpha+\hat{\beta})^2,
\]
which proves \eqref{SW_in_y-1}.

Next, differentiating $\hat v$ with respect to $y$ and using  \eqref{rel_8} and \eqref{rel_1}, we arrive at
\[
-\hat v_y\sqrt{\hat m+1}=\hat m,
\]
which is exactly \eqref{SW_in_y-3}.

Finally, differentiating $\hat u$ with respect to $y$ yields
\[
\hat u_y=\hat g_{2y}-\hat f_{1y}\hat f_2+\hat f_{2y}\hat f_1+2\hat f_{2y}\hat f_2.
\]
Moreover, combining \eqref{rel_11} and \eqref{rel_12}, we obtain
\[
\hat g_{2y}=2\hat f_2-\hat f_2\hat f_{1y}+\hat  f_1\hat f_{2y}.
\]
Substituting this expression into the formula for $\hat u_y $ and then using \eqref{rel_7} and \eqref{rel_8}, we arrive at
\[
\hat u_y=2\hat f_2(\hat\alpha-\hat\beta)^2,
\]
which proves \eqref{SW_in_y-2}.
    
\end{proof}

\begin{remark}
  Notice that \eqref{hmbbgg} implies $\hat m(y,t)+1=(\hat \alpha(y,t)+\hat{\beta}(y,t))^4>0$. Indeed, since $(\hat \alpha -\hat \beta)(\hat \alpha +\hat \beta)=1$, it follows that   $\hat \alpha +\hat \beta\neq 0$ and thus $\hat m +1 >0$.
  
  Therefore, the positivity condition required in the direct spectral analysis is guaranteed automatically, by the construction based on the solution of the Riemann--Hilbert problem. In turn, it guarantees the analytic properties of the associated eigenfunctions.
\end{remark}

Now, introduce the reverse change of variable $(y,t)\mapsto (x,t)$ by
\begin{equation}\label{x_y_}  
x(y,t)=
\int_{\eta(t)}^y
\tfrac{\dd \xi}
{\left(\hat\alpha(\xi,t)+\hat\beta(\xi,t)\right)^2}
= 
\int_{\eta(t)}^y
\tfrac{\dd \xi}{\sqrt{\hat m(\xi,t)+1}},
\end{equation}
so that $x(\eta(t),t)=0$ and $x_y(y,t)=\tfrac{1}{\sqrt{\hat m(\xi,t)+1}}$. Furthermore, differentiating \eqref{x_y_} with respect to 
$t$ and using \eqref{cons_y}, we obtain \eqref{x_t}.
Then $u(x,t)=\hat u(y(x,t),t)$ and $m(x,t)=\hat m(y(x,t),t)$  satisfy \eqref{SW} by Proposition \ref{prop:hatu_u}.

\begin{remark}
    In contrast to other Camassa--Holm-type equations, such as the Short-Pulse equation \cite{BSL17} and the modified Camassa--Holm equation \cite{BKS20}, the change of spatial variable used here does not introduce additional singularities in the physical variables. Indeed, the map $(y,t)\mapsto (x,t)$ remains a smooth diffeomorphism for each $t$. Consequently, a solution that is smooth in $(y,t)$ remains smooth after reconstruction in the original variables $(x,t)$. For the Short-Pulse and modified Camassa--Holm equations, in contrast, the corresponding transformation may lose invertibility, causing a smooth  solution in $(y,t)$ to produce wave breaking or a nonsmooth or multivalued solution in physical coordinates $(x,t)$.
\end{remark}

Now let us discuss sufficient conditions on the contour $\Gamma$,  jump matrix $\hat J_0(k)$, and residue conditions that provide the symmetries \eqref{sym_1} and \eqref{sym_2}.

\begin{proposition}\label{prop:sym}
    Let the contour $\Gamma$ be invariant w.r.t. $k\mapsto-k$ and $k\mapsto-\bar k$, the jump matrix $\hat J_0$ satisfy the symmetries
       \begin{equation*}
 \label{sym_J_1}
 \hat J_0(k)=\overline{\hat J_0(-\bar k)}
 \end{equation*}
and 
 \begin{equation*}
 \label{sym_J_2}
\hat J_0^{{-1}}(k)=\sigma_1 \hat J_0(-k)\sigma_1,
 \end{equation*} 
 and the parameters of the residue conditions satisfy $c_{k_j}=-\overline{c_{-\bar k_j}}$ and $c_{k_j}=\overline{c_{\bar k_j}}$.

    Then the solution $\hat M$ of the RH problem  \eqref{jump-y_loc}--\eqref{res_hatM__loc}   satisfies symmetries \eqref{sym_1} and \eqref{sym_2}
\end{proposition}

\begin{proof} The claim follows from the uniqueness of the solution to the RH problem
\eqref{jump-y_loc}--\eqref{res_hatM__loc}. Indeed, one verifies that
$\overline{\hat M(-\bar k)}$
and
$\sigma_1
\hat M(-k)
\sigma_1$
satisfy the same jump conditions, residue conditions, and normalization as
\(\hat M(k)\). By uniqueness, it follows that
\[
\hat M(k)=\overline{\hat M(-\bar k)}
=
\sigma_1
\hat M(-k)
\sigma_1,
\]
which proves the stated symmetries.
\end{proof}

\section{The main Riemann--Hilbert problem}\label{sec:6}
Throughout this Section, we will assume that $\epsilon>0$ is sufficiently small so that all eventual
zeros of $\tilde a(k)$, $ a(k)$, $\tilde d(k)$, and $ d(k)$ appearing in the denominators below lie outside the disks $\{|k|\leq\epsilon\}$.

\subsection[The Riemann--Hilbert problem formalism in the case $u(0,t)\le 0$]{\texorpdfstring{${u}(0,t)\le 0$}{u-negative}}

\subsubsection{The pre-Riemann--Hilbert problem} 

Assume that we are given a solution $u(x,t)$ of the SW equation in the domain $(x,t)\in(0,\infty)\times (0,T)$ such that 
\begin{enumerate}[(i)]
    \item  $u(\cdot, t)\in  H^{4}(0,\infty)$ and $u_{xx}(\cdot, t)\in  H^{2,1}(0,\infty)$ for all $t\in[0,T]$,
    \item $m(x,0)+1>0$,
    \item $u(0,t)\le 0$,
\end{enumerate}
and consider the following piecewise meromorphic matrix-valued function $M(x,t, k)$ (depending on $(x,t)$ as parameters, for all $x\geq0$ and $t\in[0,T]$) defined in the domains separated by contour depicted 
in Figure \ref{fig:contour_RH_x}:

\begin{equation}\label{M_(xt)}
M^{(xt)}(x,t,k)=\begin{cases}

\left( \frac{\Phi_{\infty 2}^{(1)}(x,t,k)}{a(k)},\Phi_{\infty 3}^{(2)}(x,t,k)\right),\quad k\in\mathbb{C}^+\cap\{|k| >\epsilon\},\\

\left( \frac{\Psi_{0 2}^{(1)}(x,t,k)}{\tilde a(k)},\Psi_{0 3 }^{(2)}(x,t,k)\right)\eul^{-\ii k\nu(0)\sigma_3},\quad k\in\mathbb{C}^+\cap\{|k| <\epsilon\},\\

\left( \Phi_{\infty 3}^{(1)}(x,t,k),\frac{\Phi_{\infty 2}^{(2)}(x,t,k)}{a^*( k)}\right),\quad k\in\mathbb{C}^-\cap\{|k| >\epsilon\},\\

\left( \Psi_{0 3}^{(1)}(x,t,k),\frac{\Psi_{0 2}^{(2)}(x,t,k)}{\tilde a^*( k)}\right)\eul^{-\ii k\nu(0)\sigma_3},\quad k\in\mathbb{C}^-\cap\{|k| <\epsilon\},

\end{cases}
\end{equation}
where the functions $\Psi_{0j}$
are defined by \eqref{psi_0}.

Then the function $M^{(xt)}(x,t,k)$ has the following properties:

\begin{enumerate}
    \item Jump relation across $\mathbb{R}\cup\{|k|=\epsilon\}$
    \begin{subequations}
        \label{jump_M_(xt)}
        \begin{equation}
           M_-^{(xt)}(x,t,k)=M_+^{(xt)}(x,t,k)J^{(xt)}(x,t,k),\quad k\in \mathbb{R}\cup\{|k|=\epsilon\} 
        \end{equation}
        where
         \begin{equation} \label{jump_M_(xt)_matr}
          J^{(xt)}(x,t,k)=\eul^{-p(x,t,k)\sigma_3} J^{(xt)}_0(k)\eul^{p(x,t,k)\sigma_3}
        \end{equation} 
        with
\begin{equation}\label{jump_M_(xt)_matr_0}
   J^{(xt)}_0(k)=\begin{cases}
       \begin{pmatrix}
       1&-r^*(k)\\r(k)&1-r(k)r^*(k)
   \end{pmatrix},\quad k\in\mathbb{R}\cap\{|k| >\epsilon\},\\
   \begin{pmatrix}
       1-\tilde r(k)\tilde r^*(k)&\tilde r^*(k)\eul^{2\ii k\nu(0)}\\-\tilde r(k)\eul^{-2\ii k\nu(0)}&1
   \end{pmatrix},\quad k\in\mathbb{R}\cap\{|k| <\epsilon\},\\
   \begin{pmatrix}
       1&0\\
       \frac{\kappa_2^0\eul^{-\ii k\nu(0)}}{a(k)\tilde a(k)}&1
   \end{pmatrix},\quad k\in\mathbb{C}^+\cap\{|k| =\epsilon\},\\
   \begin{pmatrix}
       1&-\frac{\kappa_2^0\eul^{\ii k\nu(0)}}{a^*(k)\tilde a^*(k)}\\
       0&1
   \end{pmatrix},\quad k\in\mathbb{C}^-\cap\{|k| =\epsilon\},
   \end{cases} 
\end{equation}
\end{subequations}
and $r(k)=\frac{b^*(k)}{a(k)}$, $\tilde r(k)=\frac{\tilde b^*(k)}{\tilde a(k)}$.  
Moreover,  \eqref{s_via_til_s} implies
\begin{equation}
    \label{r_tilr_con}
    \tilde r(k)=\eul^{2\ii k\nu(0)}
\left(r(k)-\frac{\kappa_2^0 \eul^{-\ii k\nu(0)}}{a(k)\tilde a(k)}\right), \quad k\in\mathbb{R}.
\end{equation}

\begin{remark}\label{rem:x-RH_reduced}
       In the  case $Q(0,0)=I$ (with $\kappa_1^0=1$ and $\kappa_2^0=0$),
    the jump matrix $J_0^{(xt)}(k)$ simplifies to
    \[
    J_0^{(xt)}(k)=
       \begin{pmatrix}
       1&-r^*(k)\\r(k)&1-r(k)r^*(k)
   \end{pmatrix},\quad k\in\mathbb{R}.
    \]   
\end{remark}

\item Behavior at $\infty$:
\begin{subequations}\label{inf_M_(xt)}
\begin{equation}\label{inf_M_exp(xt)}
     M^{(xt)}(x,t,k)=I+\frac{1}{8\ii k}M^{\infty}(x,t)+o\left(\frac{1}{k}\right),\quad k\to\infty,
\end{equation}
with
\begin{equation}\label{inf_M_inf(xt)}
M^{\infty}(x,t)=\begin{pmatrix}
    *&\frac{m_{x}}{(m+1)^{\frac{3}{2}}}\\
    -\frac{m_{x}}{(m+1)^{\frac{3}{2}}}&*\end{pmatrix}.
\end{equation}
\end{subequations}

\begin{remark}
    If, in addition,  $m(\cdot,t)\in H^{3,1}(0,\infty)$,  then the integration-by-parts argument can be iterated once more, and the remainder improves to $O(k^{-2})$. 
\end{remark}

\item $\det M^{(xt)}(x,t,k)\equiv1$

\item Symmetry properties:
\begin{equation}\label{sym-M_(xt)}
M^{(xt)}( k)=\sigma_1\overline{M^{(xt)}(\bar k)}\sigma_1,\qquad M^{(xt)}(k)=\sigma_1M^{(xt)}(-k)\sigma_1
\end{equation}

\item Residue properties. Let $\{k_j\}$ be zeros of $a(k)$ in $\mathbb{C}^+$. Assume that they are all simple. Then

\begin{subequations}\label{res-M_(xt)}
\begin{align}\label{res-M+_(xt)}
\Res_{k_j}M^{(xt)(1)}(x,t,k)&=\frac{e^{2p(x,t,k_j)}}{\dot a(k_j)b(k_j)}M^{(xt)(2)}(x,t,k_j),\\
\label{res-M-_(xt)}
\Res_{\bar k_j}M^{(xt)(2)}(x,t,k)&=\frac{e^{-2p(x,t,\bar k_j)}}{\dot a^*(\bar k_j)b^*(\bar k_j)}M^{(xt)(1)}(x,t,\bar k_j).
\end{align}
\end{subequations}

Notice that $\dot a^*(\bar k_j)b^*(\bar k_j)=\overline{\dot a( k_j)b( k_j)}$

\item Behavior at $0$:
\begin{subequations}
    \label{i_beh-M_(xt)}
\begin{equation}
M^{(xt)}(x,t, k)=Q(x,t)\left(I+\ii k M_1^{(xt)}(x,t)+k^2 M_2^{(xt)}(x,t)+o(k^2)\right) \eul^{\ii k\left( \int_0^{x} \left(\sqrt{ m+1}-1\right)(\xi,t)\dd\xi+\eta(t)-\nu(0)\right)\sigma_3}, 
\end{equation}
where
\begin{equation}
    M_1^{(xt)}=\begin{pmatrix}
    \frac{u_x(x,t)}{2}&\frac{u_x(x,t)}{2}\\
    -\frac{u_x(x,t)}{2}&-\frac{u_x(x,t)}{2}
\end{pmatrix},
\end{equation}

\begin{equation}
    M_2^{(xt)}=\begin{pmatrix}
    *&u(x,t)\\
    u(x,t)&*
\end{pmatrix}.
\end{equation}
\end{subequations}
    
\end{enumerate}

\subsubsection{The main Riemann--Hilbert problem parametrized by $y$ and $t$}

The construction of the jump matrix involves $p(x,t, k)$ (see \eqref{p}), which in turn involves $ u(x,t)$. This suggests the introduction of a new variable
\begin{equation}\label{y_}
y(x,t)=\int_0^{x} \left(\sqrt{ m+1}\right)(\xi,t)\dd\xi-\int_0^{t} \left(u\sqrt{ m+1}\right)(0,\tau)\dd\tau, 
\end{equation}
in terms of which the jump matrix and the residue conditions can be written
explicitly. Notice that $\frac{\partial y}{\partial x}=\sqrt{ m(x,t)+1}$ and $\frac{\partial y}{\partial t}=-\left(u\sqrt{ m+1}\right)(x,t)$.

\textbf{The Riemann--Hilbert problem RH$^{(xt)}$:} Given $a( k)$, $b( k)$, and the set $\{k_j\}_1^J\subset \mathbb{C}^+$, find a piecewise meromorphic $2\times 2$  matrix-valued function $\hat M^{(xt)}(y,t, k)$ that satisfies the following conditions:

\begin{enumerate}
    \item Jump relation across $\mathbb{R}\cup\{|k|=\epsilon\}$
    \begin{subequations}
        \label{jump_hatM_(xt)}
        \begin{equation}
           \hat M_-^{(xt)}(y,t, k)=\hat M_+^{(xt)}(y,t, k)\hat J^{(xt)}(y,t, k),\quad k\in\mathbb{R}\cup\{|k|=\epsilon\} 
        \end{equation}
        where
         \begin{equation}
         \hat J^{(xt)}(y,t, k)=\eul^{-\hat p(y,t, k)\sigma_3} J^{(xt)}_0( k)\eul^{\hat p(y,t, k)\sigma_3}
        \end{equation} 
        with $J^{(xt)}_0( k)$ defined in \eqref{jump_M_(xt)_matr_0} and $\hat p(y,t, k)=\ii k\left(y-\frac{t}{2k^2}\right)$.

    \end{subequations}

\item Behavior at $\infty$:
\begin{equation}\label{inf_hatM_(xt)}
    \hat M^{(xt)}(y,t, k)=I+O(\frac{1}{ k}),\quad k\to\infty.
\end{equation}

\item Residue conditions:
\begin{subequations}\label{res-hatM_(xt)_2}
\begin{align}\label{res-hatM+_(xt)_2}
\Res_{k_j}\hat M^{(xt)(1)}(y,t, k)&=\frac{\eul^{2\hat p(y,t,k_j)}}{\dot a(k_j)b(k_j)}\hat M^{(xt)(2)}(y,t,k_j),\\
\label{res-hatM-_(xt)_2}
\Res_{\bar k_j}\hat M^{(xt)(2)}(y,t, k)&=\frac{\eul^{-2\hat p(y,t,\bar k_j)}}{\dot{ a^*}(\bar k_j)b^*(\bar k_j)}\hat M^{(xt)(1)}(y,t,\bar k_j),\quad j=1,\dots,J.
\end{align}
\end{subequations}

\end{enumerate}

\begin{proposition} If the solution $\hat M^{(xt)}(y,t, k)$ of the RH problem \eqref{jump_hatM_(xt)}--\eqref{res-hatM_(xt)_2} exists, then
    \begin{enumerate} 
        \item $\det \hat M^{(xt)}(y,t, k)\equiv1$;
    
        \item  it is unique;

        \item it satisfies symmetries

        \begin{equation*}\label{sym-hatM_(xt)}
\hat M^{(xt)}( k)=\sigma_1\overline{\hat M^{(xt)}(\bar k)}\sigma_1,\qquad \hat M^{(xt)}(k)=\sigma_1\hat M^{(xt)}(-k)\sigma_1
\end{equation*}
 
    \end{enumerate}
\end{proposition}

\begin{remark}
The question of whether the RH problem
\eqref{jump_hatM_(xt)}--\eqref{res-hatM_(xt)_2} 
is solvable (without assuming that $\hat M^{(xt)}$ is determined from
a global solution to the SW equation) is a delicate one. 
The main difficulty is caused by the additional jump on the circle $|k|=\epsilon$, which appears whenever $\kappa_2^0\neq0$. This circular jump arising from matching the formulations near $k=0$ and $k=\infty$  destroys the positivity structure required for the conventional vanishing lemma \cite{ZH89}. In the special case $\kappa_2^0=0$, the circular jump disappears, the RH problem reduces to 
that with the  real line as contour, and thus the vanishing lemma yields unique solvability. If $\kappa_2^0\neq 0$, then solvability can be established under suitable small-norm assumptions on the spectral data. 
\end{remark}

In this way, we arrive at the following representational result:

\begin{theorem}\label{Prop:rep}
Let $ u(x,t)$ be a solution of the SW equation in the domain $x>0$, $0<t<T$ such that
\begin{enumerate}[(i)]
    \item $u(\cdot, t)\in  H^{4}(0,\infty)$ and $u_{xx}(\cdot, t)\in  H^{2,1}(0,\infty)$ for all $t\in[0,T]$;
    \item $1-u_{xx}(x,0)>0$ for all $x\geq0$;
    \item $u(0,t)\leq0$ for all $t\in[0,T]$.
\end{enumerate}
Then $ u(x,t)$ can be represented in terms of a unique solution of the RH problem \eqref{jump_hatM_(xt)}--\eqref{res-hatM_(xt)_2}, for which the data (jump matrix and residue condition) are given in terms of the
initial value $ u(x,0)$ via the associated spectral functions.

\end{theorem}

\begin{proof}
  The solution of the RH problem \eqref{jump_hatM_(xt)}--\eqref{res-hatM_(xt)_2} is unique. Together with  \eqref{inf_M_(xt)} and \eqref{i_beh-M_(xt)}, this yields  the following procedure for representing $ u(x,t)$ in terms of $\hat M^{(xt)}(y,t, k)$:

  \begin{enumerate}[Step 1.]

   \item Given $u(x,0)$ construct $a(k)$ and $b(k)$  (see Section \ref{sec:3});
   
    \item Having $a( k)$ and $b( k)$,  compute $\kappa_i^0$, $i=1,2$, $\tilde a(k)$ and $\tilde b(k)$ via \eqref{a_at_i} and \eqref{tilde-a--a}, and construct the RH problem \eqref{jump_hatM_(xt)}--\eqref{res-hatM_(xt)_2};

    \item Solve the constructed RH problem \eqref{jump_hatM_(xt)}--\eqref{res-hatM_(xt)_2};

    \item Evaluate the solution of this RH problem at $ k=0$:
\begin{equation*}\label{M(xt)_at_0}
       \begin{aligned}
       \hat M^{(xt)}(y,t, k)=&\begin{pmatrix}
           \hat \alpha(y,t)&\hat \beta(y,t)\\
           \hat \beta(y,t)&\hat \alpha(y,t)
       \end{pmatrix}\left( I+\ii k\begin{pmatrix}
           \hat f_1(y,t)&\hat f_1(y,t)\\-\hat f_1(y,t)&-\hat f_1(y,t)
       \end{pmatrix}\right.+
       \\
      &\left. +k^2\begin{pmatrix}
           \hat g_1(y,t)&\hat g_2(y,t)\\ \hat g_2(y,t)&-\hat g_1(y,t)
       \end{pmatrix}+o(k^2)\right)\eul^{\ii k \hat\rho(y,t)\sigma_3}, 
    \end{aligned}
    \end{equation*}
and at $ k=\infty$:
    \begin{equation*}\label{hatM_xt__inf}
        M^{(xt)}(y,t, k)=I-\frac{\ii}{2 k}\begin{pmatrix}
    *&\hat \eta(y,t)\\
    -\hat \eta(y,t)&*
\end{pmatrix}+o\left(\frac{1}{ k}\right)
    \end{equation*}

    \item Define $\hat u(y,t)$ and $\hat m(y,t)$ from these expansions in the following way (cf. \eqref{inf_M_(xt)} and \eqref{i_beh-M_(xt)}):
    \begin{subequations}
        \begin{align}\label{hat_u_via_RH_xt}
     &\hat{ u}(y,t)=\hat g_2(y,t), \\\label{hat_u_x_via_RH_xt}
    & \hat{ v}(y,t)= 2\hat f_1(y,t),\\\label{hat_m_via_RH_xt}
   & \hat{ m}(y,t)=(\hat \alpha(y,t)+\hat \beta(y,t))^4-1.
    \end{align}
    \end{subequations} 
Define $\eta(t)$ as the solution of
\begin{equation}\label{eta_via_RH_xt}
\eta_t(t)
=
-\hat g_2(\eta(t),t)
\bigl(\hat\alpha(\eta(t),t)+\hat\beta(\eta(t),t)\bigr)^2,
\qquad
\eta(0)=0.
\end{equation}
Then define $x(y,t)$ by (cf. \eqref{y__geq})
\begin{equation}\label{y_via_RH_xt}
x(y,t)
=
\int_{\eta(t)}^y
\frac{\dd s}
{\bigl(\hat\alpha(s,t)+\hat\beta(s,t)\bigr)^2}.
\end{equation}
Then
\begin{subequations}\label{v_j_via_RH_xt}
\begin{align}
   & u(x,t)=
     \hat{ u}(y(x,t),t),\\
     & m(x,t)=
     \hat{m}(y(x,t),t).
    \end{align}
    \end{subequations}
\end{enumerate}
  
\end{proof}

\begin{remark}The essential difference from the corresponding half-line problem for the CH equation  \cite{BS08} lies in the spectral data needed to reconstruct the solution. In the case of the SW equation, provided that $u(0,t)\leq 0$, the solution $u(x,t)$ is reconstructed from the unique solution of the RH problem \eqref{jump_hatM_(xt)}--\eqref{res-hatM_(xt)_2}, whose jump matrix and residue conditions are determined, through the associated spectral functions, solely by the initial datum $u(x,0)$. By contrast, for the CH equation, the reconstruction also requires spectral data determined by the boundary values.
\end{remark}

\begin{corollary}\label{cor:uniq}
Let $u_1$ and $u_2$ be two solutions of the SW equation on
$[0,\infty)\times[0,T]$ such that, for $i=1,2$,
\begin{enumerate}[(i)]
    \item $u_i(\cdot,t)\in H^4(0,\infty)$ and $u_{i,xx}(\cdot,t)\in H^{2,1}(0,\infty)$ for all $t\in[0,T]$;
    \item $1-u_{i,xx}(x,0)>0$ for all $x\geq0$;
    \item $u_i(0,t)\leq0$ for all $0\leq t\leq T$;
\end{enumerate}
If $u_1(x,0)=u_2(x,0)$ for all $x\geq0$,
then
\[
u_1(x,t)=u_2(x,t),
\qquad x\geq0,\qquad 0\leq t<T.
\]
\end{corollary}

\subsection[The Riemann--Hilbert problem formalism in the case $u(0,t)\ge 0$]{\texorpdfstring{${u}(0,t)\ge 0$}{u-positive}}

\subsubsection{The pre-Riemann--Hilbert problem} 

Assume that we are given a solution $u(x,t)$ of the SW equation in the domain $(x,t)\in(0,\infty)\times (0,T)$ such that 
\begin{enumerate}[(i)]
    \item  $u(\cdot, t)\in  H^{4}(0,\infty)$ and $u_{xx}(\cdot, t)\in  H^{2,1}(0,\infty)$ for all $t\in[0,T]$,
    \item $-u_{xx}(x,0)+1>0$,
    \item $u(0,t)\ge 0$ and $-u_{xx}(0,t)+1>0$.
\end{enumerate}
Introduce 
\begin{align*}
    d(k)=a(k)A^*(k)-b(k)B^*(k),
\end{align*}
and consider the following piecewise meromorphic matrix-valued function $M(x,t, k)$ (depending on $(x,t)$ as parameters, for all $x\geq0$ and $t\in[0,T]$) defined in the domains separated by contour depicted 
in Figure \ref{fig:contour_RH_x}:

\begin{equation}\label{M_(xt)_geq}
M^{(xt)}(x,t,k)=\begin{cases}

\left( \frac{\Phi_{\infty 1}^{(1)}(x,t,k)}{d(k)},\Phi_{\infty 3}^{(2)}(x,t,k)\right),\quad k\in\mathbb{C}^+\cap\{|k| >\epsilon\},\\

\left( \frac{\Psi_{0 2}^{(1)}(x,t,k)}{\tilde a(k)},\Psi_{0 3 }^{(2)}(x,t,k)\right)\eul^{-\ii k\nu(0)\sigma_3},\quad k\in\mathbb{C}^+\cap\{|k| <\epsilon\},\\

\left( \Phi_{\infty 3}^{(1)}(x,t,k),\frac{\Phi_{\infty 1}^{(2)}(x,t,k)}{d^*( k)}\right),\quad k\in\mathbb{C}^-\cap\{|k| >\epsilon\},\\

\left( \Psi_{0 3}^{(1)}(x,t,k),\frac{\Psi_{0 2}^{(2)}(x,t,k)}{\tilde a^*( k)}\right)\eul^{-\ii k\nu(0)\sigma_3},\quad k\in\mathbb{C}^-\cap\{|k| <\epsilon\},

\end{cases}
\end{equation}
where the functions $\Psi_{0j}$
are defined by \eqref{psi_0}.

Then the function $M^{(xt)}(x,t,k)$ has the following properties:

\begin{enumerate}
    \item Jump relation across $\mathbb{R}\cup\{|k|=\epsilon\}$
    \begin{subequations}
        \label{jump_M_(xt)_geq}
        \begin{equation}
           M_-^{(xt)}(x,t,k)=M_+^{(xt)}(x,t,k)J^{(xt)}(x,t,k),\quad k\in \mathbb{R}\cup\{|k|=\epsilon\} 
        \end{equation}
        where
         \begin{equation} \label{jump_M_(xt)_matr_geq}
          J^{(xt)}(x,t,k)=\eul^{-p(x,t,k)\sigma_3} J^{(xt)}_0(k)\eul^{p(x,t,k)\sigma_3}
        \end{equation} 
        with
\begin{equation}\label{jump_M_(xt)_matr_0_geq}
   J^{(xt)}_0(k)=\begin{cases}
       \begin{pmatrix}
          1&\frac{aB-bA}{d^*}\\
          -\frac{a^*B^*-b^*A^*}{d}&\frac{1}{dd^*}
       \end{pmatrix},\quad k\in\mathbb{R}\cap\{|k| >\epsilon\},\\
   \begin{pmatrix}
       1-\tilde r(k)\tilde r^*(k)&\tilde r^*(k)\eul^{2\ii k\nu(0)}\\-\tilde r(k)\eul^{-2\ii k\nu(0)}&1
   \end{pmatrix},\quad k\in\mathbb{R}\cap\{|k| <\epsilon\},\\
   \begin{pmatrix}
       1&0\\
     -\frac{B^*(k)}{a(k)d(k)} + \frac{\kappa_2^0\eul^{-\ii k\nu(0)}}{a(k)\tilde a(k)}&1
   \end{pmatrix},\quad k\in\mathbb{C}^+\cap\{|k| =\epsilon\},\\
   \begin{pmatrix}
       1&\frac{B(k)}{a^*(k)d^*(k)}-\frac{\kappa_2^0\eul^{\ii k\nu(0)}}{a^*(k)\tilde a^*(k)}\\
       0&1
   \end{pmatrix},\quad k\in\mathbb{C}^-\cap\{|k| =\epsilon\},
   \end{cases} 
\end{equation}
\end{subequations}
and $r(k)=\frac{b^*(k)}{a(k)}$, $\tilde r(k)=\frac{\tilde b^*(k)}{\tilde a(k)}$. 

\begin{remark}\label{rem:x-RH_reduced_geq}
       In the  case $Q(0,0)=I$ (with $\kappa_1^0=1$ and $\kappa_2^0=0$),
    the jump matrix $J_0^{(xt)}(k)$ simplifies to
    \[
     J^{(xt)}_0(k)=\begin{cases}
       \begin{pmatrix}
          1&\frac{aB-bA}{d^*}\\
          -\frac{a^*B^*-b^*A^*}{d}&\frac{1}{dd^*}
       \end{pmatrix},\quad k\in\mathbb{R}\cap\{|k| >\epsilon\},\\
   \begin{pmatrix}
       1-\tilde r(k)\tilde r^*(k)&\tilde r^*(k)\eul^{2\ii k\nu(0)}\\-\tilde r(k)\eul^{-2\ii k\nu(0)}&1
   \end{pmatrix},\quad k\in\mathbb{R}\cap\{|k| <\epsilon\},\\
   \begin{pmatrix}
       1&0\\
     -\frac{B^*(k)}{a(k)d(k)}&1
   \end{pmatrix},\quad k\in\mathbb{C}^+\cap\{|k| =\epsilon\},\\
   \begin{pmatrix}
       1&\frac{B(k)}{a^*(k)d^*(k)}\\
       0&1
   \end{pmatrix},\quad k\in\mathbb{C}^-\cap\{|k| =\epsilon\}.
   \end{cases}.
    \]   
\end{remark}

\item Behavior at $\infty$:
\begin{subequations}\label{inf_M_(xt)_geq}
\begin{equation}\label{inf_M_exp(xt)_geq}
     M^{(xt)}(x,t,k)=I+\frac{1}{8\ii k}M^{\infty}(x,t)+o\left(\frac{1}{k}\right),\quad k\to\infty,
\end{equation}
with
\begin{equation}\label{inf_M_inf(xt)_geq}
M^{\infty}(x,t)=\begin{pmatrix}
    *&\frac{m_{x}}{(m+1)^{\frac{3}{2}}}\\
    -\frac{m_{x}}{(m+1)^{\frac{3}{2}}}&*\end{pmatrix}.
\end{equation}
\end{subequations}

\begin{remark}
    If, in addition,  $m(\cdot,t)\in H^{3,1}(0,\infty)$,  then the integration-by-parts argument can be iterated once more, and the remainder improves to $O(k^{-2})$. 
\end{remark}

\item $\det M^{(xt)}(x,t,k)\equiv1$

\item Symmetry properties:
\begin{equation}\label{sym-M_(xt)_geq}
M^{(xt)}( k)=\sigma_1\overline{M^{(xt)}(\bar k)}\sigma_1,\qquad M^{(xt)}(k)=\sigma_1M^{(xt)}(-k)\sigma_1
\end{equation}

\item Residue properties. Let $\{\mu_j\}$ be zeros of $d(k)$ in $\mathbb{C}^+$. Assume that they are all simple. Then

\begin{subequations}\label{res-M_(xt)_geq}
\begin{align}\label{res-M+_(xt)_geq}
\Res_{\mu_j}M^{(xt)(1)}(x,t,k)&=\frac{e^{2p(x,t,\mu_j)}B^*(\mu_j)}{\dot d(\mu_j)a(\mu_j)}M^{(xt)(2)}(x,t,\mu_j),\\
\label{res-M-_(xt)_geq}
\Res_{\bar \mu_j}M^{(xt)(2)}(x,t,k)&=\frac{e^{-2p(x,t,\bar \mu_j)}B(\bar \mu_j)}{\dot d^*(\bar \mu_j)a^*(\bar \mu_j)}M^{(xt)(1)}(x,t,\bar \mu_j).
\end{align}
\end{subequations}

Notice that $\frac{B(\bar \mu_j)}{\dot d^*(\bar \mu_j)a^*(\bar \mu_j)}=\overline{\frac{B^*(\mu_j)}{\dot d( \mu_j)a( \mu_j)}}$

\item Behavior at $0$: 
\begin{subequations}
    \label{i_beh-M_(xt)_geq}
\begin{equation}
M^{(xt)}(x,t, k)=Q(x,t)\left(I+\ii k M_1^{(xt)}(x,t)+k^2 M_2^{(xt)}(x,t)+o(k^2)\right) \eul^{\ii k\left( \int_0^{x} \left(\sqrt{ m+1}-1\right)(\xi,t)\dd\xi+\eta(t)-\nu(0)\right)\sigma_3}, 
\end{equation}
where
\begin{equation}
    M_1^{(xt)}=\begin{pmatrix}
    \frac{u_x(x,t)}{2}&\frac{u_x(x,t)}{2}\\
    -\frac{u_x(x,t)}{2}&-\frac{u_x(x,t)}{2}
\end{pmatrix},
\end{equation}

\begin{equation}
    M_2^{(xt)}=\begin{pmatrix}
    *&u(x,t)\\
    u(x,t)&*
\end{pmatrix}.
\end{equation}
\end{subequations}

\end{enumerate}

\subsubsection{The main Riemann--Hilbert problem parametrized by $y$ and $t$}

The construction of the jump matrix involves $p(x,t, k)$ (see \eqref{p}), which in turn involves $ u(x,t)$. This suggests the introduction of a new variable
\begin{equation}\label{y__geq}
y(x,t)=\int_0^{x} \left(\sqrt{ m+1}\right)(\xi,t)\dd\xi-\int_0^{t} \left(u\sqrt{ m+1}\right)(0,\tau)\dd\tau, 
\end{equation}
in terms of which the jump matrix and the residue conditions can be written
explicitly. Notice that $\frac{\partial y}{\partial x}=\sqrt{ m(x,t)+1}$ and $\frac{\partial y}{\partial t}=-\left(u\sqrt{ m+1}\right)(x,t)$.

\textbf{The Riemann--Hilbert problem RH$^{(xt)}$:} Given $a( k)$, $b( k)$, $A( k)$, $B( k)$ and the set $\{\mu_j\}_1^J\subset \mathbb{C}^+$, find a piecewise meromorphic $2\times 2$  matrix-valued function $\hat M^{(xt)}(y,t, k)$ that satisfies the following conditions:

\begin{enumerate}
    \item Jump relation across $\mathbb{R}\cup\{|k|=\epsilon\}$
    \begin{subequations}
        \label{jump_hatM_(xt)_geq}
        \begin{equation}
           \hat M_-^{(xt)}(y,t, k)=\hat M_+^{(xt)}(y,t, k)\hat J^{(xt)}(y,t, k),\quad k\in \mathbb{R}\cup\{|k|=\epsilon\} 
        \end{equation}
        where
         \begin{equation}
         \hat J^{(xt)}(y,t, k)=\eul^{-\hat p(y,t, k)\sigma_3} J^{(xt)}_0( k)\eul^{\hat p(y,t, k)\sigma_3}
        \end{equation} 
        with $J^{(xt)}_0( k)$ defined in \eqref{jump_M_(xt)_matr_0_geq} and $\hat p(y,t, k)=\ii k\left(y-\frac{t}{2k^2}\right)$.

    \end{subequations}

\item Behavior at $\infty$:
\begin{equation}\label{inf_hatM_(xt)_geq}
    \hat M^{(xt)}(y,t, k)=I+O(\frac{1}{ k}),\quad k\to\infty.
\end{equation}

\item Residue conditions:
\begin{subequations}\label{res-hatM_(xt)_2_geq}
\begin{align}\label{res-hatM+_(xt)_2_geq}
\Res_{\mu_j}\hat M^{(xt)(1)}(y,t, k)&=\frac{\eul^{2\hat p(y,t,\mu_j)}B^*(\mu_j)}{\dot d(\mu_j)a(\mu_j)}\hat M^{(xt)(2)}(y,t,\mu_j),\\
\label{res-hatM-_(xt)_2_geq}
\Res_{\bar \mu_j}\hat M^{(xt)(2)}(y,t, k)&=\frac{\eul^{-2\hat p(y,t,\bar \mu_j)}B(\bar\mu_j)}{\dot{ d^*}(\bar \mu_j)a^*(\bar \mu_j)}\hat M^{(xt)(1)}(y,t,\bar \mu_j).
\end{align}
\end{subequations}

\end{enumerate}

\begin{proposition} If the solution $\hat M^{(xt)}(y,t, k)$ of the RH problem \eqref{jump_hatM_(xt)_geq}--\eqref{res-hatM_(xt)_2_geq} exists, then
    \begin{enumerate} 
        \item $\det \hat M^{(xt)}(y,t, k)\equiv1$;
    
        \item  it is unique;

        \item it satisfies symmetries

        \begin{equation}\label{sym-hatM_(xt)_geq}
\hat M^{(xt)}( k)=\sigma_1\overline{\hat M^{(xt)}(\bar k)}\sigma_1,\qquad \hat M^{(xt)}(k)=\sigma_1\hat M^{(xt)}(-k)\sigma_1
\end{equation}
 
    \end{enumerate}
\end{proposition}

\begin{remark}
The situation concerning the solvability of the RH problem
\eqref{jump_hatM_(xt)_geq}--\eqref{res-hatM_(xt)_2_geq}
is considerably more involved than in the case $u(0,t)<0$.
Indeed, the additional jump on the circle $|k|=\epsilon$ is determined by
the boundary spectral functions through the quotients
\[
    \frac{B^*(k)}{a(k)d(k)}, \qquad
    \frac{B(k)}{a^*(k)d^*(k)}.
\]
Consequently, the circular jump does not disappear even in the case
$\kappa_2^0=0$, and the RH problem does not reduce to a standard real-line
formulation. Therefore, the conventional vanishing lemma is not directly applicable.

Nevertheless, if the reflection data together with the circular jump are
sufficiently small in suitable Sobolev norms, and if $a(k)$ and $d(k)$ have no
zeros on the contour, then the associated Beals--Coifman singular integral
equation is uniquely solvable by a Neumann series argument \cites{BealsCoifman1984,L18,deift1993steepest}. In the general case,
solvability may be formulated as a Fredholm problem of index zero, with unique
solvability equivalent to the absence of nontrivial solutions of the
corresponding homogeneous RH problem.
\end{remark}

In this way, we arrive at the following representational result:

\begin{theorem}\label{Prop:rep_geq}
Let $u(x,t)$ be a solution of the SW equation on
$[0,\infty)\times[0,T]$ such that 
\begin{enumerate}[(i)]
    \item $u(\cdot,t)\in H^4(0,\infty)$ and
$u_{xx}(\cdot,t)\in H^{2,1}(0,\infty)$ for all
 $t\in[0,T]$;
 \item $1-u_{xx}(x,0)>0$ for $x\geq 0$;
 \item $u(0,t)\geq0$ and $1-u_{xx}(0,t)>0$ for all $0\leq t\leq T$.
\end{enumerate}
 Then $ u(x,t)$ can be represented in terms of a unique solution of the RH problem \eqref{jump_hatM_(xt)_geq}--\eqref{res-hatM_(xt)_2_geq}, for which the data (jump matrix and residue condition) are given in terms of the
initial value $ u(x,0)$ and boundary values $ u(0,t)$, $ u_x(0,t)$, and $ u_{xx}(0,t)$ via the associated spectral functions.

\end{theorem}

\begin{proof}
  The solution of the RH problem \eqref{jump_hatM_(xt)_geq}--\eqref{res-hatM_(xt)_2_geq} is unique. Together with  \eqref{inf_M_(xt)_geq} and \eqref{i_beh-M_(xt)_geq}, this yields  the following procedure for representing $ u(x,t)$ in terms of $\hat M^{(xt)}(y,t, k)$:

  \begin{enumerate}[Step 1.]
    \item Given $u(x,0)$ construct $a(k)$, $b(k)$, $\tilde a(k)$, and $\tilde b(k)$; given $u(0,t)$, $u_x(0,t)$, and $u_{xx}(0,t)$, construct $A(k)$, $B(k)$, $\tilde A(k)$, and $\tilde B(k)$  (see Section \ref{sec:3}) and construct the RH problem \eqref{jump_hatM_(xt)_geq}--\eqref{res-hatM_(xt)_2_geq};

    \item Solve the constructed RH problem \eqref{jump_hatM_(xt)_geq}--\eqref{res-hatM_(xt)_2_geq};

    \item Evaluate the solution of this RH problem at $ k=0$:
       \begin{equation*}\label{M(xt)_at_0_geq}
       \begin{aligned}
       \hat M^{(xt)}(y,t, k)=&
            \begin{pmatrix}
           \hat \alpha(y,t)&\hat \beta(y,t)\\
           \hat \beta(y,t)&\hat \alpha(y,t)
       \end{pmatrix}\left( I+\ii k\begin{pmatrix}
           \hat f_1(y,t)&\hat f_1(y,t)\\-\hat f_1(y,t)&-\hat f_1(y,t)
       \end{pmatrix}\right.+
       \\
      &\left.+k^2\begin{pmatrix}
           \hat g_1(y,t)&\hat g_2(y,t)\\ \hat g_2(y,t)&-\hat g_1(y,t)
       \end{pmatrix}+o(k^2)\right)\eul^{\ii k \hat\rho(y,t)\sigma_3},
       \end{aligned}       
    \end{equation*}
and at $ k=\infty$:
    \begin{equation*}\label{hatM_xt__inf_geq}
        M^{(xt)}(y,t, k)=I-\frac{\ii}{2 k}\begin{pmatrix}
    *&\hat \eta(y,t)\\
    -\hat \eta(y,t)&*
\end{pmatrix}+o\left(\frac{1}{ k}\right)
    \end{equation*}

    \item Define $\hat u(y,t)$ and $\hat m(y,t)$ from these expansions in the following way (cf. \eqref{inf_M_(xt)_geq} and \eqref{i_beh-M_(xt)_geq}):
    \begin{subequations}
    \begin{align}\label{hat_u_via_RH_xt_geq}
     &\hat{ u}(y,t)=\hat g_2(y,t), \\\label{hat_u_x_via_RH_xt_geq}
    & \hat{ v}(y,t)= 2\hat f_1(y,t),\\\label{hat_m_via_RH_xt_geq}
   & \hat{ m}(y,t)=(\hat \alpha(y,t)+\hat \beta(y,t))^4-1.
    \end{align}
     \end{subequations}

Define $x(y,t)$ as the solution of the problem (cf. \eqref{y__geq})
 \begin{subequations}
    \begin{align}\label{y_via_RH_xt_geq}
     &\frac{\partial x}{\partial y}=\frac{1}{(\hat \alpha(y,t)+\hat \beta(y,t))^2},\\
     &x(\eta(t),t)=0
    \end{align}
      \end{subequations}
where $\eta(t)=-\int_0^t \left(u \sqrt{ m+1}\right)(0,\tau)\dd \tau$.
Then
\begin{subequations}\label{v_j_via_RH_xt_geq}
\begin{align}
   & u(x,t)=
     \hat{ u}(y(x,t),t),\\
     & m(x,t)=
     \hat{m}(y(x,t),t).
    \end{align}
    \end{subequations}
    
\end{enumerate}
  
\end{proof}

\subsubsection{Initial Boundary Value Problem}

To fix ideas, we assume that the spectral functions have no common zeros.

\begin{theorem}\label{prop:ex}
    Let functions  $\{ u_0(x),~x\geq0; ~\{v_j(t)\}_0^2,~0\leq t\leq T<\infty\}$ be 
    such that the following conditions are satisfied:

    \begin{enumerate}

    \item Regularity conditions
    \begin{enumerate}[(i)]
    \item $u_{0xx}\in H^{2,1}(0,\infty)$;
    \item $v_j\in H^{2,0}(0,T)$, $ j=0,1,2$.
    \end{enumerate}
        \item Internal compatibility  conditions

        \begin{enumerate}[(i)]
    \item $\partial_x^j  u_0(0)=v_j(0)$, $j=0,1,2$;
    \item if $v_0(t)=0$ on $(T_1,T_2)$, then $v_{2t}(t)-2v_1(t)(1-v_2(t))=0$ on $(T_1,T_2)$;
    \item $
v_0(t)\left(1-\frac{v_2(t)}{2}\right)
-\frac{v_1^2(t)}{4}
=\frac{v_{1t}(t)}{2}$ for $0\leq t\leq T$.
\end{enumerate}

\item Sign conditions  
\begin{enumerate}[(i)]
\item $1-u_{0xx}(x)>0$ for $x\ge0$;
\item $1-v_2(t)>0$ for $0\le t\le T$;
    \item $v_0(t)\geq0$.
\end{enumerate}

\item Compatibility of boundary and initial data:
 the spectral functions 
$\tilde a(k)$, $\tilde b(k)$, $\tilde A(k)$, and $\tilde B(k)$
associated to the data via the direct $x$-problem and $t$-problem
satisfy the global relation \eqref{relations_Phi_03_i}.

\item The RH problem \eqref{jump_hatM_(xt)_geq}--\eqref{res-hatM_(xt)_2_geq} has a solution for all $y\geq \eta(t)$ and $0\leq t<T$ such that
\begin{enumerate}[(a)]
\item it satisfies the differentiability
asymptotic conditions \eqref{regularity_M_infty} and
\eqref{regularity_M_zero}.
\item 
functions 
$\hat g_2(y,t)$, $\hat f_1(y,t)$, $\hat \alpha (y,t)$, and $\hat \beta (y,t)$ 
evaluated from $\hat M^{(xt)}(y,t,k)$, see \eqref{i_beh-M_(xt)_geq} and \eqref{inf_M_(xt)_geq},
are differentiable w.r.t. $y$ and $t$;

\item functions $\hat \eta (\cdot,t)$, $(\hat \alpha(\cdot,t)+\hat{\beta}(\cdot,t))^4-1$, $\hat f_2 (\cdot,t)$, and $\hat g_2(\cdot,t)-\hat f_2 (\cdot,t)\hat f_1 (\cdot,t)+\hat f_2 (\cdot,t)^2$ evaluated from $\hat M^{(xt)}(y,t,k)$, see \eqref{i_beh-M_(xt)_geq} and \eqref{inf_M_(xt)_geq},  are in $H^{1,1}(0,\infty)$.
 
\end{enumerate}

    \end{enumerate}

Then the IBVP 
\begin{subequations}\label{SP_IBVP_}
\begin{align}\label{SP-2_IBVP_}
&\left(\sqrt{m+1}\right)_t=-\left(u\sqrt{m+1}\right)_x, \qquad m=-u_{xx};\\
     \label{ic_IBVP_}
     &u(x,0) =  u_0(x), \quad x \geq 0;\\
\label{boundary_IBVP_}
    & u(0,t) = v_0(t), \quad  u_x(0,t) = v_1(t), \quad  u_{xx}(0,t) = v_2(t), \quad
     0\leq t < T\leq\infty;\\
     \label{as_IBVP_}
     &  u(\cdot, t)\in  H^{4}(0,\infty) \text{ and } u_{xx}(\cdot, t)\in  H^{2,1}(0,\infty) \text{ for all } t\in[0,T],
\end{align}
\end{subequations}
has a unique solution, $ u(x,t)$, which can be represented in terms of the solution of the associated RH problem \eqref{jump_hatM_(xt)_geq}--\eqref{res-hatM_(xt)_2_geq} in the parametric form \eqref{hat_u_via_RH_xt_geq}, \eqref{y_via_RH_xt_geq}, \eqref{v_j_via_RH_xt_geq}. 

\end{theorem}

\begin{proof}
    The proof consists of the following steps:
\begin{enumerate}[Step 1.]
    \item Prove that $ u(x,t)$ satisfies the SW equation \eqref{SP_IBVP_}.

    \item  Prove that $u(x,0) =u_0(x)$.

    \item Prove that $ u(0,t) = v_0(t)$, $u_x(0,t) = v_1(t)$ and $ u_{xx}(0,t) = v_2(t)$.
\end{enumerate}

\textbf{Step 1.} The statement follows from the considerations in Section \ref{sec:SWinyt}.

\textbf{Step 2.} The proof of Step~2 relies on relating the solution $\hat M^{(xt)}(y,t,k)$ of the RH problem  \eqref{jump_hatM_(xt)_geq}--\eqref{res-hatM_(xt)_2_geq} at $t=0$ to the solution $\hat M^{(x)}(y,k)$ of the  RH problem \eqref{jump_hatM_(x)}--\eqref{res-hatM-_(x)}. This relation is expressed through multiplication by a suitable matrix factor:
 \begin{subequations}
     \begin{align}
         &P^{(x)}(y,k)=\eul^{-\hat p(y,0,k)\sigma_3}P_0^{(x)}(k)\eul^{\hat p(y,0,k)\sigma_3},\\
         & P_0^{(x)}(k)=\begin{cases}
        \begin{pmatrix}
           1&0\\
           -\frac{B^*(k)}{a(k)d(k)}&1
        \end{pmatrix},\quad k\in \mathbb{C}^+\cap\{|k| >\epsilon\},\\
         \begin{pmatrix}
           1&-\frac{B(k)}{a^*(k)d^*(k)}\\
           0&1
        \end{pmatrix},\quad k\in \mathbb{C}^-\cap\{|k| >\epsilon\},\\
       I, \quad \text{otherwise}.
     \end{cases}
     \end{align}
 \end{subequations}

Thus, the problem reduces to show that  $\hat{ N}^{(xt)}(y,k)$ defined by
 \begin{equation}
   \hat{ N}^{(xt)}(y,k)=\hat{ M}^{(xt)}(y,0,k)P^{(x)}(y,k)
 \end{equation}
 satisfies the RH problem \textbf{RH$^{(x)}$}.

\begin{enumerate}
    \item The jump conditions match by construction.

    \item Noticing that 
    $P_0^{(x)}(k)=I+O\left(\frac{1}{k}\right)$ as $k\to\infty$ in $\mathbb{C}^+$ (which follows from $B^*(k)=O(\frac{1}{k})$), we see that the normalization condition is satisfied.

    \item Finally, let's check that the residue conditions match.

For $k\in \mathbb{C}^+$, we have 
 \begin{align*}
     \hat{ N}^{(xt)(1)}(y,k)&=\hat{M}^{(xt)(1)}(y,0,k)-\frac{B^*(k)\eul^{2\hat p(y,0,k)}}{a(k)d(k)}\hat{M}^{(xt)(2)}(y,0,k),\\
     \hat{ N}^{(xt)(2)}(y,k)&=\hat{M}^{(xt)(2)}(y,0,k).
 \end{align*}
 Taking into account the residue condition \eqref{res-hatM+_(xt)_2_geq}, the singularities of the first column at $k=\mu_j$ cancel.

 On the other hand, at $k=k_j$, the first column is singular due to the singularity of $\frac{1}{a(k)}$, and, using $d(k_j)=-b(k_j)B^*(k_j)$, the corresponding residue condition takes the form \eqref{res-hatM+_(x)}.

 The singularities in $\mathbb{C}^-$  can be treated in a similar way.
\end{enumerate}

Thus, by uniqueness of solution of RH problem \textbf{RH$^{(x)}$}, we conclude that $\hat{ M}^{(xt)}(y,0,k)P^{(x)}(y,k)=\hat{ M}^{(x)}(y,k)$.

Since $P^{(x)}_{0}(k)=I$ in $\{|k| <\epsilon\}$, 
relations \eqref{m_0_via_RH_x_m_2} and \eqref{hat_m_via_RH_xt} imply that
\[
\hat { m}(y,0)=\hat{ m}_{0}(y).
\]
Furthermore, combining \eqref{M_(x)_x_y} and \eqref{y_via_RH_xt} yields
\[
x(y)=x(y,0),
\]
where $x(y)$ on the left-hand side is the change of the variable associated with the $x$-problem for $ m_{0}$, whereas the right-hand side is the change of the variable associated with the main RH problem at $t=0$.

\textbf{Step 3.} First, notice that by condition~\textit{(2)(iii)}, the boundary data satisfy the compatibility
condition~\eqref{inner_comp} required for the inverse $t$-spectral mapping developed
in Section~\ref{sec:5.2}.

The proof of Step~3 follows the same strategy as the proof of Step~2. We relate the solution $\hat M^{(xt)}(y,t,k)$, evaluated at $y=\eta(t)$, to the solution $\hat M^{(t)}(z,t,k)$ of the RH problem \textbf{RH$^{(t)}$} associated with the boundary data $v_0(t)$, $v_1(t)$, and $v_2(t)$. This relation is realized through multiplication by a suitable piecewise meromorphic matrix factor:
    \begin{equation}
        P^{(t)}(t,k)=\eul^{-\hat p(\eta(t),t,k)\sigma_3}P^{(t)}_0(k)\eul^{\hat p(\eta(t),t,k)\sigma_3}
    \end{equation}
with
\begin{equation}
     P^{(t)}_0(k)=\begin{cases}

\begin{pmatrix}
             \frac{d(k)}{A^*(k)}&-b(k)\\
            0&\frac{A^*(k)}{d(k)}
         \end{pmatrix},\quad k\in \mathbb{C}^+\cap\{|k| >\epsilon\},\\

\begin{pmatrix}
    \tilde a(k)\eul^{\ii k\nu(0)}
&\eul^{\ii k\nu(0)}
\frac{\tilde B(k)\tilde a(k)-\tilde b(k) \tilde A(k)}{\tilde A(k)}\\
    0&\frac{ \eul^{-\ii k\nu(0)}}{ \tilde a(k)
}
\end{pmatrix}, \quad k\in \mathbb{C}^+\cap\{|k| <\epsilon\},\\

\begin{pmatrix}
       \frac{A(k)}{d^*(k)}     &0\\
            -b^*(k)& \frac{d^*(k)}{A(k)}
         \end{pmatrix},\quad k\in \mathbb{C}^-\cap\{|k| >\epsilon\},\\

\begin{pmatrix}
\frac{\eul^{\ii k\nu(0)}
}{ \tilde a^*(k)}    &0\\
   \eul^{-\ii k\nu(0)} \frac{\tilde B^*(k)\tilde a^*(k)-\tilde b^*(k) \tilde A^*(k)}{\tilde A^*(k)} &\tilde a^*(k)\eul^{-\ii k\nu(0)}
\end{pmatrix}, \quad k\in \mathbb{C}^-\cap\{|k| <\epsilon\}.

\end{cases}
 \end{equation}

Thus the problem reduces to show that  $\hat{\hat N}^{(xt)}(t,k)$ defined by
 \begin{equation}
   \hat{\hat N}^{(xt)}(t,k)=\hat{ M}^{(xt)}(\eta(t),t,k)P^{(t)}(t,k)
 \end{equation}
 satisfies the RH problem \textbf{RH$^{(t)}$} with $z=\eta(t)$.

 \begin{enumerate}
     \item The jump conditions match by construction.

     \item The asymptotic properties of $a(k)$, $b(k)$, and $A(k)$ together with the fact that $\eul^{-2\hat p(\eta(t),t,k)}$ decays exponentially fast provide that 
     $P^{(t)}(t,k)\to I$ as $k\to\infty$ in $\mathbb{C}^+$ for all $t\leq T$. A similar argument applies in $\mathbb{C}^-$. Therefore, the normalization condition is satisfied. 

\item Let's check that the residue conditions match.

    For $k\in \mathbb{C}^+$, we have

      \begin{align*}
     \hat{\hat{ N}}^{(xt)(1)}(t,k)&=\frac{d(k)}{A^*(k)}\hat{M}^{(xt)(1)}(\eta(t),t,k),\\
     \hat{\hat{ N}}^{(xt)(2)}(t,k)&=-b(k)\eul^{-2\ii k\left(\eta(t)-\frac{t}{2 k^2}\right)}\hat{M}^{(xt)(1)}(\eta(t),t,k)+\frac{A^*(k)}{d(k)}\hat{M}^{(xt)(2)}(\eta(t),t,k).
 \end{align*}

Taking into account the residue condition \eqref{res-hatM+_(xt)_2_geq} together with the fact that $\mu_j$ is a simple zero of $d(k)$, the singularities of both columns at $k=\mu_j$ cancel.

On the other hand, at $k= \bar\nu_j$, the first column is singular due to the singularity of $\frac{1}{A^*(k)}$, and, using $d(\bar \nu_j)=-b(\bar \nu_j)B^*(\bar \nu_j)$, the corresponding residue condition takes the form \eqref{res-hatM-_(t)_geq}. 

The singularities in $\mathbb{C}^-$  can be treated in a similar way.

\end{enumerate}

Then, by uniqueness of solution of RH problem \textbf{RH$^{(t)}$}, we conclude that $\hat{\hat N}^{(xt)}(t,k)=\hat{ M}^{(t)}(\eta(t),t,k)$.

Since $v_0(t)$, $v_1(t)$ and $v_2(t)$ are expressed in terms of the first three coefficients in the expansion of $\hat{ M}^{(t)}(\eta(t),t,k)$ near $k=0$, it is important to control $P^{(t)}(t,k)$ near this point. 

Recall that \eqref{tila_at_0_} reads as
\[
\tilde a(k)=1-\frac{\ii k}{2}u_{0x}(0)+O(k^3),\quad k\to 0,
\]
and thus
\[
\frac{1}{\tilde a(k)}=1+\frac{\ii k}{2}u_{0x}(0)-\frac{k^2}{4}u_{0x}^2(0)+O(k^3),\quad k\to 0.
\]
Therefore, the global relation
\eqref{relations_Phi_03_i} implies that, for every
fixed $0\leq t<T$, the off-diagonal entry of $P^{(t)}(t,k)$ is
exponentially small as $k\to0$ in closed subsectors of $\mathbb C^+$.
Together with the expansions of $\tilde a(k)$,
$\tilde a^{-1}(k)$, and $\eul^{\pm\ii k\nu(0)}$, this gives
\begin{equation}\label{P_t_i}
\begin{aligned}
P^{(t)}(t,k)
&=
I+\ii k
\begin{pmatrix}
\nu(0)-\dfrac{u_{0x}(0)}{2}&0\\
0&-\nu(0)+\dfrac{u_{0x}(0)}{2}
\end{pmatrix}\\
&\quad
+k^2
\begin{pmatrix}
\dfrac{u_{0x}(0)\nu(0)}{2}-\dfrac{\nu^2(0)}{2}&0\\[2mm]
0&
-\dfrac{u_{0x}^2(0)}{4}
+\dfrac{u_{0x}(0)\nu(0)}{2}
-\dfrac{\nu^2(0)}{2}
\end{pmatrix}
+O(k^3),\qquad k\to 0,\quad k\in\mathbb{C}^+.
\end{aligned}
\end{equation}
Substituting $(y,t)=(\eta(t),t)$ into \eqref{hat_M_near_0} yields
\[
\hat M^{(xt)}(\eta(t),t,k)
=
\begin{pmatrix}
\hat\alpha&\hat\beta\\
\hat\beta&\hat\alpha
\end{pmatrix}
\left(
I+\ii k\begin{pmatrix}
\hat f_1&\hat f_2\\
-\hat f_2&-\hat f_1
\end{pmatrix}+k^2\begin{pmatrix}
\hat g_1&\hat g_2\\
\hat g_2&\hat g_4
\end{pmatrix}+o(k^2)
\right), \qquad k\to 0,\quad k\in\mathbb{C}^+,
\]
where all the coefficients are evaluated at $(\eta(t),t)$.

Now, set
$p_1=\nu(0)-\frac{u_{0x}(0)}{2}$.
Multiplying the preceding expansion by \eqref{P_t_i}, we obtain
\[
\begin{aligned}
\hat{\hat N}^{(xt)}(t,k)
=
\begin{pmatrix}
\hat\alpha&\hat\beta\\
\hat\beta&\hat\alpha
\end{pmatrix}
\left(
I+\ii k
\begin{pmatrix}
\hat f_1+p_1&\hat f_2\\
-\hat f_2&-\hat f_1-p_1
\end{pmatrix}
+k^2
\begin{pmatrix}
*&\hat g_2+p_1\hat f_2\\
\hat g_2+p_1\hat f_2&*
\end{pmatrix}
+o(k^2)
\right), \qquad k\to 0,\quad k\in\mathbb{C}^+.
\end{aligned}
\]

To compare this expansion with \eqref{M(t)_at_0_geq}, denote the
coefficients in the latter expansion by
$\hat\alpha^{(t)}$, $\hat\beta^{(t)}$,
$\hat f_1^{(t)}$, $\hat f_2^{(t)}$, and $\hat g^{(t)}$.
Expanding its last factor gives
\[
\hat M^{(t)}(\eta(t),t,k)
=
\begin{pmatrix}
\hat\alpha^{(t)}&\hat\beta^{(t)}\\
\hat\beta^{(t)}&\hat\alpha^{(t)}
\end{pmatrix}
\left(
I+\ii k
\begin{pmatrix}
\hat f_1^{(t)}+\rho(t)&\hat f_2^{(t)}\\
-\hat f_2^{(t)}&-\hat f_1^{(t)}-\rho(t)
\end{pmatrix}
+k^2
\begin{pmatrix}
*&\hat g^{(t)}+\rho(t)\hat f_2^{(t)}\\
\hat g^{(t)}+\rho(t)\hat f_2^{(t)}&*
\end{pmatrix}
+o(k^2)
\right).
\]

Since
$\hat{\hat N}^{(xt)}(t,k)
=
\hat M^{(t)}(\eta(t),t,k)$,
comparison of the coefficients yields
\[
\hat\alpha^{(t)}=\hat\alpha,
\qquad
\hat\beta^{(t)}=\hat\beta,
\qquad
\hat f_2^{(t)}=\hat f_2,
\]
and
\[
\hat f_1^{(t)}+\rho(t)=\hat f_1+p_1,
\qquad
\hat g^{(t)}+\rho(t)\hat f_2
=
\hat g_2+p_1\hat f_2.
\]

Consequently, by \eqref{hat_v_j_via_RH_t_geq} and \eqref{hmbbgg},
\[
\hat v_1^{(t)}
=
2\hat f_2^{(t)}
=
2\hat f_2
=
\hat v,
\]
and
\[
\hat v_2^{(t)}
=
1-\left(\hat\alpha^{(t)}+\hat\beta^{(t)}\right)^4
=
-\hat m.
\]
Moreover,
\[
\begin{aligned}
\hat v_0^{(t)}
&=
\hat g^{(t)}
-\hat f_1^{(t)}\hat f_2^{(t)}
+\left(\hat f_2^{(t)}\right)^2=
\left(\hat g_2+p_1\hat f_2-\rho(t)\hat f_2\right)
-\left(\hat f_1+p_1-\rho(t)\right)\hat f_2
+\hat f_2^2=
\hat g_2-\hat f_1\hat f_2+\hat f_2^2
=
\hat u.
\end{aligned}
\]
Therefore, we conclude that
\[
\hat u(\eta(t),t)=v_0(t),
\qquad
\hat v(\eta(t),t)=v_1(t),
\qquad
\hat m(\eta(t),t)=-v_2(t).
\]
Since $f(x,t)=\hat f(y(x,t),t)$ and $x(\eta(t),t)=0$, it follows that
\[
u(0,t)=v_0(t),
\qquad
u_x(0,t)=v_1(t),
\qquad
u_{xx}(0,t)=v_2(t).
\]

\end{proof}

\begin{remark}\label{rem:small-norm-solvability-geq}
Theorem \ref{prop:ex} has a conditional nature in the sense that it relies on  the solvability of the associated
RH problem. In turn, a sufficient condition for its solvability can be given by a small-norm argument. 

Assume first that the RH problem is pole-free and set
$\Gamma_\epsilon
=
\mathbb{R}\cup\{k\in\mathbb{C}:|k|=\epsilon\}$.
For a matrix-valued function $f$ on $\Gamma_\epsilon$, let
\[
(Cf)(k)
=
\frac{1}{2\pi\ii}
\int_{\Gamma_\epsilon}
\frac{f(s)}{s-k}\,\dd s,
\qquad k\notin\Gamma_\epsilon,
\]
be its Cauchy transform, and denote its non-tangential boundary values
by $C_\pm f$.

Set
\[
v(y,t,k)
=
\bigl(\hat J^{(xt)}(y,t,k)\bigr)^{-1},
\qquad
w(y,t,k)=v(y,t,k)-I,
\qquad
C_wf=C_-(fw).
\]
Suppose that
$w(y,t,\cdot)\in
L^2(\Gamma_\epsilon)\cap L^\infty(\Gamma_\epsilon)$
and
\[
\|C_-\|_{L^2\to L^2}
\|w(y,t,\cdot)\|_{L^\infty}<1.
\]
Then $\|C_w\|_{L^2\to L^2}<1$, so $I-C_w$ is invertible.
Consequently, the equation
\[
\mu-I=C_-(\mu w)
\]
has a unique solution $\mu\in I+L^2(\Gamma_\epsilon)$ given by
\[
\mu-I
=
(I-C_w)^{-1}C_-(w)
=
\sum_{n=0}^{\infty}C_w^nC_-(w).
\]
Moreover,
\[
\|\mu-I\|_{L^2}
\leq
\frac{
\|C_-\|_{L^2\to L^2}\|w\|_{L^2}
}{
1-\|C_-\|_{L^2\to L^2}\|w\|_{L^\infty}
}.
\]
The corresponding solution of the RH problem is
\[
\hat M^{(xt)}(y,t,k)
=
I+C(\mu w)(k).
\]
Indeed, the Plemelj relation gives
$\hat M_-^{(xt)}=\mu$ and
$\hat M_+^{(xt)}=\mu(I+w)=\hat M_-^{(xt)}v$,
which is equivalent to
$\hat M_-^{(xt)}=\hat M_+^{(xt)}\hat J^{(xt)}$.

Thus, a sufficient condition for solvability throughout the
reconstruction domain is
\[
\sup_{\substack{0\leq t<T\\y\geq\eta(t)}}
\|C_-\|_{L^2\to L^2}
\|w(y,t,\cdot)\|_{L^\infty}
<1.
\]
This condition concerns the complete jump matrix. Its verification
therefore requires estimates for the spectral functions
$a$, $b$, $\tilde a$, $\tilde b$, $A$, and $B$, as well as for
$d=aA^*-bB^*$, $\kappa_2^0$, and $\nu(0)$. In particular, the
denominators $a$, $\tilde a$, $d$, and their Schwarz conjugates must be
bounded away from zero on the corresponding parts of the contour.
For fixed $\epsilon$ and $T$, the defining Volterra equations provide
a way of verifying these bounds for sufficiently small initial and
boundary data.

If poles are present, the similar small-norm argument may be applied after replacing them by jumps on small circles.
\end{remark}

\section{Concluding Remarks}

The adaptation of the Riemann--Hilbert approach to initial boundary value problems for integrable nonlinear PDE faces, in general, the problem that the construction of the underlying RH problem  requires more data than can be specified for a well-posed  problem with initial and boundary conditions. The compatibility of boundary values that have to be prescribed for the RH problem can then be characterized in the spectral terms, using the so-called global relation.
The distinguishing feature of the problem considered in the present paper --- the IBV problem for the SW equation on the half-line --- is that the boundary is not intrinsically of inflow or outflow type: its character is determined dynamically by the unknown boundary value $u(0,t)$, which leads to two qualitatively different spectral formulations. 

Moreover, a part of the global relation, relation \eqref{relations_Phi_inf3}, need not be imposed as an additional condition in either of the two spectral formulations. Namely, in the case $u(0,t)\geq0$, it follows automatically from the analyticity and large-$k$ behavior of the spectral functions established in Section~\ref{sec:3} and therefore places no further restriction on the spectral data. 

On the other hand, in the case $u(0,t)\leq0$, relation \eqref{relations_Phi_inf3} is, in general, a genuine compatibility condition between the initial and boundary spectral data. However, the boundary spectral functions $A$ and $B$ do not enter the corresponding RH problem~\eqref{jump_hatM_(xt)}--\eqref{res-hatM_(xt)_2}, which is formulated entirely in terms of the initial spectral functions $a$ and $b$. Consequently, relation \eqref{relations_Phi_inf3} plays no role in the construction of the RH problem in this regime.

The RH framework developed here provides a spectral characterization of the IBV problem and a basis for further studies, including the analysis of admissible boundary conditions, the generalized Dirichlet-to-Neumann map, and the long-time asymptotics of solutions.


\section{Acknowledgment}
This research was partially funded by the Research Council of Norway under project 361083 (Ukraina-NASTRAN cooperation).
IK acknowledges the support from the Austrian Science Fund (FWF), grant no.
10.55776/ESP691. 
DS acknowledges the support from the National 
Research Foundation of Ukraine, grant no. 2025.07/0437.

\bibliography{shepelsky_etal}

@book{trogdon2015riemann,
author = {Trogdon, Thomas and Olver, Sheehan},
title = {{Riemann--Hilbert problems, their numerical solution, and the computation of nonlinear special functions}},
publisher = {Society for Industrial and Applied Mathematics},
year = {2015},
address = {Philadelphia, PA}
}

@book{fokas2008unified,
author = {Fokas, Athanassios S.},
title = {{A unified approach to boundary value problems}},
publisher = {Society for Industrial and Applied Mathematics},
year = {2008},
}

@article{SKPB24,
  title={{Periodic finite-band solutions 
to the focusing nonlinear Schr\"{o}dinger equation by the Fokas method: inverse and direct problems}},
  author={Shepelsky, D. and Karpenko, I. and Bogdanov, S.
          and Prilepsky, J.},
  journal={Proceedings of the Royal Society A},
  volume={480},
  pages={20230828},
  year={2024},
  publisher={The Royal Society Publishing}
}

@article{BSL17,
 author = {A.~{Boutet de Monvel} and D. Shepelsky  and L.  Zielinski},
 title = {{The short pulse equation by a Riemann–Hilbert approach}},
 journal = {Letters in Mathematical Physics},
 year = {2017},
 volume = {107},
 pages = {1345--1373},
}

@article{BKS20,
    author = {A.~{Boutet de Monvel} and I. Karpenko and D. Shepelsky},  
    title = {{A Riemann--Hilbert approach to the modified Camassa--Holm equation
   with nonzero boundary conditions}},
    journal = {J. Math. Phys.},
    year = 2020,
    volume = {61, 3},
    pages = {031504}
}

@article{CE98,
  title={Well-posedness, global existence,
and blowup phenomena
for a periodic quasi-linear hyperbolic equation},
  author={Constantin, A. and  Escher, J.},
  journal={Communications on Pure and Applied Mathematics},
  volume={LI},
  pages={0475--0504},
  year={1998},
  publisher={John Wiley \& Sons, Inc}
}

@article{deift1993steepest,
  title={A steepest descent method for oscillatory {R}iemann--{H}ilbert problems. {A}symptotics for the {MKdV} equation},
  author={Deift, P. and Zhou, X.},
  journal={Annals of Mathematic},
  volume={137},
  pages={295--368},
  year={1993},
  }

@ARTICLE{BFS06,
  author={Boutet de Monvel, Anne and Fokas, Athanassios and Shepelsky, Dmitry},
  journal={Commun. Math. Phys.}, 
  title={Integrable nonlinear evolution equations on a finite interval}, 
  year={2006},
  volume={263},
  pages={133-172}
  }

@ARTICLE{FIS,
  author={Fokas, A. and Its, A. R. and {L-Y Sung} },
  journal={Nonlinearity}, 
  title={{The nonlinear Schr\"odinger equation on the half-line}}, 
  year={2005},
  volume={18},
  number={4},
  pages={1771}
  }

@ARTICLE{BS08,
  author={Boutet de Monvel, A. and Shepelsky, D.},
  title={{The Camassa--Holm Equation on the Half-Line: a Riemann--Hilbert Approach}},
  journal={J. Geom. Anal.},
  volume={18},
  year={2008},
  pages={285--323}
}

@ARTICLE{F02,
  author={{Fokas}, A. S.},
  title={{Integrable nonlinear evolution equations on the half-line}},
  journal={Commun. Math. Phys.},
  volume={230},
  number={1},
  year={2002},
  pages={1--39}
}

@article{LF2009,
  author  = {Lenells, Jonatan and Fokas, Athanassios S.},
  title   = {{An integrable generalization of the nonlinear Schrödinger equation on the half-line and solitons}},
  journal = {Inverse Problems},
  volume  = {25},
  number  = {11},
  pages   = {115006},
  year    = {2009},
  doi     = {10.1088/0266-5611/25/11/115006},
  publisher = {IOP Publishing}
}

@article{ZH89,
  author  = {X. Zhou},
  title   = {{The Riemann--Hilbert problem and inverse scattering}},
  journal = {SIAM Journal on Mathematical Analysis},
  volume  = {20},
  pages   = {966--986},
  year    = {1989}
}

@article{L18,
  author  = {Jonatan Lenells},
  title   = {{Matrix Riemann--Hilbert problems with jumps across Carleson contours}},
  journal = {Monatshefte für Mathematik},
  volume  = {186},
  number  = {1},
  pages   = {111--152},
  year    = {2018}
}

@article{L12,
  author  = {Lenells, Jonatan},
  title   = {{Nonlinear Fourier Transforms and the mKdV Equation in the Quarter Plane}},
  journal = {Studies in Applied Mathematics},
  volume  = {129},
  number  = {4},
  pages   = {347--371},
  year    = {2012},
  doi     = {10.1111/j.1467-9590.2012.00560.x}
}

@article{Alber1999,
  author  = {Alber, M. S. and Camassa, R. and Fedorov, Y. N. and Holm, D. D. and Marsden, J. E.},
  title   = {On billiard solutions of nonlinear {PDEs}},
  journal = {Physics Letters A},
  volume  = {264},
  pages   = {171--178},
  year    = {1999}
}

@article{Alber1995,
  author  = {Alber, M. S. and Camassa, R. and Holm, D. D. and Marsden, J. E.},
  title   = {On the link between umbilic geodesics and soliton solutions of nonlinear {PDEs}},
  journal = {Proceedings of the Royal Society A},
  volume  = {450},
  pages   = {677--692},
  year    = {1995}
}

@article{Borzi2005,
  author  = {Borzi, C. H. and Kraenkel, R. A. and Manna, M. A. and Pereira, A.},
  title   = {Nonlinear dynamics of short traveling capillary-gravity waves},
  journal = {Physical Review E},
  volume  = {71},
  pages   = {026307},
  year    = {2005}
}

@article{BoutetShepelsky2009,
  author  = {Boutet de Monvel, A. and Shepelsky, D.},
  title   = {Long time asymptotics of the {Camassa--Holm} equation on the half-line},
  journal = {Annales de l'Institut Fourier},
  volume  = {59},
  pages   = {3015--3056},
  year    = {2009}
}

@article{FaquirMannaNeveu2007,
  author  = {Faquir, M. and Manna, M. A. and Neveu, A.},
  title   = {An integrable equation governing short waves in a long-wave model},
  journal = {Proceedings of the Royal Society A},
  volume  = {463},
  pages   = {1939--1954},
  year    = {2007}
}

@article{HunterSaxton1991,
  author  = {Hunter, J. K. and Saxton, R.},
  title   = {Dynamics of director fields},
  journal = {SIAM Journal on Applied Mathematics},
  volume  = {51},
  pages   = {1498--1521},
  year    = {1991}
}

@incollection{Kruskal1975,
  author    = {Kruskal, M.},
  title     = {Nonlinear wave equations},
  booktitle = {Dynamical Systems, Theory and Applications},
  series    = {Lecture Notes in Physics},
  volume    = {38},
  publisher = {Springer},
  address   = {Berlin},
  pages     = {310--354},
  year      = {1975}
}

@article{Lenells2008,
  author  = {Lenells, J.},
  title   = {{Poisson structure of a modified Hunter--Saxton equation}},
  journal = {Journal of Physics A: Mathematical and Theoretical},
  volume  = {41},
  pages   = {285207},
  year    = {2008}
}

@article{BoutetDeMonvelShepelskyZielinski2011,
  author  = {Boutet de Monvel, Anne and Shepelsky, Dmitry and Zielinski, Lech},
  title   = {The short-wave model for the {Camassa--Holm} equation: a {Riemann--Hilbert} approach},
  journal = {Inverse Problems},
  volume  = {27},
  number  = {10},
  pages   = {105006},
  year    = {2011},
  doi     = {10.1088/0266-5611/27/10/105006}
}

@article{KarpenkoShepelsky2026,
  author  = {Karpenko, Iryna and Shepelsky, Dmitry},
  title   = {{The Modified Camassa--Holm Equation on the Half Line: A Riemann--Hilbert Approach}},
  journal = {Studies in Applied Mathematics},
  volume  = {156},
  number  = {6},
  pages   = {e70254},
  year    = {2026},
  doi     = {10.1111/sapm.70254}
}

@article{Karpenko2026SineGordon,
  author        = {Karpenko, Iryna},
  title         = {The sine--Gordon equation in light-cone coordinates on the half-lines revisited: a Riemann--Hilbert approach},
  journal = {Letters in Mathematical Physics},
  year          = {2026},
  volume  = {116},
  number  = {116},
  doi           = {10.1007/s11005-026-02147-8}
}

@article{BealsCoifman1984,
  author  = {Beals, Richard and Coifman, Ronald R.},
  title   = {Scattering and inverse scattering for first order systems},
  journal = {Communications on Pure and Applied Mathematics},
  volume  = {37},
  number  = {1},
  pages   = {39--90},
  year    = {1984},
  doi     = {10.1002/cpa.3160370105}
}

\end{document}